\documentclass{amsart}
\usepackage{graphicx}
\usepackage{amscd}
\usepackage{amsmath}
\usepackage{amsfonts}
\usepackage{amssymb}
\newtheorem{theorem}{Theorem}
\theoremstyle{plain}

\newtheorem{condition}{Condition}

\newtheorem{corollary}{Corollary}

\newtheorem{definition}{Definition}
\newtheorem{example}{Example}

\newtheorem{lemma}{Lemma}

\newtheorem{proposition}{Proposition}
\newtheorem{remark}{Remark}

\numberwithin{equation}{section}

\begin{document}
\title[Pointwise Symmetrization Inequalities]{Pointwise Symmetrization Inequalities for Sobolev functions and applications}
\author{Joaquim Mart\'{\i}n$^{\ast}$}
\address{Department of Mathematics\\
Universitat Aut\`{o}noma de Barcelona}
\email{jmartin@mat.uab.cat}
\author{Mario Milman}
\address{Department of Mathematics\\
Florida Atlantic University}
\email{extrapol@bellsouth.net}
\urladdr{http://www.math.fau.edu/milman}
\thanks{$^{\ast}$Supported in part by Grants MTM2007-60500, MTM2008-05561-C02-02 and
by 2009SGR1128.}
\thanks{This paper is in final form and no version of it will be submitted for
publication elsewhere.}

\begin{abstract}
We develop a technique to obtain new symmetrization inequalities that provide
a unified framework to study Sobolev inequalities, concentration inequalities
and sharp integrability of solutions of elliptic equations.
\end{abstract}\maketitle
\tableofcontents

\section{Introduction}

Symmetrization is a very useful classical tool in PDE's and the theory of
Sobolev spaces. The standard symmetrization inequalities, like many other
inequalities in the theory of Sobolev spaces, are often formulated as norm
inequalities. One drawback is that these inequalities need to be (re)proven
separately for different classes of spaces (e.g. $L^{p},$ Lorentz, Orlicz,
Lorentz-Karamata, etc.). For this purpose interpolation can be a useful tool,
but one may lose information in the extreme cases. Moreover, the end point
Sobolev embeddings usually require a different type of spaces (often called
``extrapolation spaces''). Thus, for example, the optimal embeddings of
$L^{p}$ based Sobolev spaces on $n-$dimensional Euclidean space are the
Lorentz $L(p^{\ast},p)$ spaces, where $\frac{1}{p^{\ast}}=\frac{1}{p}-\frac
{1}{n},$ $1\leq p<n,$ but for the limiting case $p=n$ it is necessary to
replace the Lorentz norms by suitable variants in order to accommodate
exponential integrability. One way to deal with this problem is to use
pointwise rearrangement inequalities; among the many contributions in this
direction here we only mention just a few \cite{gar}, \cite{tal}, \cite{tal1},
\cite{kol}, \cite{Ban}, \cite{bec}, \cite{ga}, \cite{alvino}, \cite{Cia},
\cite{carl}, \cite{mmpams}, \cite{mmletter}, \cite{ra}, \cite{Leo}, [32], and
refer the reader to the references therein. An added complication arises
because different geometries produce different types of optimal spaces: a
dramatic example is provided by Gaussian measure, where the optimal target
spaces for the embeddings of $L^{p}$ based Sobolev spaces are the
$L^{p}(LogL)^{p/2}$ spaces (cf. \cite{gro}, \cite{fei}, [1], \cite{bec2},
\cite{bec1}, and the references therein). Likewise, in the study of
integrability of solutions of elliptic equations, the corresponding optimal
results depend on the geometry. As a consequence, although many of the methods
used in the treatment of the different cases are similar each case still
requires a separate treatment.

In our recent work (cf. \cite{mmpjfa}, \cite{mmjfa}, \cite{mmcomptes}) we have
developed new symmetrization inequalities that address all these issues and
can be applied to provide a unified treatment of sharp Sobolev-Poincar\'{e}
inequalities, concentration inequalities and sharp integrability of solutions
of elliptic equations. Our inequalities combine three basic features, each of
which may have been considered before but, apparently, not all of them
simultaneously; namely our inequalities are (i) pointwise rearrangement
inequalities, (ii) incorporate in their formulation the isoperimetric profile
and (iii) are formulated in terms of oscillations.

The first feature (i) allows us to treat without effort the class of all
rearrangement invariant function norms. Let us illustrate this point with the
classical P\'{o}lya-Szeg\"{o} inequality. On $\mathbb{R}^{n}$ this principle
can be informally stated as%
\begin{equation}
\left\|  \left|  \nabla f^{\circ}\right|  \right\|  _{L^{p}(\mathbb{R}^{n}%
)}\leq\left\|  \left|  \nabla f\right|  \right\|  _{L^{p}(\mathbb{R}^{n}%
)},\text{ }1\leq p\leq\infty, \label{int1}%
\end{equation}
where $f^{\circ}$ is the symmetric rearrangement\footnote{$f^{\circ
}(x)=f^{\ast}(\omega_{n}\left|  x\right|  ^{n}),$ is the symmetric decreasing
rearrangmeent of $f,$ $\omega_{n}$ is the measure of the unit ball in
$\mathbb{R}^{n}.$} of $f.$ This inequality leaves open the question of what
would be the corresponding results for other function norms, indeed, different
types of norms are often treated one case at a time in the literature. The
formulation of (\ref{int1}) we use takes the form
\begin{equation}
\left|  \nabla f^{\circ}\right|  ^{\ast\ast}(t)\leq\left|  \nabla f\right|
^{\ast\ast}(t), \label{int2}%
\end{equation}
where $f^{\ast\ast}(t)=\frac{1}{t}\int_{0}^{t}f^{\ast}(s)ds,$ and $f^{\ast}$
is the non increasing rearrangement of $f$ with respect to Lebesgue measure on
$\mathbb{R}^{n}$. The point is that (\ref{int2}) readily implies%
\begin{equation}
\left\|  \left|  \nabla f^{\circ}\right|  \right\|  _{X(\mathbb{R}^{n})}%
\leq\left\|  \left|  \nabla f\right|  \right\|  _{X(\mathbb{R}^{n})},
\label{int3}%
\end{equation}
for all rearrangement invariant spaces $X$ on $\mathbb{R}^{n}$ (see Section
\ref{secc:ri} below)$.$

The fact that our inequalities incorporate the isoperimetric profile [feature
(ii)] allows us to treat different geometries from a unified point of view.
Indeed, it is the isoperimetric profile itself that helps us determine the
correct function spaces! For example, as we show below (cf. Theorem
\ref{teomain}), the isoperimetric inequality can be reformulated on metric
probability spaces $\left(  \Omega,d,\mu\right)  $, (cf. \cite{mmcomptes}, and
also \cite{bmr}, \cite{kol}, \cite{mmpjfa}, \cite{mmjfa}, for Euclidean or
Gaussian versions, see also \cite{coul} for a somewhat different perspective)
as follows\footnote{Although the Euclidean version of (\ref{int4}) is
implicitly proven in \cite{alvino} it is not used in this form in that paper.}%
\begin{equation}
f_{\mu}^{\ast\ast}(t)-f_{\mu}^{\ast}(t)\leq\frac{t}{I(t)}\left|  \nabla
f\right|  _{\mu}^{\ast\ast}(t), \label{int4}%
\end{equation}
where $f_{\mu}^{\ast\ast}(t)=\frac{1}{t}\int_{0}^{t}f_{\mu}^{\ast}(s)ds,$ and
$f_{\mu}^{\ast}$ is the non increasing rearrangement of $f$ with respect to
the measure $\mu\ $and $I(t)=I_{(\Omega,d,\mu)}(t)$ is the corresponding
isoperimetric profile. If we apply a rearrangement invariant function norm
$X\ $on $\Omega$ (see Section \ref{secc:ri} below) to (\ref{int4}) we obtain
Sobolev-Poincar\'{e} type estimates of the form\footnote{The spaces $\bar{X}$
are defined in Section \ref{secc:ri} below.}%
\begin{equation}
\left\|  f\right\|  _{LS(X)}:=\left\|  \left(  f_{\mu}^{\ast\ast}(t)-f_{\mu
}^{\ast}(t)\right)  \frac{I(t)}{t}\right\|  _{\bar{X}}\leq\left\|  \left|
\nabla f\right|  _{\mu}^{\ast\ast}\right\|  _{\bar{X}}. \label{int5'}%
\end{equation}
These embeddings turn out to be best possible in all the classical cases, at
least for spaces that are far from $L^{1}$ (the integrated form of
(\ref{int4}) can be used to cope with this problematic end point as well, see
Proposition \ref{l1l1} below and \cite{mmpjfa} for the Euclidean case). To see
how the isoperimetric profile helps to determine the correct spaces consider
the following basic model cases: (a) $\mathbb{R}^{n}$ with Euclidean measure.
Let $X=L^{p},$ $1\leq p\leq n,$ and let $p^{\ast}$ be the usual Sobolev
exponent defined by $\frac{1}{p^{\ast}}=\frac{1}{p}-\frac{1}{n},$ then from
the fact that $I(t)=c_{n}t^{1-1/n}$ it follows that\footnote{Here the symbol
$f\simeq g$ indicates the existence of a universal constant $c>0$ (independent
of all parameters involved) such that $(1/c)f\leq g\leq c\,f$. Likewise the
symbol $f\preceq g$ will mean that there exists a universal constant $c>0$
(independent of all parameters involved) such that $f\leq c\,g$.}
\begin{equation}
\left\|  \left(  f^{\ast\ast}(t)-f^{\ast}(t)\right)  \frac{I(t)}{t}\right\|
_{L^{p}}\simeq\left(  \int_{0}^{\infty}\left(  \left(  f^{\ast\ast}%
(t)-f^{\ast}(t)\right)  t^{\frac{1}{p^{\ast}}}\right)  ^{p}\frac{dt}%
{t}\right)  ^{1/p}. \label{tarde}%
\end{equation}
Moreover, if $1\leq p<n$, then it follows easily from Hardy's inequality that%
\[
\int_{0}^{\infty}\left(  \left(  f^{\ast\ast}(t)-f^{\ast}(t)\right)
t^{\frac{1}{p^{\ast}}}\right)  ^{p}\frac{dt}{t}\simeq\int_{0}^{\infty}\left(
f^{\ast\ast}(t)t^{\frac{1}{p^{\ast}}}\right)  ^{p}\frac{dt}{t}=\left\|
f\right\|  _{L(p^{\ast},p)}^{p}.
\]

(b) $\mathbb{R}^{n}$ with Gaussian measure $\gamma_{n}$. Let $X=L^{p},$
$1<p<\infty,$ then (compare with \cite{gro}, \cite{fei}), since $I(t)\simeq
t(\log1/t)^{1/2}$ for $t$ near zero, we have%
\begin{align}
\left\|  \left(  f_{\gamma_{n}}^{\ast\ast}(t)-f_{\gamma_{n}}^{\ast}(t)\right)
\frac{I(t)}{t}\right\|  _{L^{p}}  &  \simeq\left(  \int_{0}^{1}\left(  \left(
f_{\gamma_{n}}^{\ast\ast}(t)-f_{\gamma_{n}}^{\ast}(t)\right)  (\log\frac{1}%
{t})^{1/2}\right)  ^{p}\frac{dt}{t}\right)  ^{1/p}\nonumber\\
&  \simeq\left\|  f\right\|  _{L^{p}(Log)^{p/2}}. \label{tarde1}%
\end{align}

We note that feature (iii) allows us to use systematically spaces that are
defined in terms of oscillations (cf. \cite{bds}, \cite{bmr}, \cite{mp}) so
that, in particular, we can treat the borderline cases in a unified fashion.
For example, in the Gaussian case (\ref{tarde1}) we can let $p=\infty,$ and we
obtain the concentration result (cf. \cite{mmjfa})%
\begin{equation}
f\in Lip(\mathbb{R}^{n})\Rightarrow\left\|  \left(  f_{\gamma_{n}}^{\ast\ast
}(t)-f_{\gamma_{n}}^{\ast}(t)\right)  \frac{I(t)}{t}\right\|  _{L^{\infty}%
}<\infty\Rightarrow f\in e^{L^{2}}; \label{compara1}%
\end{equation}
while on $\mathbb{R}^{n}$ with Euclidean measure, $p^{\ast}=\infty$ is allowed
in (\ref{tarde}), and indeed, when $p=n,$ our condition is
optimal\footnote{Thus our conditions slightly improve the exponential
integrability of the borderline cases. More generally, this feature makes our
inequalities and spaces relevant for the theory of concentration of
inequalities (cf. \cite{ledouxbk}, \cite{mmjfa}).} (cf. \cite{bmr}) and reads%
\begin{align}
f  &  \in W_{1}^{n}(\mathbb{R}^{n})\Rightarrow\left\|  f\right\|
_{L(\infty.n)}=\left(  \int_{0}^{\infty}(f^{\ast\ast}(t)-f^{\ast}(t))^{n}%
\frac{dt}{t}\right)  ^{1/n}<\infty\nonumber\\
&  \Rightarrow f\in e^{L^{n^{\prime}}}. \label{compara2}%
\end{align}

It also follows that if the isoperimetric profile does not depend on the
dimension (e.g. this is case in the Gaussian case) then (\ref{int4}) and
(\ref{int5'}) are ``dimension free''.

Returning to the P\'{o}lya-Szeg\"{o} inequality (\ref{int2}) note that, by
construction, the inequality requires the choice of a distinguished
rearrangement. A \textit{posteriori}, one can see that the choice of the
optimal symmetric rearrangement in (\ref{int1}) is ultimately connected with
the solution of the isoperimetric problem on $\mathbb{R}^{n}$. Thus, it is not
surprising that the corresponding inequality in the Gaussian case also
requires a special rearrangement that is connected with the corresponding
solution of the Gaussian isoperimetric problem (cf. \cite{bo}, \cite{sut},
\cite{erh}, \cite{carloss}, and the references therein, and also \cite{mmjfa}
for a more recent treatment).

More generally, to obtain a general version of the P\'{o}lya-Szeg\"{o}
principle valid on metric spaces, we divide the problem at hand in two. First,
we derive a general inequality that does not require us to make a specific
choice of rearrangements but involves the isoperimetric profile, namely (cf.
Theorem \ref{teomain} below)%
\[
\int_{0}^{t}((-f_{\mu}^{\ast})^{\prime}(\cdot)I(\cdot))^{\ast}(s)ds\preceq
\int_{0}^{t}\left|  \nabla f\right|  _{\mu}^{\ast}(s)ds,
\]
where the second rearrangement on the left hand side is with respect to the
Lebesgue measure on $(0,1)$. The second step requires the construction of a
suitable rearrangement. At this point we only know how to construct special
rearrangements for some model cases. For more on this see the discussion in
Section \ref{secc::p-s}, where we consider in detail three important model
examples: (a) measures on $\mathbb{R}^{n}$ which are products of measures of
the form
\[
\mu^{\Phi}=Z_{\Phi}^{-1}\exp\left(  -\Phi(\left|  x\right|  )\right)  dx,
\]
where $\Phi$ is convex and $\sqrt{\Phi}$ is concave and where $Z_{\Phi}^{-1}$
is a normalization constant chosen to ensure that $\mu^{\Phi}(\mathbb{R)=}1$;
(b) the $n-$sphere $\mathbb{S}^{n}$, and (c) the model spaces studied by
Barthe, Ros and others (cf. \cite{Ros} and the references quoted therein). In
each of these model cases we show that a suitable version of the
P\'{o}lya-Szeg\"{o} principle (\ref{int3}) holds.

In Section \ref{secc::po} we derive Poincar\'{e} inequalities and, using the
results of Section \ref{secc::p-s}, we show their sharpness in the model
cases. A typical result in this section gives the equivalence between
Poincar\'{e} inequalities of the form%
\[
\left\|  g-\int_{\Omega}gd\mu\right\|  _{Y}\preceq\left\|  \left|  \nabla
g\right|  \right\|  _{X}%
\]
and the boundedness of certain Hardy type operators associated with the
corresponding isoperimetric profiles (=``isoperimetric Hardy operators'') (cf.
Theorem \ref{opti00} below). These results led us to introduce the metric
probability spaces of ``isoperimetric Hardy type'' (cf. \cite{mmmazya}): these
are exactly the spaces where this characterization of Poincar\'{e}
inequalities holds. This concept turns out to have interesting applications.

Section \ref{secc:emanuel} was inspired by the remarkable recent results by E.
Milman (cf. \cite{MiE}, \cite{mie1}, \cite{mie2} and the references therein).
E. Milman showed that, for Riemannian manifolds satisfying suitable convexity
conditions (cf. Example \ref{ej:emanu} below), we have an equivalence between
isoperimetry, Poincar\'{e} inequalities and concentration. In this section we
show that E. Milman's equivalences hold for metric spaces\footnote{Note that
in this paper we assume that all isoperimetric profiles are concave.} of
isoperimetric Hardy type. We should stress that this result does not provide
us with a proof of E. Milman's results since the precise connection between
isoperimetric Hardy type and convexity conditions is still an open problem.

Isoperimetric Hardy type also plays a fundamental role in Section
\ref{secc::transference}, where we develop a simple transference principle
that allows us to transfer Poincar\'{e} inequalities from one metric space to
another, if we have a suitable majorization of the corresponding isoperimetric
profiles. More precisely, we show that if for two metric probability spaces we
have
\[
I_{(\Omega_{1},d_{1},\mu_{1})}(t)\geq cI_{(\Omega,d,\mu)}(t),\ t\in(0,1/2],
\]
and $(\Omega,d,\mu)$ is of isoperimetric Hardy type then any Poincar\'{e}
inequality of the form%
\[
\left\|  g-\int_{{\Omega}}gd\mu\right\|  _{Y(\Omega)}\leq c\left\|  \left|
\nabla g\right|  \right\|  _{X(\Omega)},\text{ for all }g\in Lip(\Omega),
\]
\ can be transferred to a corresponding Poincar\'{e} inequality for
$\Omega_{1}$ (cf. Theorem \ref{teosubordinado}),%

\[
\left\|  g-\int_{{\Omega}_{1}}gd\mu_{1}\right\|  _{Y(\Omega_{1})}\leq
c\left\|  \left|  \nabla g\right|  \right\|  _{X(\Omega_{1})},\ \text{for all
}g\in Lip(\Omega_{1}).
\]
This easy to formulate principle thus allows for the transference of
Poincar\'{e} inequalities from all the model cases discussed above. For
example, the Levi-Gromov isoperimetric inequality implies that Poincar\'{e}
inequalities for the $n-$sphere can be transferred to compact connected
manifolds with Ricci curvature bounded from below by $\rho>0$ (cf.\ Corollary
\ref{corosuboesferico}), extending earlier work in \cite{islas} for the
$L^{p}$ case). Likewise, Poincar\'{e} inequalities valid for $\mathbb{R}^{n}$
with Gaussian measure (cf. \cite{mmjfa}) can be transferred to Riemannian
manifolds $(M,g)$ with isoperimetric profile $I$ for which we have (cf.
Corollary \ref{corosubogauss})
\[
I(t)\geq ct\left(  \log\frac{1}{t}\right)  ^{1/2},\ t\in(0,1/2].
\]
In the same vein we can transfer Poincar\'{e} inequalities valid for
$(\mathbb{R}^{n},\mu_{p}^{\otimes n})$ with $\mu_{p}=Z_{p}^{-1}\exp\left(
-\left|  x\right|  ^{p}\right)  dx,1<p\leq2,$ this leads to simplifications to
recent results of \cite{BaKo} (cf. Corollary \ref{coco}).

When the first version of our manuscript was being typed we received a query
from Professor Hans Triebel concerning certain Sobolev inequalities with
dimension free constants. We give a brief answer to some of Prof. Triebel's
questions in Section \ref{secc:trie}.

In a different direction, in Section \ref{secc:semi} we extend E. Milman's
methods (based on the use of semigroup technique of Ledoux and Bakry and
Ledoux (cf. \cite{led1}, \cite{led2}, \cite{led3}, \cite{bakrled}, and the
references therein) to estimate isoperimetric profiles associated with
functional inequalities involving r.i. spaces.

In Section \ref{secc:ga}, motivated by the results and methods of Gallot
\cite{ga} (cf. also \cite{Ta} and \cite{Ban}), we extend our results and prove
inequalities for the Laplacian. For example, the corresponding extension of
(\ref{int4}) is given by
\begin{equation}
f_{\mu}^{\ast\ast}(t)-f_{\mu}^{\ast}(t)\leq\frac{1}{t}\int_{0}^{t}\left(
\frac{s}{I(s)}\right)  ^{2}\left|  \Delta f\right|  _{\mu}^{\ast\ast}(s)ds.
\label{int6}%
\end{equation}
When $I(t)$ is concave, a global standing assumption in this paper, then
(\ref{int6}) implies the more suggestive inequality (compare with
(\ref{int4}))%
\begin{equation}
f_{\mu}^{\ast\ast}(t)-f_{\mu}^{\ast}(t)\leq\left(  \frac{t}{I(t)}\right)
^{2}\frac{1}{t}\int_{0}^{t}\left|  \Delta f\right|  _{\mu}^{\ast\ast}(s)ds.
\label{cerca}%
\end{equation}
As a consequence we obtain higher order Sobolev-Poincar\'{e} inequalities of
the form%
\begin{equation}
\left\|  \left(  f_{\mu}^{\ast\ast}(t)-f_{\mu}^{\ast}(t)\right)  \left(
\frac{I(t)}{t}\right)  ^{2}\right\|  _{\bar{X}}\preceq\left\|  \left|  \Delta
f\right|  \right\|  _{X}. \label{cerca1}%
\end{equation}
Although we only consider second order inequalities in this paper, estimates
like (\ref{cerca}) and (\ref{cerca1}) are easy to iterate to inequalities
involving higher order derivatives (cf. \cite[Theorem 3.2]{mp}) leading to new
sharp higher order embeddings for Sobolev spaces based on r.i. spaces. Once
again the results are sharp and include sharpenings of the borderline cases.
Our results in this direction extend and unify earlier Euclidean results (cf.
\cite{cwp}, \cite{edmund}, \cite{Cia}, \cite{mp}, \cite{mmjmaa} and the
references therein), as well as $L^{p}$ and Orlicz-Gaussian results (cf.
\cite{fei}, \cite{bakme}, \cite{bakme2}, \cite{shi}).

Using variants of techniques developed by Maz'ya \cite{mazyatal}, and Talenti
and his school (cf. \cite{Ta}, \cite{tal}, \cite{tal1}, \cite{Ta1}, \cite{AFT}
and the references therein), the higher order results of Section \ref{secc:ga}
can be considerably extended in order to study the sharp integrability of
solutions of non-linear elliptic equations of the form%

\begin{equation}
\left\{
\begin{array}
[c]{ll}%
-div(a(x,u,\nabla u))=fw & \text{ in }G,\\
u=0 & \text{ on }\partial G,
\end{array}
\right.  \label{int21}%
\end{equation}
where $G$ is an open domain of $\mathbb{R}^{n}$ ($n\geq2),$ $w\ $is a
nonnegative measurable function on $\mathbb{R}^{n}$, such that the measure
$\mu=w(x)dx,$ is a probability measure, $a(x,\eta,\xi):G\times\mathbb{R}%
\times\mathbb{R}^{n}\rightarrow\mathbb{R}^{n}$ is a Carath\'{e}odory function
such that,%
\[
a(x,t,\xi).\xi\geq w(x)\left|  \xi\right|  ^{p},\ \ for\text{ }a.e.\ \ x\in
G,\ \forall\eta\in\mathbb{R},\ \forall\xi\in\mathbb{R}^{n}.
\]
This material is developed in Section \ref{secc:el} where we consider
\textit{a priori} estimates of entropy solutions of (\ref{int21}). For
example, for $p=2,$ we show that an entropic solution of (\ref{int21})
satisfies%
\[
\left\|  \left(  u_{\mu}^{\ast\ast}(t)-u_{\mu}^{\ast}(t)\right)  \left(
\frac{I(t)}{t}\right)  ^{2}\right\|  _{\bar{X}}\preceq\left\|  f_{\mu}%
^{\ast\ast}\right\|  _{\bar{X}},
\]
from where we can obtain sharp \textit{a priori} integrability results for
entropy solutions. Moreover, we also obtain estimates on the regularity of the
gradient. For example, extending results in \cite{AFT} we have (cf. Theorem
\ref{remderv} below)
\[
\left|  \nabla u\right|  _{\mu}^{\ast}(t)\leq\left(  \frac{2}{t}\int
_{t/2}^{\mu(G)}\left(  \frac{I(s)}{s}f_{\mu}^{\ast\ast}(s)\right)
^{2}ds\right)  ^{1/2}.
\]
These estimates can be used to obtain norm estimates under suitable
assumptions on $\bar{X}$ (cf. Theorem \ref{remderv} below):%
\[
\left\|  \frac{I(t)}{t}\left|  \nabla u\right|  _{\mu}^{\ast}(t)\right\|
_{\bar{X}}\preceq\left\|  f_{\mu}^{\ast\ast}\right\|  _{\bar{X}}.
\]
Again we point out that the isoperimetric profile determines the nature of the
correct integrability conditions.

In Section \ref{secc:ma} we discuss the connection between Maz'ya's capacitary
inequalities and the method of symmetrization by truncation. We conclude in
Section \ref{secc:appendix} by recording a few (and only a few)
bibliographical notes.

Finally a few words about the techniques. A common method to obtain
rearrangement inequalities is via interpolation or extrapolation (cf.
\cite{ca}, \cite{jm}) however these methods do not necessarily produce the
best possible end point results. Maz'ya \cite{maz'yabook} has shown that
Sobolev inequalities self improve using his technique of smooth cut-offs. In a
different direction, Maz'ya, and independently Federer and Fleming, (cf.
\cite{maz'yabook}, \cite{fed}), also showed the equivalence between
isoperimetry and Sobolev embeddings. It is easy to see that these ideas are
closely related. Indeed, consider the following three versions of the
classical Gagliardo-Nirenberg inequality in increasing order of precision, for
$f\in C_{0}^{\infty}(\mathbb{R}^{n}),$%
\begin{equation}
\left\|  f\right\|  _{L(n^{\prime},\infty)}\preceq\left\|  \left|  \nabla
f\right|  \right\|  _{L^{1}},\text{ weak type Gagliardo-Nirenberg} \label{gn1}%
\end{equation}%
\begin{equation}
\left\|  f\right\|  _{L^{n^{\prime}}}\preceq\left\|  \left|  \nabla f\right|
\right\|  _{L^{1}},\text{ classical Gagliardo-Nirenberg} \label{gn2}%
\end{equation}%
\begin{equation}
\left\|  f\right\|  _{L(n^{\prime},1)}\preceq\left\|  \left|  \nabla f\right|
\right\|  _{L^{1}},\text{ sharp Gagliardo-Nirenberg,} \label{gn3}%
\end{equation}
and note that for an approximating sequence $\{f_{n}\}_{n}\mapsto\chi_{A}$ the
left hand sides of (\ref{gn1}), (\ref{gn2}), (\ref{gn3}) all tend to $\left|
A\right|  ^{1/n^{\prime}}$, while the right hand sides are always a multiple
of $\mu^{+}(A),$ the perimeter of $A.$ Thus, disregarding constants, the
Maz'ya-Federer-Fleming equivalence theorem shows that (\ref{gn1})
automatically self improves to (\ref{gn3}).

Although in this paper we don't formally use interpolation/extrapolation
theory we borrow one basic idea from this field that originates in the work of
Calder\'{o}n \cite{ca} (cf. also \cite{BS}), in PDE's this idea also appears
in the work of Talenti (\cite{tal} and \cite{tal1}, see also Section
\ref{secc:sharp} below), and was somewhat later taken up in the extrapolation
theory of Jawerth and Milman \cite{jm}; namely that families of inequalities
can be characterized in terms of pointwise rearrangement inequalities. Indeed,
in Calder\'{o}n's program \cite{ca} families of inequalities for a given
operator are characterized in terms of pointwise rearrangement inequalities
from which each individual functional norm inequalities follows readily. The
point is that one norm inequality is not enough to effect this characterization.

Take the inequalities (\ref{gn1}), (\ref{gn2}), (\ref{gn3}), which as we have
argued above, are, in some sense, equivalent, in this case the ``correct'' way
to express this phenomenon is via the rearrangement inequality (\ref{int4}).
The technique to prove this equivalence uses systematically Maz'ya's smooth
truncations method as a tool to obtain rearrangement inequalities
(``symmetrization by truncation''). We notice parenthetically that truncations
are also a basic tool in interpolation/extrapolation theory (for more on this
see Section \ref{secc:trunc}).\newline \textbf{Acknowledgement.} \textit{We
are very grateful to the referee for providing us with valuable references and
many helpful suggestions to shorten and improve the quality of the paper}.

\section{Background}

We use for the most part a standard notation. For the discussion on metric
spaces it will simplify the discussion somewhat to consider only probability
spaces, a convention we keep for the rest of the paper.

We always consider connected metric spaces $\left(  \Omega,d,\mu\right)  $
equipped with a separable, non-atomic, Borel probability measure $\mu$. For
measurable functions $u:\Omega\rightarrow\mathbb{R},$ the distribution
function of $u$ is given by
\[
\mu_{u}(t)=\mu\{x\in{\Omega}:\left|  u(x)\right|  >t\}\text{ \ \ \ \ }(t>0).
\]
The \textbf{decreasing rearrangement} $u_{\mu}^{\ast}$ of $u$ is the
right-continuous non-increasing function from $[0,\infty)$ into $[0,\infty]$
which is equimeasurable with $u$. Namely,
\[
u_{\mu}^{\ast}(s)=\inf\{t\geq0:\mu_{u}(t)\leq s\}.
\]

It is easy to see that for any measurable set $E\subset\Omega$
\[
\int_{E}\left|  u(x)\right|  d\mu\leq\int_{0}^{\mu(E)}u_{\mu}^{\ast}(s)ds.
\]
In fact, the following stronger property holds (cf. \cite{BS}),%
\begin{equation}
\sup_{\mu(E)\leq t}\int_{E}\left|  u(x)\right|  d\mu=\int_{0}^{\mu(E)}u_{\mu
}^{\ast}(s)ds. \label{hp}%
\end{equation}
Since $u_{\mu}^{\ast}$ is decreasing, the function $u_{\mu}^{\ast\ast},$
defined by
\[
u_{\mu}^{\ast\ast}(t)=\frac{1}{t}\int_{0}^{t}u_{\mu}^{\ast}(s)ds,
\]
is also decreasing and, moreover,
\[
u_{\mu}^{\ast}\leq u_{\mu}^{\ast\ast}.
\]

On occasion, when rearrangements are taken with respect to the Lebesgue
measure or when the measure is clear from the context, we may omit the measure
and simply write $u^{\ast}$ and $u^{\ast\ast}$, etc.

For a Borel set $A\subset\Omega,$ the \textbf{perimeter} or \textbf{Minkowski
content} of $A$ is defined by
\[
\mu^{+}(A)=\lim\inf_{h\rightarrow0}\frac{\mu\left(  A_{h}\right)  -\mu\left(
A\right)  }{h},
\]
where $A_{h}=\left\{  x\in\Omega:d(x,A)<h\right\}  .$

The \textbf{isoperimetric profile} $I_{(\Omega,d,\mu)}$ is defined as the
pointwise maximal function $I_{(\Omega,d,\mu)}:[0,1]\rightarrow\left[
0,\infty\right)  $ such that%
\[
\mu^{+}(A)\geq I_{(\Omega,d,\mu)}\left(  \mu(A)\right)  ,
\]
holds for all Borel sets $A$. A set $A$ for which equality above is attained
will be called an \textbf{isoperimetric domain.}

\begin{example}
Let $(\Omega,d,\mu)$ be the metric measure space obtained from a $C^{\infty}$
complete oriented $n-$dimensional Riemannian manifold $(M,g)$, where $d$ is
the induced geodesic distance and $\mu$ is absolutely continuous with respect
to dvol$_{M}.$

(i) (cf. \cite[Proposition 1.5.1]{bayle}) $I_{(\Omega,d,\mu)}(t)$ is
continuous, and $I_{(\Omega,d,\mu)}(t)>0$ for $t\in(0,1).$

(ii) (cf. \cite[Proposition 1.2.2]{bayle})%
\[
I_{(\Omega,d,\mu)}(t)=I_{(\Omega,d,\mu)}(1-t),\forall t\in\lbrack0,1].
\]
\end{example}

\begin{example}
\label{ej:emanu}Suppose that $(\Omega,d,\mu)$ is as in the previous example.
We say that $(\Omega,d,\mu)$ satisfies E. Milman's convexity conditions if
$d\mu=e^{-\Psi}dvol_{M},$ where $\Psi$ is such that $\Psi\in C^{2}(M),$ and as
tensor fields Ric$_{g}+Hess_{g}(\Psi)\geq0$ on $M.$ Then it is known that
$I_{(\Omega,d,\mu)}$ is also concave (cf. \cite{mie1} and the extensive list
of references therein).
\end{example}

In view of the previous examples, and in order to balance generality with
power and simplicity, we will assume throughout the paper that our spaces
satisfy the following

\begin{condition}
\label{pedida}The metric probability spaces $(\Omega,d,\mu)$ considered in
this paper are assumed to have isoperimetric profiles $I_{(\Omega,d,\mu)}$
which are concave, continuous, increasing on $(0,1/2),$ symmetric about the
point $1/2$ and such that $I_{(\Omega,d,\mu)}(0)=0$.
\end{condition}

A continuous, concave function, $I:[0,1]\rightarrow\left[  0,\infty\right)  $,
increasing on $(0,1/2)$ and symmetric about the point $1/2,$ with $I(0)=0,$
and\ such that
\[
I_{(\Omega,d,\mu)}\geq I,
\]
will be called an \textbf{isoperimetric estimator} for $(\Omega,d,\mu).$

For a function $f$ on $\Omega$ which is Lipschitz in every ball (briefly $f\in
Lip(\Omega))$ we define, as usual, the \textbf{modulus of the gradient} by
\begin{equation}
|\nabla f(x)|=\limsup_{d(x,y)\rightarrow0}\frac{|f(x)-f(y)|}{d(x,y)}%
,\label{vista}%
\end{equation}
and zero at isolated points\footnote{For example on $\Omega=R^{n},$ the class
of Lipschitz functions on every ball coincides with the class of locally
Lipschitz functions (cf. \cite[pp. 184, 189]{BH} for more details).}.

\begin{condition}
For each $c\in R,|\nabla f(x)|=0,a.e.$ in the set $\{x:f(x)=c\}.$ This
condition is verified in all the classical cases (cf. [59] and also [62]).
\end{condition}

\subsection{Rearrangement invariant spaces\label{secc:ri}}

We recall briefly the basic definitions and conventions we use from the theory
of rearrangement-invariant (r.i.) spaces and refer the reader to \cite{BS},
\cite{KPS}, as well as \cite{pu}, \cite{pu1} and \cite{pu11}, for a complete
treatment. We say that a Banach function space $X=X({\Omega})$ on $({\Omega
},d,\mu)$ is rearrangement-invariant (r.i.) space, if $g\in X$ implies that
all $\mu-$measurable functions $f$ with the same rearrangement function with
respect to the measure $\mu$, i.e. such that $f_{\mu}^{\ast}=g_{\mu}^{\ast},$
also belong to $X,$ and, moreover, $\Vert f\Vert_{X}=\Vert g\Vert_{X}$.

Since $\mu(\Omega)=1,$ for any r.i. space $X({\Omega})$ we have%

\begin{equation}
L^{\infty}(\Omega)\subset X(\Omega)\subset L^{1}(\Omega), \label{nuevadeli}%
\end{equation}
with continuous embeddings.

An r.i. space $X({\Omega})$ can be represented by a r.i. space on the interval
$(0,1),$ with Lebesgue measure, $\bar{X}=\bar{X}(0,1),$ such that%
\[
\Vert f\Vert_{X}=\Vert f_{\mu}^{\ast}\Vert_{\bar{X}},
\]
for every $f\in X.$ A characterization of the norm $\Vert\cdot\Vert_{\bar{X}}$
is available (see \cite[Theorem 4.10 and subsequent remarks]{BS}). Typical
examples of r.i. spaces are the $L^{p}$-spaces, Lorentz spaces and Orlicz spaces.

A useful property of r.i. spaces states that if
\[
\int_{0}^{r}f_{\mu}^{\ast}(s)ds\leq\int_{0}^{r}g_{\mu}^{\ast}(s)ds,\text{
\ holds for all\ }r>0,
\]
then, for any r.i. space $X=X({\Omega}),$%
\[
\left\|  f\right\|  _{X}\leq\left\|  g\right\|  _{X}.
\]
The \textbf{associate space }of $X({\Omega})$\footnote{The associate space of
the associate space $X^{\prime}(\Omega)$ satisfies
\[
\left(  X^{\prime}(\Omega)\right)  ^{\prime}=X^{\prime\prime}(\Omega
)=X(\Omega).
\]
} is the r.i. space $X^{\prime}(\Omega)$ of all functions for which%
\begin{equation}
\left\|  h\right\|  _{X^{\prime}(\Omega)}=\sup_{g\neq0}\frac{\int_{\Omega
}\left|  g(x)h(x)\right|  d\mu}{\left\|  g\right\|  _{X(\Omega)}}<\infty.
\label{holhol}%
\end{equation}
Therefore the following generalized H\"{o}lder inequality holds%

\[
\int_{\Omega}\left|  g(x)h(x)\right|  d\mu\leq\left\|  g\right\|  _{X(\Omega
)}\left\|  h\right\|  _{X^{\prime}(\Omega)}.
\]

The \textbf{fundamental function\ }of $X$ is defined by
\[
\phi_{X}(s)=\left\|  \chi_{E}\right\|  _{X},
\]
where $E$ is any measurable subset of $\Omega$ with $\mu(E)=s.$ We can assume
without loss of generality that $\phi_{X}$ is concave. Moreover,%
\begin{equation}
\phi_{X^{\prime}}(s)\phi_{X}(s)=s. \label{dual}%
\end{equation}
\ For example, let $N$ be a Young's function, then the fundamental function of
the corresponding Orlicz space $L_{N}$ is given by
\begin{equation}
\phi_{L_{N}}(t)=1/N^{-1}(1/t). \label{quetiempos aquellos}%
\end{equation}

$\ $Associated with an r.i. space $X$ there are some useful Lorentz and
Marcinkiewicz spaces, namely the Lorentz and Marcinkiewicz spaces defined by
the quasi-norms
\[
\left\|  f\right\|  _{M(X)}=\sup_{t>0}f^{\ast}(t)\phi_{X}(t),\text{
\ \ }\left\|  f\right\|  _{\Lambda(X)}=\int_{0}^{1}f^{\ast}(t)d\phi_{X}(t).
\]
Notice that
\[
\phi_{M(X)}(t)=\phi_{\Lambda(X)}(t)=\phi_{X}(t),
\]
and that
\begin{equation}
\Lambda(X)\subset X\subset M(X). \label{tango}%
\end{equation}

Let $p>0$ and let $X$ be a r.i. space on $\Omega;$ the $p-$%
\textbf{convexification} $X^{(p)}$ of $X,$ (cf. \cite{lt}) is
defined\textbf{\ }by
\[
X^{(p)}=\{x:\left|  x\right|  ^{p}\in X\},\text{ \ \ }\left\|  x\right\|
_{X^{(p)}}=\left\|  \left|  x\right|  ^{p}\right\|  _{X}^{1/p}.
\]
We will say that $X$ is $p-$\textbf{convex} if and only if $X^{(1/p)}$ is a
Banach space.

Classically conditions on r.i. spaces are formulated in terms of the Hardy
operators defined by
\[
Pf(t)=\frac{1}{t}\int_{0}^{t}f(s)ds;\text{ \ \ \ }Q_{a}f(t)=\frac{1}{t^{a}%
}\int_{t}^{\infty}s^{a}f(s)\frac{ds}{s},\text{ \ \ }0\leq a<1,
\]
(if $a=0,$ we shall simply write $Q$ instead of $Q_{0}),$ the boundedness of
these operators on r.i. spaces can be simply described in terms of the so
called Boyd indices defined by
\[
\bar{\alpha}_{X}=\inf\limits_{s>1}\dfrac{\ln h_{X}(s)}{\ln s}\text{ \ \ and
\ \ }\underline{\alpha}_{X}=\sup\limits_{s<1}\dfrac{\ln h_{X}(s)}{\ln s},
\]
where $h_{X}(s)$ denotes the norm of the dilation operator $E_{s},$ $s>0,$ on
$\bar{X}$, defined by
\[
E_{s}f(t)=\left\{
\begin{array}
[c]{ll}%
f^{\ast}(\frac{t}{s}) & 0<t<s,\\
0 & s<t<1.
\end{array}
\right.
\]
The operator $E_{s}$ is bounded on $\bar{X}$ for every r.i. space $X(\Omega)$
and for every $s>0$. Moreover,
\begin{equation}
h_{X}(s)\leq\max\{1,s\}. \label{ccdd}%
\end{equation}
For example, if $X=L^{p}$, then $\overline{\alpha}_{L^{p}}=\underline{\alpha
}_{L^{p}}=\frac{1}{p}.$ It is well known that if $X$ is a r.i. space,
\begin{equation}%
\begin{array}
[c]{c}%
P\text{ is bounded on }\bar{X}\text{ }\Leftrightarrow\overline{\alpha}%
_{X}<1,\\
Q_{a}\text{ is bounded on }\bar{X}\text{ }\Leftrightarrow\underline{\alpha
}_{X}>a.
\end{array}
\label{alcance}%
\end{equation}
Finally, the following result will be useful in Section \ref{secc:el}

\begin{lemma}
\label{cota01}Let $Y$ be a r.i., space, let $q>0$ and let $w(s)$ be a monotone
function. Then
\[
\left\|  \left(  \frac{1}{t}\int_{t}^{1}\left(  w(s)f^{\ast}(s)\right)
^{q}ds\right)  ^{1/q}\right\|  _{Y}\leq c\left\|  wf\right\|  _{Y}\text{ \ if
\ }\underline{\alpha}_{Y}>1/q.
\]
\end{lemma}

\begin{proof}
Is an elementary adaptation of the main result in \cite{MS}.
\end{proof}

\begin{remark}
In Section \ref{secc:hardy} and Section \ref{secc:el} we introduce new Hardy
operators that are associated with isoperimetric profiles and will play a role
in our theory.
\end{remark}

In \cite{mmjfa} and \cite{mmcomptes} we introduced the ``isoperimetric''
spaces $LS(X)$ defined by the condition%
\[
\left\|  f\right\|  _{LS(X)}:=\left\|  \left(  f_{\mu}^{\ast\ast}(t)-f_{\mu
}^{\ast}(t)\right)  \frac{I(t)}{t}\right\|  _{\bar{X}}<\infty.
\]
The inequality (\ref{int5'}) can be thus reformulated as
\begin{equation}
\left\|  f\right\|  _{LS(X)}\leq\left\|  P(\left|  \nabla f\right|  _{\mu
}^{\ast})\right\|  _{\bar{X}}. \label{corres}%
\end{equation}
The $LS(X)$ spaces not only give sharp embedding theorems that include
borderline cases but, due to the fact that their definition incorporates the
isoperimetric profile, they automatically ``select'' the optimal spaces
associated with a given geometry\footnote{In particular see the discussion
right after (\ref{int5'}) above. In the classical borderline cases these
isoperimetric spaces capture exponential integrability conditions and thus
seem to have a natural role in concentration inequalities (cf. Remark
\ref{remarkao1}, and \cite{ledouxbk}, \cite{mmjfa}).}.

The concept of median plays a role in the study of Poincar\'{e} inequalities
(cf. Section \ref{secc::po})

\begin{definition}
Let $f\ $be a measurable function, a real number $m_{e}$ will be called a
\textbf{median} of $f$ if
\[
\mu\left\{  f\geq m_{e}\right\}  \geq1/2\text{ \ and }\mu\left\{  f\leq
m_{e}\right\}  \geq1/2.
\]
\end{definition}

For most purposes to prove Poincar\'{e} inequalities (see (\ref{poin00})
below) it makes no difference if we work with a median $m_{e}$ or use the
``expectation'' $\int_{\Omega}fd\mu$. We record this fact in the next
lemma\footnote{Although the result is known we include a proof for the sake of completeness.}

\begin{lemma}
\label{media}Let $X$ be a r.i. space on $\Omega$. Then,%
\[
\frac{1}{2}\left\|  f-\int_{\Omega}fd\mu\right\|  _{X}\leq\left\|
f-m_{e}\right\|  _{X}\leq3\left\|  f-\int_{\Omega}fd\mu\right\|  _{X}.
\]
\end{lemma}

\begin{proof}
By (\ref{nuevadeli}) we have%
\[
\left|  \int_{\Omega}fd\mu-m_{e}\right|  \leq\int_{\Omega}\left|
f-m_{e}\right|  d\mu\leq\left\|  f-m_{e}\right\|  _{X},
\]
thus,%
\begin{align*}
\left\|  f-\int_{\Omega}fd\mu\right\|  _{X}  &  =\left\|  f-m_{e}+\int
_{\Omega}fd\mu+m_{e}\right\|  _{X}\\
&  \leq\left\|  f-m_{e}\right\|  _{X}+\left|  \int_{\Omega}fd\mu-m_{e}\right|
\\
&  \leq2\left\|  f-m_{e}\right\|  _{X}.
\end{align*}
To prove the converse we can assume that $m_{e}\geq\int_{\Omega}fd\mu$
(otherwise exchange $f$ by $-f$). Therefore, by Chebyshev's inequality, we
have%
\begin{align*}
1/2  &  \leq\mu\left\{  f\geq m_{e}\right\} \\
&  \leq\mu\left\{  \left|  f-\int_{\Omega}fd\mu\right|  \geq m_{e}%
-\int_{\Omega}fd\mu\right\} \\
&  \leq\frac{1}{\left(  m_{e}-\int_{\Omega}fd\mu\right)  }\int_{\Omega}\left|
f-\int_{\Omega}fd\mu\right|  .
\end{align*}
Consequently,%
\[
\left(  m_{e}-\int_{\Omega}fd\mu\right)  \leq2\left\|  f-\int_{\Omega}%
fd\mu\right\|  _{X},
\]
which implies%
\[
\left\|  m_{e}-\int_{\Omega}fd\mu\right\|  _{X}\leq2\left\|  f-\int_{\Omega
}fd\mu\right\|  _{X}.
\]
Therefore,%
\begin{align*}
\left\|  f-m_{e}\right\|  _{X}  &  =\left\|  f-\int_{\Omega}fd\mu-m_{e}%
+\int_{\Omega}fd\mu\right\|  _{X}\\
&  \leq3\left\|  f-\int_{\Omega}fd\mu\right\|  _{X}.
\end{align*}
\end{proof}

\section{Symmetrization using truncation and Isoperimetry \label{secc:trunc}}

The characterization of norm inequalities in terms of pointwise rearrangement
inequalities is a theme that seems to have originated in Interpolation theory.
In PDE's this idea appears prominently in the work of Talenti (cf. \cite{tal}
and \cite{tal1}) where it appears as a comparison principle. In interpolation
theory this method was developed in Calder\'{o}n's masterful paper \cite{ca}
(cf. also \cite{BS}), this idea is also important in the extrapolation theory
developed in \cite{jm}. Interestingly, while in our work we try to
characterize Sobolev norm inequalities in terms of rearrangement inequalities,
we generally don't use interpolation/extrapolation. In fact, the smooth
cut-off method, an idea apparently originating in the work of Maz'ya
\cite{maz'yabook} (cf. also \cite{bakr}, \cite{haj}, \cite{tar}, and the
references therein), shows that Sobolev inequalities have remarkable self
improving properties\footnote{In some sense this implies that a Sobolev
inequality carries the information of a family of Sobolev inequalities. If
this is combined with the chain rule one can see that one Sobolev inequality
also carries the ``reiteration'' property. Therefore, from our point of view,
Sobolev inequalities need not be interpolated but can be ``extrapolated''.}.
Combining these ideas with a basic technique of interpolation/extrapolation
(i.e. cutting off at levels dependent on the rearrangement of the function to
which we apply the cut-off itself!) we developed the technique of
``symmetrization by truncation''. The main result in this section is a natural
extension of similar, somewhat less general results, we obtained elsewhere
(cf. \cite{mmpjfa}, \cite{mmjfa}, see also \cite{BH} for the equivalence
between (\ref{isop}) and (\ref{ledo})).

\begin{theorem}
\label{teomain} Let $I$ $:[0,1]\rightarrow\left[  0,\infty\right)  $ be an
isoperimetric estimator on $(\Omega,d,\mu).$ The following statements hold and
are in fact equivalent:

\begin{enumerate}
\item  Isoperimetric inequality: for all Borel sets $A\subset\Omega,$%
\begin{equation}
\mu^{+}(A)\geq I(\mu(A)). \label{isop}%
\end{equation}

\item  Ledoux's inequality: for all functions $f$ $\in$ $Lip(\Omega),$%
\begin{equation}
\int_{0}^{\infty}I(\mu_{f}(s))ds\leq\int_{{\Omega}}\left|  \nabla f(x)\right|
d\mu. \label{ledo}%
\end{equation}

\item  Maz'ya's inequality\footnote{See Mazya \cite{mazyatal} and also Talenti
\cite{Ta}.}: for all functions $f$ $\in$ $Lip(\Omega),$%
\begin{equation}
(-f_{\mu}^{\ast})^{\prime}(s)I(s)\leq\frac{d}{ds}\int_{\{\left|  f\right|
>f_{\mu}^{\ast}(s)\}}\left|  \nabla f(x)\right|  d\mu. \label{dosa}%
\end{equation}

\item  P\'{o}lya-Szeg\"{o}'s inequality: for all functions $f$ $\in$
$Lip(\Omega),$%
\begin{equation}
\int_{0}^{t}((-f_{\mu}^{\ast})^{\prime}(.)I(.))^{\ast}(s)ds\leq\int_{0}%
^{t}\left|  \nabla f\right|  _{\mu}^{\ast}(s)ds. \label{provadas}%
\end{equation}
(The second rearrangement on the left hand side is with respect to the
Lebesgue measure).

\item  Oscillation inequality: for all functions $f$ $\in$ $Lip(\Omega),$%
\begin{equation}
(f_{\mu}^{\ast\ast}(t)-f_{\mu}^{\ast}(t))\leq\frac{t}{I(t)}\left|  \nabla
f\right|  _{\mu}^{\ast\ast}(t). \label{rea}%
\end{equation}
\end{enumerate}
\end{theorem}

\begin{proof}
$\mathbf{(1)\Rightarrow(2).}$ Note that $f\in Lip(\Omega)$ implies that
$\left|  f\right|  \in Lip(\Omega),$ and, moreover, we have (cf.
(\ref{vista}))%
\[
\left|  \nabla f(x)\right|  \geq\left|  \nabla\left|  f\right|  (x)\right|  .
\]
By the co-area inequality applied to $\left|  f\right|  $ (cf. \cite[Lemma
3.1]{BH}), and the isoperimetric inequality (\ref{isop}), it follows that
\begin{align*}
\int_{\Omega}\left|  \nabla f(x)\right|  d\mu &  \geq\int_{\Omega}\left|
\nabla\left|  f\right|  (x)\right|  d\mu\geq\int_{0}^{\infty}\mu^{+}(\{\left|
f\right|  >s\})ds\\
&  \geq\int_{0}^{\infty}I(\mu_{f}(s))ds\text{ }.
\end{align*}
$\mathbf{(2)\Rightarrow(3).}$ Let $0<t_{1}<t_{2}<\infty.$ The smooth
truncations of $f$ are defined by
\[
f_{t_{1}}^{t_{2}}(x)=\left\{
\begin{array}
[c]{ll}%
t_{2}-t_{1} & \text{if }\left|  f(x)\right|  \geq t_{2},\\
\left|  f(x)\right|  -t_{1} & \text{if }t_{1}<\left|  f(x)\right|  <t_{2},\\
0 & \text{if }\left|  f(x)\right|  \leq t_{1}.
\end{array}
\right.
\]
Applying (\ref{ledo}) to $f_{t_{1}}^{t_{2}}$ we obtain,
\[
\int_{0}^{\infty}I(\mu_{f_{t_{1}}^{t_{2}}}(s))ds\leq\int_{\Omega}\left|
\nabla f_{t_{1}}^{t_{2}}(x)\right|  d\mu.
\]
We have (cf. \cite{haj})
\[
\left|  \nabla f_{t_{1}}^{t_{2}}\right|  =\left|  \nabla\left|  f_{t_{1}%
}^{t_{2}}\right|  \right|  =\left|  \nabla\left|  f\right|  \right|
\chi_{\left\{  t_{1}<\left|  f\right|  <t_{2}\right\}  },
\]
and, moreover,
\begin{equation}
\int_{0}^{\infty}I(\mu_{f_{t_{1}}^{t_{2}}}(s))ds=\int_{0}^{t_{2}-t_{1}}%
I(\mu_{f_{t_{1}}^{t_{2}}}(s))ds. \label{A1}%
\end{equation}
Observe that, for $0<s<t_{2}-t_{1}$,%
\[
\mu\left\{  \left|  f\right|  \geq t_{2}\right\}  \leq\mu_{f_{t_{1}}^{t_{2}}%
}(s)\leq\mu\left\{  \left|  f\right|  >t_{1}\right\}  .
\]
Consequently, by the properties of $I$, we have%
\[
\int_{0}^{t_{2}-t_{1}}I(\mu_{f_{t_{1}}^{t_{2}}}(s))ds\geq(t_{2}-t_{1}%
)\min\{I(\mu\left\{  \left|  f\right|  \geq t_{2}\right\}  ),I(\mu\left\{
\left|  f\right|  >t_{1}\right\}  )\}.
\]
Let us see that $f_{\mu}^{\ast}$ is locally absolutely continuous. Indeed, for
$s>0$ and $h>0,$ pick $t_{1}=f_{\mu}^{\ast}(s+h),$ $t_{2}=f_{\mu}^{\ast}(s),$
then
\begin{equation}
s\leq\mu\left\{  \left|  f(x)\right|  \geq f_{\mu}^{\ast}(s)\right\}  \leq
\mu_{f_{t_{1}}^{t_{2}}}(s)\leq\mu\left\{  \left|  f(x)\right|  >f_{\mu}^{\ast
}(s+h)\right\}  \leq s+h. \label{A2}%
\end{equation}
Combining (\ref{A1}) and (\ref{A2}) we have,
\begin{equation}
\left(  f_{\mu}^{\ast}(s)-f_{\mu}^{\ast}(s+h)\right)  \min\{I(s+h),I(s)\}\leq
\int_{\left\{  f_{\mu}^{\ast}(s+h)<\left|  f\right|  <f_{\mu}^{\ast
}(s)\right\}  }\left|  \nabla\left|  f\right|  (x)\right|  d\mu
\label{truncation}%
\end{equation}
which implies that $f_{\mu}^{\ast}$ is absolutely continuous in $[a,b]$
($0<a<b<1).$ Indeed, for any finite family of non-overlapping intervals
$\{\left(  a_{k},b_{k}\right)  \}_{k=1}^{r},$ with $\left(  a_{k}%
,b_{k}\right)  \subset\lbrack a,b],$ and, $\sum_{k=1}^{r}\left(  b_{k}%
-a_{k}\right)  \leq\delta,$ we have%
\[
\mu\left\{  \cup_{k=1}^{r}\left\{  f_{\mu}^{\ast}(b_{k})<\left|  f\right|
<f_{\mu}^{\ast}(a_{k})\right\}  \right\}  =\sum_{k=1}^{r}\mu\left\{  f_{\mu
}^{\ast}(b_{k})<\left|  f\right|  <f_{\mu}^{\ast}(a_{k})\right\}  \leq
\sum_{k=1}^{r}\left(  b_{k}-a_{k}\right)  \leq\delta.
\]
Therefore, combining this fact with (\ref{truncation}), we have%
\begin{align*}
\sum_{k=1}^{r}\left(  f_{\mu}^{\ast}(a_{k})-f_{\mu}^{\ast}(b_{k})\right)
\min\{I(a),I(b)\}  &  \leq\sum_{k=1}^{r}\left(  f_{\mu}^{\ast}(a_{k})-f_{\mu
}^{\ast}(b_{k})\right)  \min\{I(a_{k}),I(b_{k})\}\\
&  \leq\sum_{k=1}^{r}\int_{\left\{  f_{\mu}^{\ast}(b_{k})<\left|  f\right|
<f_{\mu}^{\ast}(a_{k})\right\}  }\left|  \nabla\left|  f\right|  (x)\right|
d\mu\\
&  =\int_{\cup_{k=1}^{r}\left\{  f_{\mu}^{\ast}(b_{k})<\left|  f\right|
<f_{\mu}^{\ast}(a_{k})\right\}  }\left|  \nabla\left|  f\right|  (x)\right|
d\mu\\
&  \leq\int_{0}^{\delta}\left|  \nabla\left|  f\right|  \right|  _{\mu}^{\ast
}(t)dt\\
&  \leq\int_{0}^{\delta}\left|  \nabla f\right|  _{\mu}^{\ast}(t)dt.
\end{align*}

The local absolute continuity follows.

Finally, using (\ref{truncation}) again we get,
\begin{align*}
\frac{\left(  f_{\mu}^{\ast}(s)-f_{\mu}^{\ast}(s+h)\right)  }{h}%
\min(I(s+h),I(s))  &  \leq\int_{\left\{  f_{\mu}^{\ast}(s+h)<\left|  f\right|
<f_{\mu}^{\ast}(s)\right\}  }\left|  \nabla\left|  f\right|  (x)\right|
d\mu\\
&  \leq\frac{1}{h}\int_{\left\{  f_{\mu}^{\ast}(s+h)<\left|  f\right|  \leq
f_{\mu}^{\ast}(s)\right\}  }\left|  \nabla\left|  f\right|  (x)\right|  d\mu\\
&  \leq\frac{1}{h}\int_{\left\{  f_{\mu}^{\ast}(s+h)<\left|  f\right|  \leq
f_{\mu}^{\ast}(s)\right\}  }\left|  \nabla f(x)\right|  d\mu.
\end{align*}
Letting $h\rightarrow0$ we obtain (\ref{dosa}).

$\mathbf{(2)\Rightarrow(4).}$ As before, the truncation argument shows that%
\[
\int_{0}^{t_{2}-t_{1}}I(\mu_{f_{t_{1}}^{t_{2}}}(s))ds\leq\int_{\left\{
t_{1}<\left|  f\right|  <t_{2}\right\}  }\left|  \nabla\left|  f\right|
\right|  \chi_{\left\{  t_{1}<\left|  f\right|  <t_{2}\right\}  }d\mu.
\]
Observe that for $0<s<t_{2}-t_{1}$
\[
\mu_{f_{t_{1}}^{t_{2}}}(s)=\mu\left\{  \left|  f\right|  >t_{1}+s\right\}
=\mu_{f}(t_{1}+s),
\]
thus%
\[
\int_{0}^{t_{2}-t_{1}}I(\mu_{f_{t_{1}}^{t_{2}}}(s))ds=\int_{t_{1}}^{t_{2}%
}I(\mu_{f}(s))ds.
\]
We have seen in the proof of $[(2)\Rightarrow(3)]$ that $f_{\mu}^{\ast}$ is
absolutely continuous.$\ $Therefore we get%
\begin{equation}
\int_{t_{1}}^{t_{2}}I(\mu_{f}(s))ds=\int_{\mu_{f}(t_{2})}^{\mu_{f}(t_{1}%
)}I(\mu_{f}(f_{\mu}^{\ast}(s)))\left(  -f_{\mu}^{\ast}\right)  ^{\prime}(s)ds.
\label{previo}%
\end{equation}
Let $m$ be the Lebesgue on $[0,\infty),$ then (see \cite[Lemma 1, pag. 84]%
{Cw})
\begin{equation}
s-m\left\{  r\in(0,\infty):f_{\mu}^{\ast}(r)=f_{\mu}^{\ast}(s)\right\}  \leq
m_{f_{\mu}^{\ast}}(f_{\mu}^{\ast}(s))\leq s. \label{cwi}%
\end{equation}
Recall that since $f$ and $f_{\mu}^{\ast}$ are equimeasurable,
\[
\mu_{f}(s)=m_{f_{\mu}^{\ast}}(s),\text{ for all }s\geq0.
\]
Inserting this in (\ref{cwi}) we find
\[
s-m\left\{  r\in(0,\infty):f_{\mu}^{\ast}(r)=f_{\mu}^{\ast}(s)\right\}
\leq\mu_{f}(f_{\mu}^{\ast}(s))\leq s.
\]
It follows that $\mu_{f}(f_{\mu}^{\ast}(s))=s,$ unless $s$ belongs to an
interval where $f_{\mu}^{\ast}$ is constant, in which case $\left(  f_{\mu
}^{\ast}\right)  ^{\prime}=0.$ Therefore, if we set $t_{1}=f_{\mu}^{\ast}(a)$
and $t_{2}=f_{\mu}^{\ast}(b)$ ($a<b)$ in (\ref{previo}), we obtain
\begin{align}
\int_{f_{\mu}^{\ast}(a)}^{f_{\mu}^{\ast}(b)}I(\mu_{f}(s))ds  &  =\int_{\mu
_{f}(f_{\mu}^{\ast}(a))}^{\mu_{f}(f_{\mu}^{\ast}(b))}I(\mu_{f}(f_{\mu}^{\ast
}(s)))\left(  -f_{\mu}^{\ast}\right)  ^{\prime}(s)ds\nonumber\\
&  =\int_{a}^{b}I(s)\left(  -f_{\mu}^{\ast}\right)  ^{\prime}(s)ds.
\label{salvado}%
\end{align}
Consider a finite family of intervals $\left(  a_{i},b_{i}\right)  ,$
$i=1,\ldots,k$, with $0<a_{1}<b_{1}\leq a_{2}<b_{2}\leq\cdots\leq a_{k}%
<b_{k}<1.$ Then,
\begin{align*}
\int_{\cup_{1\leq i\leq k}(a_{i},b_{i})}\left(  -f_{\mu}^{\ast}\right)
^{^{\prime}}(s)I(s)ds  &  =\sum_{i=1}^{k}\int_{f_{\mu}^{\ast}(a_{i})}^{f_{\mu
}^{\ast}(b_{i})}I(\mu_{f}(s))ds\text{ \ \ (by (\ref{salvado}))}\\
&  \leq\sum_{i=1}^{k}\int_{\left\{  f_{\mu}^{\ast}(b_{i})<\left|  f\right|
<f_{\mu}^{\ast}(a_{i})\right\}  }\left|  \nabla\left|  f\right|  (x)\right|
d\mu\\
&  =\int_{\cup_{1\leq i\leq k}\left\{  f_{\mu}^{\ast}(b_{i})<\left|  f\right|
<f_{\mu}^{\ast}(a_{i})\right\}  }\left|  \nabla\left|  f\right|  (x)\right|
d\mu\\
&  \leq\int_{0}^{\sum_{i=1}^{k}\left(  b_{i}-a_{i}\right)  }\left|
\nabla\left|  f\right|  \right|  _{\mu}^{\ast}(s)ds\\
&  \leq\int_{0}^{\sum_{i=1}^{k}\left(  b_{i}-a_{i}\right)  }\left|  \nabla
f\right|  _{\mu}^{\ast}(s)ds.
\end{align*}
Now, by a routine limiting process we can show that, for any measurable set
$E\subset$ $(0,1),$ we have
\[
\int_{E}(-f_{\mu}^{\ast})^{\prime}(s)I(s)ds\leq\int_{0}^{m(E)}\left|  \nabla
f\right|  _{\mu}^{\ast}(s)ds.
\]
Therefore,%
\begin{align*}
\sup_{m(E)\leq t}\int_{E}(-f_{\mu}^{\ast})^{\prime}(s)I(s)ds  &  \leq
\sup_{m(E)\leq t}\int_{0}^{m(E)}\left|  \nabla f\right|  _{\mu}^{\ast}(s)ds\\
&  =\int_{0}^{t}\left|  \nabla f\right|  _{\mu}^{\ast}(s)ds.
\end{align*}
Consequently by (\ref{hp}) we get%
\[
\int_{0}^{t}((-f_{\mu}^{\ast})^{\prime}(\cdot)I(\cdot))^{\ast}(s)ds\leq
\int_{0}^{t}\left|  \nabla f\right|  _{\mu}^{\ast}(s)ds.
\]

$\mathbf{(3)\Rightarrow(5).}$ We will integrate by parts. Let us note first
that using (\ref{truncation}) we have that, for $0<s<t,$
\begin{equation}
s\left(  f_{\mu}^{\ast}(s)-f_{\mu}^{\ast}(t\right)  )\leq\frac{s}%
{\min\{I(s),I(t)\}}\int_{0}^{t-s}\left|  \nabla\left|  f\right|  \right|
_{\mu}^{\ast}(s)ds. \label{boca}%
\end{equation}
Now, using (\ref{boca}) we see that $\lim_{s\rightarrow0}s\left(  f_{\mu
}^{\ast}(s)-f_{\mu}^{\ast}(t\right)  )<\infty.$ Therefore,%
\begin{align}
f_{\mu}^{\ast\ast}(t)-f_{\mu}^{\ast}(t)  &  =\frac{1}{t}\int_{0}^{t}\left(
f_{\mu}^{\ast}(s)-f_{\mu}^{\ast}(t)\right)  ds\nonumber\\
&  =\frac{1}{t}\left\{  \left[  s\left(  f_{\mu}^{\ast}(s)-f_{\mu}^{\ast
}(t)\right)  \right]  _{0}^{t}+\int_{0}^{t}s\left(  -f_{\mu}^{\ast}\right)
^{\prime}(s)ds\right\} \nonumber\\
&  \leq\frac{1}{t}\int_{0}^{t}s\left(  -f_{\mu}^{\ast}\right)  ^{\prime
}(s)ds\nonumber\\
&  =A(t).\nonumber
\end{align}
Since $s/I(s)$ is increasing on $0<s<1$, we get
\begin{align*}
A(t)  &  \leq\frac{1}{I(t)}\int_{0}^{t}I(s)\left(  -f_{\mu}^{\ast}\right)
^{\prime}(s)ds\\
&  \leq\frac{1}{I(t)}\int_{0}^{t}\left(  \frac{\partial}{\partial s}%
\int_{\left\{  \left|  f\right|  >f_{\mu}^{\ast}(s)\right\}  }\left|  \nabla
f(x)\right|  d\mu\right)  ds\text{ (by (\ref{dosa}))}\\
&  \leq\frac{1}{I(t)}\int_{\left\{  \left|  f\right|  >f_{\mu}^{\ast
}(t)\right\}  }\left|  \nabla f(x)\right|  d\mu\text{ }\\
&  \leq\frac{t}{I(t)}\left|  \nabla f\right|  _{\mu}^{\ast\ast}(t).
\end{align*}

$\mathbf{(4)\Rightarrow(5)}.$ Once again we use integration by parts. We now
show that under our current assumptions (\ref{boca}) still holds. Let $0<s<t.$
Since $I$ increases on $(0,1/2),$ and is symmetric about $1/2,$ we have%
\[
\left(  f_{\mu}^{\ast}(s)-f_{\mu}^{\ast}(t\right)  )\min\{I(t),I(s)\}\leq
\int_{s}^{t}(-f_{\mu}^{\ast})^{\prime}(z)I(z)dz.
\]
Therefore, by the basic properties of rearrangements,
\begin{align*}
\left(  f_{\mu}^{\ast}(s)-f_{\mu}^{\ast}(t\right)  )\min\{I(t),I(s)\}  &
\leq\int_{0}^{t-s}((-f_{\mu}^{\ast})^{\prime}(.)I(.))^{\ast}(z)dz\\
&  \leq\int_{0}^{t-s}\left|  \nabla f\right|  _{\mu}^{\ast}(z)dz.
\end{align*}
Thus, once again we have
\begin{equation}
s\left(  f_{\mu}^{\ast}(s)-f_{\mu}^{\ast}(t\right)  )\leq\frac{s}%
{\min\{I(t),I(s)\}}\int_{0}^{t-s}\left|  \nabla\left|  f\right|  \right|
_{\mu}^{\ast}(z)dz. \label{boca01}%
\end{equation}
Therefore proceeding as before we find%
\begin{align*}
f_{\mu}^{\ast\ast}(t)-f_{\mu}^{\ast}(t)  &  \leq\frac{1}{t}\int_{0}%
^{t}s\left(  -f_{\mu}^{\ast}\right)  ^{^{\prime}}(s)ds\\
&  \leq\frac{1}{I(t)}\int_{0}^{t}I(s)\left(  -f_{\mu}^{\ast}\right)  ^{\prime
}(s)ds\\
&  \leq\frac{1}{I(t)}\int_{0}^{t}(\left(  -f_{\mu}^{\ast}\right)  ^{\prime
}(.)I(.))^{\ast}(s)ds\text{,}%
\end{align*}
where in the last step we used a basic property of the decreasing
rearrangement. Combining the last estimate with (\ref{provadas}) we find that%
\[
f_{\mu}^{\ast\ast}(t)-f_{\mu}^{\ast}(t)\leq\frac{t}{I(t)}\left|  \nabla
f\right|  _{\mu}^{\ast\ast}(t),
\]
as we wished to show.

$\mathbf{(5)\Rightarrow(1).}$ Let $A$ be a Borel set with $0<\mu(A)<1.$ We may
assume, without loss, that $\mu^{+}(A)<\infty.$ By \cite[Lemma 3.7]{BH} we can
select a sequence $\{f_{n}\}_{n\in N}$ of Lip functions such that
$f_{n}\underset{L^{1}}{\rightarrow}\chi_{A}$, and%
\[
\mu^{+}(A)\geq\lim\sup_{n\rightarrow\infty}\left\|  \left|  \nabla
f_{n}\right|  \right\|  _{L^{1}}.
\]
Therefore,%
\begin{align}
\lim\sup_{n\rightarrow\infty}I(t)(\left(  f_{n}\right)  _{\mu}^{\ast\ast
}(t)-\left(  f_{n}\right)  _{\mu}^{\ast}(t))  &  \leq\lim\sup_{n\rightarrow
\infty}\int_{0}^{t}\left|  \nabla f_{n}(s)\right|  _{\mu}^{\ast}%
ds\label{bbb}\\
&  \leq\lim\sup_{n\rightarrow\infty}\int_{\Omega}\left|  \nabla f_{n}\right|
d\mu\nonumber\\
&  \leq\mu^{+}(A).\nonumber
\end{align}
As is well known, $f_{n}\underset{L^{1}}{\rightarrow}\chi_{A}$ implies that
(cf. \cite[Lemma 2.1]{gar}):
\begin{align*}
\left(  f_{n}\right)  _{\mu}^{\ast\ast}(t)\rightarrow\left(  \chi_{A}\right)
_{\mu}^{\ast\ast}(t)  &  ,\text{ uniformly for }t\in\lbrack0,1]\text{, and }\\
\left(  f_{n}\right)  _{\mu}^{\ast}(t)\rightarrow\left(  \chi_{A}\right)
_{\mu}^{\ast}(t)\text{ }  &  \text{at all points of continuity of }\left(
\chi_{A}\right)  _{\mu}^{\ast}.
\end{align*}
Let $r=\mu(A),$ and observe that $\left(  \chi_{A}\right)  _{\mu}^{\ast\ast
}(t)=\min\{1,\frac{r}{t}\},$ then, we deduce that for all $t>r,$ $\left(
f_{n}\right)  _{\mu}^{\ast\ast}(t)\rightarrow\frac{r}{t},$ and $\left(
f_{n}\right)  _{\mu}^{\ast}(t)\rightarrow\left(  \chi_{A}\right)  _{\mu}%
^{\ast}(t)=\chi_{(0,r)}(t)=0.$ Inserting this information back in (\ref{bbb}),
we get%
\[
\frac{r}{t}I(t)\leq\mu^{+}(A),\;\forall t>r.
\]
Now, since $I(t)$ is continuous, we may let $t\rightarrow r$ and we find that
\[
I(\mu(A))\leq\mu^{+}(A),
\]
as we wished to show.
\end{proof}

\begin{remark}
In connection with inequality (\ref{ledo}) see also Remark \ref{remarkao1} below.
\end{remark}

\begin{proposition}
\label{l1l1}Let $I:[0,1]\rightarrow\left[  0,\infty\right)  $ be an
isoperimetric estimator on $(\Omega,d,\mu).$ Suppose that there exists a
constant $c>0$ such that
\begin{equation}
\int_{t}^{1}\frac{I(s)}{s}\frac{ds}{s}\leq c\frac{I(t)}{t},\;\;t\in(0,1).
\label{concon}%
\end{equation}
Then, for all $f\in Lip(\Omega),$
\begin{equation}
\int_{0}^{t}\left(  \frac{I(\cdot)}{\left(  \cdot\right)  }[f_{\mu}^{\ast\ast
}(\cdot)-f_{\mu}^{\ast}(\cdot)]\right)  ^{\ast}ds\leq4c\int_{0}^{t}\left|
\nabla f\right|  _{\mu}^{\ast}(s)ds. \label{l1}%
\end{equation}
\end{proposition}

\begin{proof}
We will first show that%
\begin{equation}
\int_{0}^{t}(f_{\mu}^{\ast\ast}(s)-f_{\mu}^{\ast}(s))\frac{I(s)}{s}ds\leq
c\int_{0}^{t}\left|  \nabla f\right|  _{\mu}^{\ast}(s)ds. \label{aaa}%
\end{equation}
As we have seen before
\[
t(f_{\mu}^{\ast\ast}(t)-f_{\mu}^{\ast}(t))\leq\int_{0}^{t}s\left(  -f_{\mu
}^{\ast}\right)  ^{\prime}(s)ds.
\]
Therefore, the left hand side of (\ref{aaa}) is controlled by
\[
B(t)=\int_{0}^{t}\left(  \int_{0}^{s}x\left(  -f_{\mu}^{\ast}\right)
^{\prime}(x)dx\right)  \frac{I(s)}{s^{2}}ds.
\]
Using our current assumptions and Fubini's theorem, we find%
\begin{align*}
B(t)  &  =\int_{0}^{t}x\left(  -f_{\mu}^{\ast}\right)  ^{\prime}(x)\int
_{x}^{t}\frac{I(s)}{s^{2}}dsdx\\
&  \leq\int_{0}^{t}x\left(  -f_{\mu}^{\ast}\right)  ^{\prime}(x)\int_{x}%
^{1}\frac{I(s)}{s^{2}}dsdx\\
&  \leq c\int_{0}^{t}x\left(  -f_{\mu}^{\ast}\right)  ^{\prime}(x)\frac
{I(x)}{x}dx\\
&  \leq c\int_{0}^{t}((-f_{\mu}^{\ast})^{\prime}(.)I(.))^{\ast}(s)ds\\
&  \leq c\int_{0}^{t}\left|  \nabla f\right|  _{\mu}^{\ast}(s)ds\text{ \ \ (by
(\ref{provadas})).}%
\end{align*}
The proof of (\ref{aaa}) is complete. By Theorem \ref{teomain} we also have
\[
(f_{\mu}^{\ast\ast}(t)-f_{\mu}^{\ast}(t))\leq\frac{t}{I(t)}\left|  \nabla
f\right|  _{\mu}^{\ast\ast}(t).
\]
Therefore, by Lemma 2 of \cite{mmpote}, we see that (\ref{l1}) holds.
\end{proof}

\begin{remark}
\label{re:concavo}Suppose that there exists $\alpha>1,$ such that the
isoperimetric estimator $I^{\alpha}$ is concave. Then, condition
(\ref{concon}) holds. In fact, since the function $I(s)/s^{1/\alpha}$ $\ $is
decreasing, it follows that%
\begin{align*}
\int_{t}^{1}\frac{I(s)}{s}\frac{ds}{s}  &  =\int_{t}^{1}\frac{I(s)}%
{s^{1/\alpha}}\frac{ds}{s^{2-1/\alpha}}\\
&  \leq\frac{I(t)}{t^{1/\alpha}}\int_{t}^{1}\frac{ds}{s^{2-1/\alpha}}\\
&  \leq\frac{\alpha}{\alpha+1}\frac{I(t)}{t}.
\end{align*}
\end{remark}

\begin{remark}
\label{nesta}We note for future use that if (\ref{concon}) holds then
Proposition \ref{l1} implies that for all r.i. spaces $X$ (cf. the discussion
in Section \ref{secc:ri} below) we have%
\[
\left\|  \left(  f_{\mu}^{\ast\ast}(t)-f_{\mu}^{\ast}(t)\right)  \frac
{I(t)}{t}\right\|  _{\bar{X}}\leq\left\|  \left|  \nabla f\right|  \right\|
_{X}.
\]
\end{remark}

\section{P\'{o}lya-Szeg\"{o}\label{secc::p-s}}

The theme of this section is that, under the presence of more symmetry, we can
chose a special rearrangement such that the general P\'{o}lya-Szeg\"{o}
inequality takes a more familiar form, to wit: ``there is a special
symmetrization that does not increase the norm of the gradient''. As an
application, in the next sections we shall show sharp Poincar\'{e}-Sobolev
inequalities for our model cases.

\subsection{Model Case 1: log concave measures\label{secc:logconcave}}

We consider product measures on $\mathbb{R}^{n}$ constructed using measures on
$\mathbb{R}$ defined by%
\[
\mu^{\Phi}=Z_{\Phi}^{-1}\exp\left(  -\Phi(\left|  x\right|  )\right)
dx=\varphi(x)dx,
\]
where $\Phi$ is convex, $\sqrt{\Phi}$ concave and where $Z_{\Phi}^{-1}$ is
chosen to ensure that $\mu^{\Phi}(\mathbb{R)=}1$. It is known that the
isoperimetric problem is solved by half-lines (cf. \cite{Bor} and \cite{Bob})
and the isoperimetric profile is given by
\[
I_{\mu^{\Phi}}(t)=\varphi\left(  H^{-1}(\min\{t,1-t\}\right)  =\varphi\left(
H^{-1}(t\right)  ),\text{ \ \ \ }t\in\lbrack0,1],
\]
where $H$ is the distribution function of $\mu^{\Phi}$, i.e. $H:\mathbb{R}%
\rightarrow(0,1)$ is the increasing function given by
\[
H(r)=\int_{-\infty}^{r}\varphi(x)dx.
\]

In what follows we will, furthermore, assume that $\Phi(0)=0,$ and that $\Phi$
is $\mathcal{C}^{2}$ on $[\Phi^{-1}(1),+\infty);$ then it is known (see
\cite{BCR}) that there exist constants $c_{1,},c_{2}$ such that, for all
$t\in\lbrack0,1]$,%
\begin{equation}
c_{1}L_{\Phi}(t)\leq I_{\mu^{\Phi}}(t)\leq c_{2}L_{\Phi}(t), \label{qq}%
\end{equation}
where%
\[
L_{\Phi}(t)=\min\{t,1-t\}\Phi^{\prime}\circ\Phi^{-1}\left(  \log\frac{1}%
{\min\{t,1-t\}}\right)  .
\]

We consider the product probability measures $\mu^{\Phi\otimes n}$ on
$\mathbb{R}^{n}.$ Their isoperimetric profiles $I_{\mu^{\Phi\otimes n}}$ are
dimension free (cf. \cite{BCR}): there exists a universal constant $c(\Phi)$
such that
\begin{equation}
I_{\mu^{\Phi}}(t)\geq\inf_{n\geq1}I_{\mu^{\Phi\otimes n}}(t)\geq c(\Phi
)I_{\mu^{\Phi}}(t). \label{tenso}%
\end{equation}
In what follows we shall write $\mu=\mu^{\Phi\otimes n}.$ For a measurable set
$\Omega\subset\mathbb{R}^{n},$ we let $\Omega^{\circ}$ be the half space
defined by
\[
\Omega^{\circ}=\{x=(x_{1},.....x_{n}):x_{1}<r\},\;\;r\in\mathbb{R},
\]
where $r\in\mathbb{R}$ is selected so that
\[
\mu(\Omega^{\circ})=\mu(\Omega),\text{ or more explicitly }r=H^{-1}(\mu
(\Omega)).
\]
It follows from (\ref{tenso}) that
\begin{align*}
\mu^{+}(\Omega)  &  \geq I_{\mu}(\mu(\Omega))\\
&  \geq c(\Phi)I_{\mu^{\Phi}}(\mu(\Omega^{\circ}))\\
&  =c(\Phi)\varphi\left(  H^{-1}(\mu(\Omega)\right)  )\\
&  =c(\Phi)\mu^{+}(\Omega^{\circ}).
\end{align*}

There is a natural rearrangement associated with the symmetrization operation
$\Omega\rightarrow\Omega^{\circ}.$ For $f:\mathbb{R}^{n}\rightarrow\mathbb{R}$
we let%
\[
f^{\circ}(x)=f_{\mu}^{\ast}(H(x_{1})).
\]

\begin{remark}
Note that, as in the Euclidean case, $f^{\circ}$ is equimeasurable with $f:$%
\begin{align*}
\mu_{f^{\circ}}(t)  &  =\mu\{x:f^{\circ}(x)>t\})=\mu\{x:f_{\mu}^{\ast}%
(H(x_{1}))>t\}\\
&  =\mu\{x:H(x_{1})\leq\mu_{f}(t)\}=\mu\{x:x_{1}\leq H^{-1}(\mu_{f}(t))\}\\
&  =\mu^{\Phi}\left\{  (-\infty,H^{-1}(\mu_{f}(t)))\right\} \\
&  =\mu_{f}(t).
\end{align*}
\end{remark}

We can now show the following generalization of the P\'{o}lya-Szeg\"{o} principle.

\begin{theorem}
\label{polya-szego} Consider the probability space $(\mathbb{R}^{n},\mu),$
with $\mu=\mu^{\Phi\otimes n}.$ The following P\'{o}lya-Szeg\"{o} inequality
holds: for all $f\in Lip(\mathbb{R}^{n}),$%
\begin{equation}
\int_{0}^{t}\left|  \nabla f^{\circ}\right|  _{\mu}^{\ast}(s)ds\leq\frac
{1}{c(\Phi)}\int_{0}^{t}\left|  \nabla f\right|  _{\mu}^{\ast}(s)ds.
\label{provadasLog}%
\end{equation}
In fact, (\ref{provadasLog}) is equivalent to all the inequalities listed in
Theorem \ref{teomain} above.
\end{theorem}

\begin{proof}
Let $N$ be an arbitrary Young's function. Let $s=H(x_{1})$. Then,
\begin{align*}
\int_{0}^{1}N\left(  (-f_{\mu}^{\ast})^{\prime}(s)I_{\mu^{\Phi}}(s)\right)
ds  &  =\int_{\mathbb{R}}N(\left(  -f_{\mu}^{\ast}\right)  ^{\prime}%
(H(x_{1}))I_{\mu^{\Phi}}(H(x_{1}))\left|  H^{\prime}(x_{1})\right|  dx_{1}\\
&  =\int_{\mathbb{R}^{n}}N(\left(  -f_{\mu}^{\ast}\right)  ^{\prime}%
(H(x_{1}))I_{\mu^{\Phi}}(H(x_{1}))d\mu\\
&  =\int_{\mathbb{R}^{n}}N(\left|  \nabla f^{\circ}(x)\right|  )d\mu,
\end{align*}
where in the last step we have used the fact that
\[
\left|  \nabla f^{\circ}(x)\right|  =(f_{\mu}^{\ast})^{\prime}(H(x_{1}%
))H^{\prime}(x_{1})=(-f_{\mu}^{\ast})^{\prime}(H(x_{1}))I_{\mu^{\Phi}}%
(H(x_{1})).
\]
Since $N$ is increasing, then by \cite[exercise 3 pag. 88]{BS}, we have
\[
\int_{\mathbb{R}^{n}}N(\left|  \nabla f^{\circ}(x)\right|  )d\mu=\int_{0}%
^{1}N\left(  \left|  \nabla f^{\circ}\right|  _{\mu}^{\ast}(s)\right)  ds.
\]
Thus,
\[
\int_{0}^{1}N\left(  (-f_{\mu}^{\ast})^{\prime}(s)I_{\mu^{\Phi}}(s)\right)
ds=\int_{0}^{1}N\left(  \left|  \nabla f^{\circ}\right|  _{\mu}^{\ast
}(s)\right)  ds.
\]
Therefore, by \cite[exercise 5 pag. 88]{BS}, we have
\begin{equation}
\int_{0}^{t}((-f_{\mu}^{\ast})^{\prime}(\cdot)I_{\mu^{\Phi}}(\cdot))^{\ast
}(s)ds=\int_{0}^{t}\left|  \nabla f^{\circ}\right|  _{\mu}^{\ast}ds.
\label{psps}%
\end{equation}
Combining (\ref{psps}) with (\ref{tenso}) and (\ref{provadas}) we find
\begin{align*}
\int_{0}^{t}\left|  \nabla f^{\circ}\right|  _{\mu}^{\ast}ds  &  =\int_{0}%
^{t}((-f_{\mu}^{\ast})^{\prime}(\cdot)I_{\mu^{\Phi}}(\cdot))^{\ast}(s)ds\\
&  \leq\frac{1}{c(\Phi)}\int_{0}^{t}((-f_{\mu}^{\ast})^{\prime}(\cdot)I_{\mu
}(\cdot))^{\ast}(s)\\
&  \leq\frac{1}{c(\Phi)}\int_{0}^{t}\left|  \nabla f\right|  _{\mu}^{\ast
}(s)ds,
\end{align*}
as we wished to show.
\end{proof}

\begin{remark}
If $\mu^{\Phi}$ is the Gaussian measure, then $c(\Phi)=1$, and we recover the
classical Gaussian P\'{o}lya-Szeg\"{o} principle (see \cite{erh}).
\end{remark}

\subsection{Model Case 2: the $n-$sphere\label{secc:sphere}}

Let $\mathbb{S}^{n}$ be the unit sphere in $\mathbb{R}^{n+1},n\geq2.$ Let
$\omega_{n}=2\pi^{\frac{n+1}{2}}/\Gamma(\frac{n+1}{2})$ be the $n-$dimensional
Hausdorff measure of $\mathbb{S}^{n}.$ On $\mathbb{S}^{n}$ we consider the
geodesic distance $d$ and the uniform probability measure $\sigma_{n}$. For
$\theta\in\lbrack-\pi/2,\pi/2]$, let%

\[
\varphi_{n}(\theta)=\frac{\omega_{n-1}}{\omega_{n}}\cos^{n-1}\theta
\text{\ \ \ \ and \ \ \ \ \ }\Phi_{n}(\theta)=\int_{-\pi/2}^{\theta}%
\varphi_{n}(s)ds.
\]
The spherical cap
\[
C_{\theta}=\left\{  (\theta_{1},.....,\theta_{n})\in\mathbb{S}^{n}:\theta
_{1}<\theta\right\}
\]
has $\sigma_{n}-$measure $\Phi_{n}(\theta)$ and boundary measure $\varphi
_{n}(\theta)$. Thus, by the L\'{e}vy-Schmidt result, the isoperimetric
function of the sphere $I_{\mathbb{S}^{n}}$ coincides with $I_{n}=\varphi
_{n}\circ\Phi_{n}^{-1}$ $\ $(see \cite{bart})$.$ This function is continuous
on $[0,1]$ and symmetric with respect to $1/2,$ and $I_{n}(0)=I_{n}(1)=0.$
Moreover, $\left(  I_{n}\right)  ^{\frac{n}{n-1}}$ is concave.

Given a measurable set $\Omega\subset\mathbb{S}^{n},$ we let $\Omega^{\circ}$
be the spherical cap defined by
\[
\Omega^{\circ}=\{(\theta_{1},.....,\theta_{n})\in\mathbb{S}^{n}:\theta
_{1}<\theta\},
\]
where $\theta\in\lbrack-\pi/2,\pi/2]$ is selected so that
\[
\Phi_{n}(\theta)=\sigma_{n}(\Omega).
\]
In other words, $\theta$ is defined by
\[
\theta=\Phi^{-1}(\sigma_{n}(\Omega)).
\]
Since spherical caps are the subsets of $\mathbb{S}^{n}$ which yield the
equality in the isoperimetric inequality, we get
\[
\sigma_{n}^{+}(\Omega)\geq I_{n}(\sigma_{n}(\Omega))=\sigma_{n}^{+}%
(\Omega^{\circ}).
\]
Let $f:\mathbb{S}^{n}\rightarrow\mathbb{R}$, associated with the operation
$\Omega\rightarrow\Omega^{\circ}$ we define the rearrangement $f^{\circ}$ by
\[
f^{\circ}(\theta_{1},.....,\theta_{n})=f_{\sigma_{n}}^{\ast}(\Phi_{n}%
(\theta_{1})).
\]

\begin{theorem}
\label{polya-szego-esfe} Consider the space $(\mathbb{S}^{n},d,\sigma_{n}).$
The following P\'{o}lya-Szeg\"{o} inequality holds: for all f $\in
Lip(\mathbb{S}^{n}),$%
\begin{equation}
\int_{0}^{t}\left|  \nabla f^{\circ}\right|  _{\sigma_{n}}^{\ast}(s)ds\leq
\int_{0}^{t}\left|  \nabla f\right|  _{\sigma_{n}}^{\ast}(s)ds.
\label{provadasSn}%
\end{equation}
Moreover, (\ref{provadasSn}) is equivalent to any of the inequalities stated
in Theorem \ref{teomain} above.
\end{theorem}

\begin{proof}
The proof is almost identical to the proof of Theorem \ref{polya-szego}. Using
spherical coordinates we have
\[
\omega_{n}=\int_{(-\pi/2,\pi/2)^{n-1}\times(-\pi,\pi)}\prod_{i=1}^{n-1}%
\cos^{n-i}\theta_{i}d\theta_{1}\cdots d\theta_{n}=\int_{\mathbb{S}^{n}}%
s_{n}(\theta^{\otimes n})d\theta^{\otimes n}.
\]
Therefore,
\[
d\sigma_{n}=\frac{1}{\omega_{n}}s_{n}(\theta^{\otimes n})d\theta^{\otimes n}.
\]
Let $N$ be a Young's function, and let $s=\Phi_{n}(\theta_{1})$. For
notational convenience we let $I=\int_{0}^{1}N\left(  (-f_{\sigma_{n}}^{\ast
})^{\prime}(s)I_{n}(s)\right)  ds.$ Then,%
\begin{align*}
I &  =\int_{-\pi/2}^{\pi/2}N(\left(  -f_{\sigma_{n}}^{\ast}\right)  ^{\prime
}\Phi_{n}(\theta_{1}))I_{n}(\Phi_{n}(\theta_{1}))\left|  \Phi_{n}^{\prime
}(\theta_{1})\right|  d\theta_{1}\\
&  =\int_{-\pi/2}^{\pi/2}N(\left(  -f_{\sigma_{n}}^{\ast}\right)  ^{\prime
}\Phi_{n}(\theta_{1}))I_{n}(\Phi_{n}(\theta_{1}))\frac{\omega_{n-1}}%
{\omega_{n}}\cos^{n-1}\theta_{1}d\theta_{1}\\
&  =\int_{\mathbb{S}^{n-1}}s_{n-1}(\theta^{\otimes(n-1)})d\theta
^{\otimes(n-1)}\int_{-\pi/2}^{\pi/2}N(\left|  \nabla f^{\circ}(\theta
_{1},.....,\theta_{n})\right|  )\frac{1}{\omega_{n}}\cos^{n-1}\theta
_{1}d\theta_{1}\\
&  =\int_{\mathbb{S}^{n}}N(\left|  \nabla f^{\circ}(\theta_{1},.....,\theta
_{n})\right|  )\frac{1}{\omega_{n}}s_{n}(\theta^{\otimes n})d\theta^{\otimes
n}\\
&  =\int_{\mathbb{S}^{n}}N(\left|  \nabla f^{\circ}(\theta_{1},.....,\theta
_{n})\right|  )d\sigma_{n},
\end{align*}
where we have used the fact that
\begin{align*}
(-f_{\sigma_{n}}^{\ast})^{\prime}(\Phi_{n}(\theta_{1}))I_{n}(\Phi_{n}%
(\theta_{1})) &  =(f_{\sigma_{n}}^{\ast})^{\prime}(\Phi_{n}(\theta_{1}%
))\Phi_{n}^{\prime}(\theta_{1})\\
&  =\left|  \nabla f^{\circ}(\theta_{1},.....,\theta_{n})\right|  .
\end{align*}
At this point we proceed in the same way as in the proof of Theorem
\ref{polya-szego}.
\end{proof}

\begin{remark}
\label{preferida}Since $\left(  I_{n}\right)  ^{\frac{n}{n-1}}$ is concave,
then by Remark \ref{re:concavo} we have that for all $f\in Lip(\mathbb{S}%
^{n})$
\begin{equation}
\int_{0}^{t}\left(  \frac{I(\cdot)}{\left(  \cdot\right)  }[f_{\sigma_{n}%
}^{\ast\ast}(\cdot)-f_{\sigma_{n}}^{\ast}(\cdot)]\right)  ^{\ast}ds\leq
4c\int_{0}^{t}\left|  \nabla f\right|  _{\sigma_{n}}^{\ast}(s)ds.
\label{destinada}%
\end{equation}
Therefore, (\ref{destinada}) is equivalent to any of the inequalities stated
in Theorem \ref{teomain} above.$\ $We also have (cf. Remark \ref{nesta} above)%
\[
\left\|  \frac{I(t)}{t}[f_{\sigma_{n}}^{\ast\ast}(t)-f_{\sigma_{n}}^{\ast
}(t)]\right\|  _{\bar{X}}\leq\left\|  \left|  \nabla f\right|  \right\|
_{X},
\]
without any restrictions on the indices of $X.$
\end{remark}

\subsection{Model Case 3: Model Riemannian manifolds\label{secc:ros}}

The analysis in the previous sections can be extended to a general class of
model spaces described for example in Ros \cite{Ros}, and the references
therein. In this section we complete the analysis of model spaces by showing
that the P\'{o}lya-Szeg\"{o} inequality holds for Ros's spaces.

We recall briefly the construction and refer to \cite{Ros} and \cite{mmmazya}
for more details. Let $M_{0}$ be an $n_{0}$-dimensional Riemannian manifold
with geodesic distance $d$. A probability measure $\mu^{0}$ on $M$ that is
absolutely continuous with respect to the volume $dVol_{M}$ will be called a
\textbf{model measure}, if there exists a continuous family (in the sense of
the Hausdorff distance on compact subsets) $\mathcal{D}=\{D^{t}:0\leq
t\leq1\}$ of closed subsets of $M_{0}$ satisfying the following conditions:

\begin{enumerate}
\item $D^{s}\subset D^{t},$ for $0\leq s<t\leq$ and $\mu^{0}(D^{t})=t,$

\item $D^{t}$ is a smooth isoperimetric domain of $\mu^{0}$ and $I_{\mu^{0}%
}(t)=\mu^{0}(D^{t})$ is positive and smooth for $0<t<1$, where $I_{\mu^{0}}$
denotes the isoperimetric profile of $M_{0},$

\item  The $r$-enlargement of $D^{t}$, defined by $(D^{t})_{r}=\{x\in
M_{0}:d(x,D^{t})\leq r\}$ verifies $(D^{t})_{r}=D^{s}$ for some $s=s(t,r)$,
$0\leq t\leq1,$

\item $D^{1}=M_{0}$ and $D^{0}$ is either a point or the empty set.

\item  We shall also assume that the corresponding isoperimetric profile
$I_{\mu}$ satisfies our usual assumptions (cf. Condition \ref{pedida} above).
\end{enumerate}

Let $f:M_{0}\rightarrow\mathbb{R}$. The rearrangement $f^{\circ}%
:M_{0}\rightarrow\mathbb{R},$ is defined by%
\[
f^{\circ}(x)=f_{\mu^{0}}^{\ast}(p(x)),
\]
where
\[%
\begin{array}
[c]{c}%
p:M_{0}\rightarrow\lbrack0,1]\\
x\in\partial D^{t}\rightarrow t,
\end{array}
\]
($\partial D^{t}$ denotes the boundary of $D^{t}).$ Since $p$ is measure
preserving (cf. \cite{mmmazya}) it is easy to verify that $f^{\circ}$ is
equimeasurable with $f:$%
\begin{align*}
\mu_{f^{\circ}}^{0}(t)  &  =\mu^{0}\{x:f^{\circ}(x)>t\}\\
&  =\mu^{0}\{x:f_{\mu^{0}}^{\ast}(p(x))>t\}\\
&  =\mu^{0}\{x:p(x)\leq\mu_{f}^{0}(t)\}\\
&  =\mu^{0}\{x:p^{-1}(0,\mu_{f}^{0}(t))\}\\
&  =\mu_{f}^{0}(t).
\end{align*}
Moreover, from (cf. \cite{mmmazya})
\[
\left|  \nabla p(x)\right|  =\left|  I_{\mu^{0}}(p(x))\right|
\]
we see that
\begin{align*}
\left|  \nabla f^{\circ}(x)\right|   &  =(-f_{\mu^{0}}^{\ast})^{\prime
}(p(x))\left|  \nabla p(x)\right| \\
&  =\left|  (-f_{\mu^{0}}^{\ast})^{\prime}(p(x))I_{\mu^{0}}(p(x))\right|  .
\end{align*}
Therefore the analysis of Theorem \ref{polya-szego} can be repeated verbatim
and yields

\begin{theorem}
Let $\left(  M_{0},d\right)  $ be an $n_{0}$-dimensional Riemannian manifold
endowed with a model measure $\mu_{0}.$ Then, the following P\'{o}%
lya-Szeg\"{o} inequality holds: for all $f\in Lip(M_{0})$%
\[
\int_{0}^{t}\left|  \nabla f^{\circ}\right|  _{\mu^{0}}^{\ast}(s)ds\leq
\int_{0}^{t}\left|  \nabla f\right|  _{\mu^{0}}^{\ast}(s)ds.
\]
\end{theorem}

\section{Poincar\'{e} Inequalities\label{secc::po}}

Let $\left(  \Omega,d,\mu\right)  $ be a metric probability space, and let $I$
be an isoperimetric estimator for $(\Omega,d,\mu).$

In this section we study Poincar\'{e} type inequalities of the form%
\begin{equation}
\left\|  g-\int_{\Omega}gd\mu\right\|  _{Y}\preceq\left\|  \left|  \nabla
g\right|  \right\|  _{X},\text{ \ \ }g\in Lip(\Omega), \label{poin00}%
\end{equation}
where $X,Y$ are rearrangement-invariant spaces on $\Omega.$

It is easy to see that, when $X=Y=L^{1}(\Omega),$ the inequality
(\ref{poin00}) follows readily from Ledoux's inequality (\ref{ledo}). Indeed,
using (\ref{ledo}) we can readily see that for all $f\in Lip(\Omega),$%
\begin{equation}
\int_{{\Omega}}\left|  f(x)-m_{e}\right|  d\mu\leq\frac{1}{2I(1/2)}%
\int_{{\Omega}}\left|  \nabla f(x)\right|  d\mu, \label{pr1}%
\end{equation}
where $m_{e}$ is a median of $f$. Indeed, set $f^{+}=\max(f-m_{e},0)$ and
$f^{-}=-\min(f-m_{e},0)$ so that $f-m_{e}=f^{+}-f^{-}.$ Then,%
\begin{align*}
\int_{{\Omega}}\left|  f-m_{e}\right|  d\mu &  =\int_{{\Omega}}f^{+}d\mu
+\int_{{\Omega}}f^{-}d\mu\\
&  =\int_{0}^{\infty}\mu_{f^{+}}(s)ds+\int_{0}^{\infty}\mu_{f^{-}}(s)ds\\
&  =(A),\text{ say.}%
\end{align*}
Each of these integrals can be estimated using the properties of the
isoperimetric estimator and Ledoux's inequality (\ref{ledo}). First we use the
fact that $\frac{I(s)}{s}$ is decreasing combined with the definition of
median, to find that%
\[
2\mu_{g}(s)I\left(  \frac{1}{2}\right)  \leq I(\mu_{g}(s)),\text{ where
}g=f^{+}\text{ or }g=f^{-}.
\]
Consequently,%
\begin{align*}
(A)  &  \leq\frac{1}{2I(\frac{1}{2})}\left(  \int_{0}^{\infty}I(\mu_{f^{+}%
}(s))ds+\int_{0}^{\infty}I(\mu_{f^{-}}(s))ds\right) \\
&  \leq\frac{1}{2I(\frac{1}{2})}\left(  \int_{{\Omega}}\left|  \nabla
f^{+}(x)\right|  d\mu+\int_{{\Omega}}\left|  \nabla f^{-}(x)\right|
d\mu\right)  \text{ (by (\ref{ledo}))}\\
&  =\frac{1}{2I(1/2)}\int_{{\Omega}}\left|  \nabla f(x)\right|  d\mu.
\end{align*}
Thus,%
\[
\int_{{\Omega}}\left|  f(x)-m_{e}\right|  d\mu\leq\frac{1}{2I(1/2)}%
\int_{{\Omega}}\left|  \nabla f(x)\right|  d\mu.
\]

The isoperimetric Hardy operator $Q_{I}$ is the operator defined on measurable
functions on $(0,1)$ by
\[
Q_{I}f(t)=\int_{t}^{1}f(s)\frac{ds}{I(s)},
\]
where $I$ is an isoperimetric estimator. We consider the possibility of
characterizing Poincar\'{e} inequalities of the form (\ref{poin00}) in terms
of the boundedness of $Q_{I}$ as an operator from $\bar{X}$ to $\bar{Y}.$

\begin{theorem}
\label{opti00}Let $X,Y$ be two r.i. spaces on $\Omega$. Suppose that there
exists an absolute constant $C$, such for every positive function $f\in\bar
{X},$ with supp$f\subset(0,1/2),$ we have
\begin{equation}
\left\|  Q_{I}f\right\|  _{\bar{Y}}\leq C\left\|  f\right\|  _{\bar{X}}.
\label{har}%
\end{equation}
Then, for all $g\in Lip(\Omega),$%
\begin{equation}
\left\|  g-\int_{\Omega}gd\mu\right\|  _{Y}\preceq\left\|  \left|  \nabla
g\right|  \right\|  _{X}. \label{revhad}%
\end{equation}
Moreover:

\begin{enumerate}
\item [(a)]Suppose that the operator $\tilde{Q}_{I}f(t)=\frac{I(t)}{t}\int
_{t}^{1/2}f(s)\frac{ds}{I(s)}$ is bounded on $\bar{X}.$ Then, for all $g\in
Lip(\Omega),$ we have
\[
\left\|  g-\int_{\Omega}gd\mu\right\|  _{Y}\preceq\left\|  \left(
g-\int_{\Omega}gd\mu\right)  _{\mu}^{\ast}(t)\frac{I(t)}{t}\right\|  _{\bar
{X}}\preceq\left\|  \left|  \nabla g\right|  \right\|  _{X}.
\]

\item[(b)] If $\overline{\alpha}_{X}<1$, or if the isoperimetric estimator $I$
satisfies (\ref{concon}), then, for all $g\in Lip(\Omega)$ we have,%
\begin{equation}
\left\|  g-\int_{\Omega}gd\mu\right\|  _{Y}\preceq\left\|  g-\int_{\Omega
}gd\mu\right\|  _{LS(X)}\preceq\left\|  \left|  \nabla g\right|  \right\|
_{X}. \label{perdida02}%
\end{equation}
\end{enumerate}
\end{theorem}

\begin{proof}
Let $g$ $\in Lip(\Omega)$. Write
\[
g_{\mu}^{\ast}(t)=\int_{t}^{1/2}\left(  -g_{\mu}^{\ast}\right)  ^{\prime
}(s)ds+g_{\mu}^{\ast}(1/2),\text{ }t\in(0,1/2].
\]
Thus,
\begin{align*}
\left\|  g\right\|  _{Y}  &  =\left\|  g_{\mu}^{\ast}\right\|  _{\bar{Y}%
}\preceq\left\|  g_{\mu}^{\ast}\chi_{\lbrack0,1/2]}\right\|  _{\bar{Y}}\\
&  \preceq\left\|  \int_{t}^{1/2}\left(  -g_{\mu}^{\ast}\right)  ^{\prime
}(s)ds\right\|  _{\bar{Y}}+g_{\mu}^{\ast}(1/2)\left\|  1\right\|  _{\bar{Y}}\\
&  \leq\left\|  \int_{t}^{1/2}\left(  -g_{\mu}^{\ast}\right)  ^{\prime
}(s)I(s)\frac{ds}{I(s)}\right\|  _{\bar{Y}}+2\left\|  1\right\|  _{\bar{Y}%
}\left\|  g\right\|  _{L^{1}}\\
&  \preceq\left\|  \left(  -g_{\mu}^{\ast}\right)  ^{\prime}(s)I(s)\right\|
_{\bar{X}}+\left\|  g\right\|  _{L_{1}}\text{ \ \ (by (\ref{har}))}\\
&  \preceq\left\|  \left|  \nabla g\right|  \right\|  _{X}\text{ }+\left\|
g\right\|  _{L^{1}}\text{(by (\ref{provadas})).}%
\end{align*}

Therefore,
\begin{align*}
\left\|  g-\int_{\Omega}gd\mu\right\|  _{Y}  &  \preceq\left\|  \left|  \nabla
g\right|  \right\|  _{X}\text{ }+\left\|  g-\int_{\Omega}gd\mu\right\|
_{L^{1}}\\
&  \preceq\left\|  \left|  \nabla g\right|  \right\|  _{X}+\left\|  \left|
\nabla g\right|  \right\|  _{L^{1}}\text{ (by (\ref{pr1}))}\\
&  \preceq\left\|  \left|  \nabla g\right|  \right\|  _{X}\text{ (by
(\ref{nuevadeli})).}%
\end{align*}

\textbf{Part (a)} It will be convenient to let $\bar{X}_{I}$ be the r.i. space
on $(0,1)$ defined by the condition
\[
\left\|  h\right\|  _{\bar{X}_{I}}=\left\|  h(t)\frac{I(t)}{t}\right\|
_{\bar{X}}<\infty.
\]
We start by proving that
\begin{equation}
\left\|  f\right\|  _{\bar{Y}}\preceq\left\|  f_{\mu}^{\ast}\right\|
_{\bar{X}_{I}}. \label{perdida}%
\end{equation}
Indeed, let $0<t<1/2$. From
\[
f_{\mu}^{\ast}(t)\ln2\leq\int_{t/2}^{t}f_{\mu}^{\ast}(s)\frac{ds}{s}\leq
\int_{t/2}^{1/2}f_{\mu}^{\ast}(s)\frac{I(s)}{s}\frac{ds}{I(s)},
\]
we see that for $t\in(0,1/2),$%
\[
f_{\mu}^{\ast}(t)\preceq\int_{t/2}^{1/2}f_{\mu}^{\ast}(s)\frac{I(s)}{s}%
\frac{ds}{I(s)}+f_{\mu}^{\ast}(1/2).
\]
Consequently,
\begin{align*}
\left\|  f_{\mu}^{\ast}(t)\chi_{(0,1/2)}(t)\right\|  _{\bar{Y}}  &
\preceq\left\|  \int_{t/2}^{1}\left(  f_{\mu}^{\ast}(s)\frac{I(s)}{s}\right)
\chi_{(0,1/2)}(s)\frac{ds}{I(s)}\right\|  _{\bar{Y}}+\left\|  f\right\|
_{L^{1}}\\
&  \leq2\left\|  Q_{I}\left(  f_{\mu}^{\ast}(s)\frac{I(s)}{s}\chi
_{(0,1/2)}(s)\right)  \right\|  _{\bar{Y}}+\left\|  f\right\|  _{L^{1}}\text{
\ (by (\ref{ccdd}))}\\
&  \preceq\left\|  f_{\mu}^{\ast}(t)\frac{I(t)}{t}\right\|  _{\bar{X}%
}+\left\|  f\right\|  _{L^{1}}\\
&  \preceq\left\|  f_{\mu}^{\ast}\right\|  _{\bar{X}_{I}},
\end{align*}
where in the last step we estimated $\left\|  f\right\|  _{L^{1}}$ as follows
\begin{align*}
\left\|  f\right\|  _{L^{1}}  &  =\int_{0}^{1}f_{\mu}^{\ast}(t)dt\leq2\int
_{0}^{1/2}f_{\mu}^{\ast}(t)dt\\
&  =\int_{0}^{1/2}f_{\mu}^{\ast}(t)\frac{I(t)}{t}\frac{t}{I(t)}dt\\
&  \leq\frac{2}{I(1/2)}\int_{0}^{1}f_{\mu}^{\ast}(t)\frac{I(t)}{t}dt\\
&  \preceq\left\|  f_{\mu}^{\ast}(t)\frac{I(t)}{t}\right\|  _{\bar{X}}\text{
(by (\ref{nuevadeli})).}%
\end{align*}
From the previous discussion we see that
\begin{align*}
\left\|  f\right\|  _{\bar{Y}}  &  \preceq\left\|  f_{\mu}^{\ast}%
(t)\chi_{(0,1/2)}(t)\right\|  _{\bar{Y}}\\
&  \preceq\left\|  f_{\mu}^{\ast}(t)\frac{I(t)}{t}\right\|  _{\bar{X}}\\
&  =\left\|  f_{\mu}^{\ast}\right\|  _{\bar{X}_{I}}.
\end{align*}
Now, we show that for all $f\in\bar{X},$ with supp$f\subset(0,1/2),$%
\[
\left\|  Q_{I}f\right\|  _{\bar{X}_{I}}\preceq\left\|  f\right\|  _{\bar{X}}.
\]
Indeed, this is equivalent to the boundedness of the operator $\tilde{Q}_{I}$:%
\begin{align*}
\left\|  Q_{I}f\right\|  _{\bar{X}_{I}}  &  =\left\|  \int_{t}^{1}%
f(s)\frac{ds}{I(s)}\right\|  _{\bar{X}_{I}}\\
&  =\left\|  \frac{I(t)}{t}\int_{t}^{1}f(s)\frac{ds}{I(s)}\right\|  _{\bar{X}%
}\\
&  =\left\|  \tilde{Q}_{I}f\right\|  _{\bar{X}}\\
&  \preceq\left\|  f\right\|  _{\bar{X}}.
\end{align*}
Consequently, by the first part of the theorem we have that for all $g\in
Lip(\Omega)$%
\begin{equation}
\left\|  \left(  g-\int_{\Omega}gd\mu\right)  _{\mu}^{\ast}(t)\frac{I(t)}%
{t}\right\|  _{\bar{X}}=\left\|  \left(  g-\int_{\Omega}gd\mu\right)  _{\mu
}^{\ast}(t)\right\|  _{\bar{X}_{I}}\preceq\left\|  \left|  \nabla g\right|
\right\|  _{X}. \label{PPPPP}%
\end{equation}
Finally, combining (\ref{PPPPP}) and (\ref{perdida}) we obtain
\begin{align*}
\left\|  g-\int_{\Omega}gd\mu\right\|  _{Y}  &  =\left\|  \left(
g-\int_{\Omega}gd\mu\right)  _{\mu}^{\ast}(t)\right\|  _{\bar{Y}}\\
&  \preceq\left\|  \left(  g-\int_{\Omega}gd\mu\right)  _{\mu}^{\ast}%
(t)\frac{I(t)}{t}\right\|  _{\bar{X}}\\
&  \leq\left\|  \left|  \nabla g\right|  \right\|  _{X}.
\end{align*}

\textbf{Part (b)} We first show that%
\begin{equation}
\left\|  f\right\|  _{Y}\preceq\left\|  f\right\|  _{LS(X)}+\left\|
f\right\|  _{L^{1}}. \label{harhar}%
\end{equation}
Since $\left(  f_{\mu}^{\ast\ast}\right)  ^{\prime}(t)=-\frac{1}{t}\left(
f_{\mu}^{\ast\ast}(t)-f_{\mu}^{\ast}(t)\right)  ,$ using the fundamental
theorem of Calculus yields
\[
f_{\mu}^{\ast\ast}(t)=\int_{t}^{1/2}\left(  f_{\mu}^{\ast\ast}(s)-f_{\mu
}^{\ast}(s)\right)  \frac{ds}{s}+f_{\mu}^{\ast\ast}(1/2),\text{ \ \ }%
0<t\leq1/2.
\]
Therefore,%
\begin{align*}
\left\|  f_{\mu}^{\ast}(t)\chi_{(0,1/2)}(t)\right\|  _{\bar{Y}}  &
\leq\left\|  \int_{t}^{1/2}\left(  f_{\mu}^{\ast\ast}(s)-f_{\mu}^{\ast
}(s)\right)  \frac{ds}{s}\right\|  _{\bar{Y}}+f_{\mu}^{\ast\ast}(1/2)\left\|
1\right\|  _{\bar{Y}}\\
&  \preceq\left\|  \int_{t}^{1}\frac{I(s)}{s}\left(  f_{\mu}^{\ast\ast
}(s)-f_{\mu}^{\ast}(s)\right)  \chi_{(0,1/2)}(s)\frac{ds}{I(s)}\right\|
_{\bar{Y}}+\left\|  f\right\|  _{L^{1}}\\
&  \preceq\left\|  (f_{\mu}^{\ast\ast}(t)-f_{\mu}^{\ast}(t))\chi
_{(0,1/2)}(t)\frac{I(t)}{t}\right\|  _{\bar{X}}+\left\|  f\right\|  _{L^{1}}\\
&  \preceq\left\|  (f_{\mu}^{\ast\ast}(t)-f_{\mu}^{\ast}(t))\frac{I(t)}%
{t}\right\|  _{\bar{X}}\text{ }+\left\|  f\right\|  _{L^{1}}.
\end{align*}
Consequently,
\begin{align*}
\left\|  f_{\mu}^{\ast}\right\|  _{\bar{Y}}  &  \preceq\left\|  f_{\mu}^{\ast
}(t)\chi_{(0,1/2)}(t)\right\|  _{\bar{Y}}\\
&  \preceq\left\|  (f_{\mu}^{\ast\ast}(t)-f_{\mu}^{\ast}(t))\frac{I(t)}%
{t}\right\|  _{\bar{X}}+\left\|  f\right\|  _{L^{1}}\\
&  =\left\|  f\right\|  _{LS(X)}+\left\|  f\right\|  _{L^{1}}.
\end{align*}
Assume that $\overline{\alpha}_{X}<1$. We are going to prove (\ref{perdida02}%
)$.$ Let $g\in Lip({\Omega)}$. Applying successively (\ref{harhar}),
(\ref{corres}), (\ref{pr1}), (\ref{nuevadeli}), and the fact that $P$ is a
bounded operator on $\bar{X},$ we have
\begin{align*}
\left\|  g-\int_{\Omega}gd\mu\right\|  _{Y}  &  =\left\|  \left(
g-\int_{\Omega}gd\mu\right)  _{\mu}^{\ast}\right\|  _{\bar{Y}}\\
&  \preceq\left\|  g-\int_{\Omega}gd\mu\right\|  _{LS(X)}+\left\|
g-\int_{\Omega}gd\mu\right\|  _{L^{1}}\\
&  \preceq\left\|  P\left(  \left|  \nabla\left(  g-\int_{\Omega}gd\mu\right)
\right|  _{\mu}^{\ast}\right)  \right\|  _{\bar{X}}+\left\|  \left|  \nabla
g\right|  \right\|  _{L^{1}}\\
&  \preceq\left\|  P\left(  \left|  \nabla g\right|  _{\mu}^{\ast}\right)
\right\|  _{\bar{X}}+\left\|  \left|  \nabla g\right|  \right\|  _{\bar{X}}\\
&  \preceq\left\|  \left|  \nabla g\right|  \right\|  _{X}.
\end{align*}
Finally, suppose that $I$ satisfies (\ref{concon}). Then, by Remark
\ref{preferida},%
\[
\left\|  g\right\|  _{LS(X)}\preceq\left\|  \left|  \nabla g\right|  \right\|
_{X},
\]
as we wished to show.
\end{proof}

\subsection{Poincar\'{e} inequalities for the model cases}

In this section we show the equivalence of Poincar\'{e} inequalities and the
boundedness of the isoperimetric Hardy operator $Q_{I}$ for all the model
cases considered in the previous section.

Let $(\Gamma,\varrho)$ denote any of the following probability metric spaces:

\begin{enumerate}
\item  Log concave measures $(\mathbb{R}^{n},d\mu^{\Phi\otimes n})$ (cf.
Section \ref{secc:logconcave}).

\item  The $n-$sphere $(\mathbb{S}^{n},d,\sigma_{n})$ (cf. Section
\ref{secc:sphere}).

\item  An $n_{0}$-dimensional Riemannian Model manifold $\left(
M_{0},d\right)  $ endowed with a model measure $\mu_{0}$. (cf. Section
\ref{secc:ros}).
\end{enumerate}

\begin{theorem}
\label{optimal} Consider the probability space $(\Gamma,\varrho).$ Let
$X=X(\Gamma)$, $Y=Y(\Gamma)$ be r.i. spaces. Then, the following statements
are equivalent

\begin{enumerate}
\item
\begin{equation}
\left\|  f-\int_{\Gamma}fd\varrho\right\|  _{Y}\preceq\left\|  \left|  \nabla
f\right|  \right\|  _{X},\text{ for all }f\in Lip(\Gamma). \label{poin}%
\end{equation}

\item
\begin{equation}
\left\|  \int_{t}^{1}f(s)\frac{ds}{I_{\varrho}(s)}\right\|  _{\bar{Y}}%
\preceq\left\|  f\right\|  _{\bar{X}},\text{ \ \ for all positive f}\in\bar
{X},\text{ with }supp(f)\subset(0,1/2). \label{poinpoin}%
\end{equation}
\end{enumerate}
\end{theorem}

\begin{proof}
$\mathbf{(2)\rightarrow(1)}$ was proved in Theorem \ref{opti00}.

We naturally divide the proof of the implications $\mathbf{(1)\rightarrow(2)}$
in three cases as follows:

\textbf{Case a)} Log concave measures.

Given a\ positive measurable function $f\ $with $suppf$ $\subset(0,1/2),$
consider
\[
F(t)=\int_{t}^{1}f(s)\frac{ds}{I_{\mu^{\Phi}}(s)},\text{ \ \ \ }t\in(0,1),
\]
and define
\[
u(x)=F(H(x_{1})),\text{ \ \ \ \ \ }x\in\mathbb{R}^{n}.
\]
Then,%
\[
\left|  \nabla u(x)\right|  =\left|  \frac{\partial}{\partial x_{1}%
}u(x)\right|  =\left|  -f(H(x_{1}))\frac{H^{\prime}(x_{1})}{I_{\mu^{\Phi}%
}(H(x_{1}))}\right|  =f(H(x_{1})).
\]
Let $N$ be a Young's function and let $s=H(x_{1})$. Then,
\begin{align*}
\int_{\mathbb{R}^{n}}N(f(H(x_{1})))d\mu &  =\int_{\mathbb{R}}N(f(H(x_{1}%
)))d\mu^{\Phi}\\
&  =\int_{0}^{1}N(f(s))ds.
\end{align*}
Therefore,
\begin{equation}
\left|  \nabla u\right|  _{\mu}^{\ast}(t)=f^{\ast}(t), \label{grany}%
\end{equation}
and
\begin{equation}
u_{\mu}^{\ast}(t)=\int_{t}^{1}f(s)\frac{ds}{I_{\mu^{\Phi}}(s)}. \label{grany1}%
\end{equation}
By Lemma \ref{media}, (\ref{poin}) is equivalent to
\[
\left\|  u-m_{e}\right\|  _{Y}\preceq\left\|  \left|  \nabla u\right|
\right\|  _{X},
\]
where $m_{e}$ is a median of $u.$ Since$\ \mu\left\{  u=0\right\}  \geq1/2,$
it follows that $0$ is a median of $u$. Consequently,%
\begin{equation}
\left\|  u\right\|  _{Y}\preceq\left\|  \left|  \nabla u\right|  \right\|
_{X}. \label{grany2}%
\end{equation}
From (\ref{grany}) and (\ref{grany1}) it follows that
\[
\left\|  u\right\|  _{Y}=\left\|  u_{\mu}^{\ast}\right\|  _{\bar{Y}}\text{ and
}\left\|  \left|  \nabla u\right|  \right\|  _{X}=\left\|  \left|  \nabla
u\right|  _{\mu}^{\ast}\right\|  _{\bar{X}}=\left\|  f\right\|  _{\bar{X}},
\]
therefore, inserting this information back in (\ref{grany2}), and since (see
Section \ref{tenso})
\[
I_{\varrho}\simeq I_{\mu^{\Phi}}%
\]
we obtain (\ref{poinpoin}).

\textbf{Case b)} The $n-$sphere $(\mathbb{S}^{n},d,\sigma_{n})$.

The argument given in case a) can be repeated verbatim with the following
changes: Given a\ positive measurable function $f$ $\ $with $suppf$
$\subset(0,1/2),$ let
\[
F(t)=\int_{t}^{1}f(s)\frac{ds}{I_{\sigma_{n}}(s)},\text{ \ \ \ }t\in(0,1),
\]
and define $u$ (in spherical coordinates) by
\[
u(\theta_{1},.....,\theta_{n})=F(\Phi(\theta_{1})),\text{ \ \ \ \ \ }%
(\theta_{1},.....,\theta_{n})\in\mathbb{S}^{n}.
\]

\textbf{Case c) }An $n_{0}$-dimensional Riemannian Model manifold $\left(
M_{0},d\right)  $ endowed with a model measure $\mu_{0}.$

This case was proved in \cite{mmmazya}, but we include a brief sketch of its
proof for the sake of completeness. As in Section \ref{secc:ros}, we consider%
\[%
\begin{array}
[c]{c}%
p:M_{0}\rightarrow\lbrack0,1]\\
x\in\partial D^{t}\rightarrow t.
\end{array}
\]
Then (see \cite{mmmazya} for the details) $p\in Lip(M_{0})$ with $\left|
\nabla p(x)\right|  =I_{\mu_{0}}(p(x))$ and the map $p:\left(  M_{0},\mu
_{0}\right)  \rightarrow\left(  \lbrack0,1],ds\right)  $ is a
measure-preserving transformation.

Let $f\in\bar{X}$ be a positive function, with supp$f\subset(0,1/2),$ and
define
\[
F(x)=\int_{p(x)}^{1}f(s)\frac{ds}{I_{\mu_{0}}(s)}.
\]
$F\in Lip(M_{0}),$ and
\[
\left|  \nabla F(x)\right|  =f(p(x))\frac{1}{I_{\mu_{0}}(p(x))}\left|  \nabla
p(x)\right|  =f(p(x)).
\]
Moreover, since $p$ is a measure-preserving transformation, we have
\[
\left|  F\right|  _{\mu_{0}}^{\ast}(s)=\int_{t}^{1}f(s)\frac{ds}{I_{\mu_{0}%
}(s)}\text{ \ \ and \ \ }\left|  \nabla F\right|  _{\mu_{0}}^{\ast}%
(s)=f^{\ast}(s)\text{.}%
\]
Now since$\ \mu_{0}\left\{  F=0\right\}  \geq1/2,$ $0$ is a median of $F$.
Therefore, from
\[
\left\|  F-0\right\|  _{Y}\preceq\left\|  \left|  \nabla F\right|  \right\|
_{X},
\]
we obtain
\[
\left\|  \int_{t}^{1}f(s)\frac{ds}{I_{\mu_{0}}(s)}\right\|  _{\bar{Y}}%
\preceq\left\|  f\right\|  _{\bar{X}}.
\]
\end{proof}

\begin{example}
\label{exex}Let $\alpha\geq0,$ \ $p\in\lbrack1,2]$, $\gamma=\exp
(2\alpha/(2-p)),$ and consider the family of log concave measures
\[
\mu_{p,\alpha}=Z_{p,\alpha}^{-1}\exp\left(  -\left|  x\right|  ^{p}%
(\log(\gamma+\left|  x\right|  )^{\alpha}\right)  dx.
\]
Using estimate (\ref{qq}) (see \cite{BCR} and \cite{BCR1}) we get
\begin{equation}
I_{\mu_{p,\alpha}^{\otimes n}}(s)\simeq s\left(  \log\frac{1}{s}\right)
^{1-\frac{1}{p}}\left(  \log\log\left(  e+\frac{1}{s}\right)  \right)
^{\frac{\alpha}{p}}=s\beta_{p,\alpha}(s),\text{ \ \ \ \ \ }0<s\leq1/2,
\label{asim}%
\end{equation}
moreover the constants that appear in equivalence (\ref{asim}) are independent
of $n.$ The corresponding operators $Q_{\mu_{p,\alpha}^{\otimes n}}$ and
$\tilde{Q}_{\mu_{p,\alpha}^{\otimes n}}$ associated with $\mu_{p,\alpha
}^{\otimes n}$ are given by%
\[
Q_{I_{\mu_{p,\alpha}^{\otimes n}}}f(t)\simeq\int_{t}^{1/2}f(s)\frac{ds}%
{s\beta_{p,\alpha}(s)}\text{ \ and \ }\tilde{Q}_{I_{\mu_{p,\alpha}^{\otimes
n}}}f(t)\simeq\beta_{p,\alpha}(t)\int_{t}^{1/2}f(s)\frac{ds}{s\beta_{p,\alpha
}(s)}.
\]
Given $X$ a r.i. space such that $\underline{\alpha}_{X}>0,$ then the operator
\ $\tilde{Q}_{I_{\mu_{p,\alpha}^{\otimes n}}}$ is bounded on $X.$ Indeed, pick
$\underline{\alpha}_{X}>a>0,$ then since $t^{a}\beta_{p,\alpha}(t)$ is
increasing near zero, we get
\[
\tilde{Q}_{I_{\mu_{p,\alpha}^{\otimes n}}}f(t)\simeq\frac{t^{a}\beta
_{p,\alpha}(t)}{t^{a}}\int_{t}^{1/2}f(s)\frac{ds}{s\beta_{p,\alpha}(s)}%
\preceq\frac{1}{t^{a}}\int_{t}^{1/2}s^{a}f(s)\frac{ds}{s}=Q_{a}f(t).
\]
We conclude noting that $Q_{a}$ is bounded on $X$ on account of the fact that
$\underline{\alpha}_{X}>a$ (see Remark \ref{alcance}).
\end{example}

\begin{example}
In the case of the sphere, the operators $Q_{I_{\sigma_{n}}}$ and $\tilde
{Q}_{I_{\sigma_{n}}}$ associated with $\sigma_{n}$ are given by%
\[
Q_{I_{\sigma_{n}}}f(t)\simeq\int_{t}^{1/2}f(s)s^{1/n}\frac{ds}{s}\text{ \ and
\ }\tilde{Q}_{I_{\sigma_{n}}}f(t)\simeq t^{1-1/n}\int_{t}^{1/2}f(s)s^{1/n}%
\frac{ds}{s}.
\]
Given $X$ a r.i. the operator \ $\tilde{Q}_{I_{\sigma_{n}}}$ is bounded on $X$
if and only if $\underline{\alpha}_{X}>1/n.$
\end{example}

\section{Poincar\'{e} Inequalities and Cheeger's
inequality\label{secc:emanuel}}

\subsection{Poincar\'{e} inequalities and Hardy operators\label{secc:hardy}}

The study of the model cases suggests the possibility of characterizing sharp
Poincar\'{e} inequalities in terms of the boundedness of the Hardy operators
$Q_{I}.$ However, for general metric spaces this is not possible. In fact (cf.
\cite{mmmazya} for the details), for a given $0<\beta<1/2,$ consider%
\[
I(s)=s^{1-\beta},\text{ \ \ \ }0\leq s\leq1/2.
\]
Let $\Omega$ be a $2(1-\beta)-$John domain on $\mathbb{R}^{2},$ $(\left|
\Omega\right|  =1)$. The isoperimetric profile $I_{\Omega}(s)$ of $\Omega$
satisfies (cf. \cite{HK})
\[
I_{\Omega}(s)\simeq I(s),\text{ \ \ \ }0\leq s\leq1/2,
\]
and (cf. \cite{KM1})
\[
\left\|  g-\int_{\Omega}g\right\|  _{L^{\frac{4}{1-2\beta}}}\preceq\left\|
\left|  \nabla g\right|  \right\|  _{L^{2}}.
\]
However, the operator%
\[
Q_{I_{\Omega}}f(t)=\int_{t}^{1/2}f(u)\frac{du}{I_{\Omega}(u)}%
\]
is not bounded from $L^{2}\ $to $L^{\frac{4}{1-2\beta}}.$ In fact, the extra
properties required on the metric spaces are not related with the form of the
isoperimetric profile. Indeed, it is possible to build a compact surface of
revolution $M$ such that there exists a constant $c$ depending only of $I$
such that%
\[
cI(s)\leq I_{M}(s)\leq I(s),\text{ \ \ \ }0\leq s\leq1/2,
\]
and, such that for any pair of r.i. spaces $X,Y$ on $M,$ the Poincar\'{e}
inequality
\[
\left\|  g-\int_{M}gdVol_{M}\right\|  _{Y}\preceq\left\|  \left|  \nabla
g\right|  \right\|  _{X},\ \ g\in Lip(M).
\]
is equivalent to
\[
Q_{I_{M}}:\bar{X}\rightarrow\bar{Y}\text{ is bounded.}%
\]
The present discussion motivated the developments in the next sections.

\subsection{Isoperimetric Hardy type\label{secc:isohar}}

We single out probability metric spaces that are suitable for our analysis.

\begin{definition}
\label{def:isohar}We shall say that a probability metric space $(\Omega
,d,\mu)$ is of isoperimetric Hardy type if for any given isoperimetric
estimator $I,$ the following are equivalent for all r.i. spaces $X=X(\Omega)$,
$Y=Y(\Omega)$.

\begin{enumerate}
\item  There exists a constant $c=c(X,Y)$ such that for all $f\in
Lip(\Omega)$
\[
\left\|  f-\int_{\Omega}fd\mu\right\|  _{Y}\leq c\left\|  \left|  \nabla
f\right|  \right\|  _{X}.
\]

\item  There exists a constant $c_{1}=c_{1}(X,Y)>0$ such that for all positive
functions $f\in\bar{X},$ with $supp(f)\subset(0,1/2)$ we have
\[
\left\|  Q_{I}f\right\|  _{\bar{Y}}\leq c_{1}\left\|  f\right\|  _{\bar{X}},
\]
where $Q_{I}$ is the isoperimetric Hardy operator
\begin{equation}
Q_{I}f(t)=\int_{t}^{1}f(s)\frac{ds}{I(s)}. \label{olvidada}%
\end{equation}
\end{enumerate}
\end{definition}

\begin{example}
\label{examarkao}By Theorem \ref{optimal} all the model spaces are of Hardy
isoperimetric type.
\end{example}

Our first application was motivated by the remarkable recent work of E. Milman
(cf. \cite{MiE}, \cite{mie2}, \cite{miemie}) on the equivalence of Cheeger's
inequality, Poincar\'{e}'s inequality and concentration, under suitable
convexity conditions. More precisely, E. Milman has shown that\footnote{We
refer to E. Milman's papers for an account of the history of the problem.}

\begin{theorem}
\label{teoem1}(E. Milman) Let $(\Omega,d,\mu)$ be a space satisfying E.
Milman's convexity conditions (cf. Example \ref{ej:emanu} above). Then
following statements are equivalent

\noindent(E1) Cheeger's inequality: there exists a positive constant $C$ such
that%
\[
I_{(\Omega,d,\mu)}\geq Ct,\ \ \ t\in(0,1/2].
\]
(E2) Poincar\'{e}'s inequality: there exists a positive constant $P$ such that
for all $f\in Lip(\Omega),$%
\[
\left\|  f-m_{e}\right\|  _{L^{2}(\Omega)}\leq P\left\|  \left|  \nabla
f\right|  \right\|  _{L^{2}(\Omega)}.
\]
(E3) Exponential concentration: there exist positive constants $c_{1},c_{2}$
such that for all $f\in Lip(\Omega)$ with $\left\|  f\right\|  _{Lip(\Omega
)}\leq1,$%
\[
\mu\{\left|  f-m_{e}\right|  >t\}\leq c_{1}e^{-c_{2}t},\text{ \ }t\in(0,1).
\]
(E4) First moment inequality: there exists a positive constant $F$ such that
for all $f\in Lip(\Omega)$ with $\left\|  f\right\|  _{Lip(\Omega)}\leq1,$%
\[
\left\|  f-m_{e}\right\|  _{L^{1}(\Omega)}\leq F.
\]
\end{theorem}

Moreover, E. Milman also shows

\begin{theorem}
\label{teoem2}Let $(\Omega,d,\mu)$ be a space satisfying E. Milman's convexity
conditions. Let $1\leq q<\infty,$ and let $N$ be a Young's function such that
$\frac{N(t)^{1/q}}{t}$ is non-decreasing, and there exists $\alpha>\max
\{\frac{1}{q}-\frac{1}{2},0\}$ such that $\frac{N(t^{\alpha})}{t}$
non-increasing. Then, the following statements are equivalent:

\noindent(E5) $(L_{N},L^{q})$ Poincar\'{e} inequality holds: there exists a
positive constant $P$ such that for all $f\in Lip(\Omega)$%
\[
\left\|  f-m_{e}\right\|  _{L_{N}(\Omega)}\leq P\left\|  \left|  \nabla
f\right|  \right\|  _{L^{q}(\Omega)}.
\]
(E6) Any isoperimetric profile estimator $I$ satisfies: there exists a
constant $c>0$ such that%
\[
I(t)\geq c\frac{t^{1-1/q}}{N^{-1}(1/t)},\text{ \ }t\in(0,1/2].
\]
\end{theorem}

Milman approaches these results using a variety of different tools including
the semigroup approach of Ledoux (\cite{led1}, \cite{led2}, \cite{led3}).

We shall show a simple proof that these equivalences hold for probability
metric spaces of Hardy type. On the other hand at this writing the precise
connection between isoperimetric Hardy type and convexity remains an open problem.

\begin{theorem}
\label{teoema1}Suppose that $(\Omega,d,\mu)$ is a metric probability space of
isoperimetric Hardy type. Then
\[
(E1)\Leftrightarrow(E2)\Leftrightarrow(E3)\Leftrightarrow(E4).
\]
\end{theorem}

\begin{proof}
Suppose that Cheeger's inequality (E1) holds, $I(s)\succeq s,$ $s\in(0,1/2).$
Therefore, for all $f\geq0,$ with $supp(f)\subset(0,1/2),$ we have%
\begin{equation}
Q_{I}f(t)=\int_{t}^{1}f(s)\frac{ds}{I(s)}\preceq Qf(t)=\int_{t}^{1}%
f(s)\frac{ds}{s}. \label{valida}%
\end{equation}
In particular, since $Q:L^{2}(0,1)\rightarrow L^{2}(0,1),$ we see that%
\[
\left\|  Q_{I}f\right\|  _{L^{2}}\leq C\left\|  f\right\|  _{L^{2}},\text{ for
all }f\text{ }\geq0,\text{ such that }supp(f)\subset(0,1/2).
\]
Consequently, by the isoperimetric Hardy property, the $(L^{2},L^{2})$
Poincar\'{e} inequality (E2) holds. Conversely, if the $(L^{2},L^{2})$
Poincar\'{e} inequality holds, then
\[
\left\|  Q_{I}f\right\|  _{L^{2}}\leq C\left\|  f\right\|  _{L^{2}},\text{for
all }f\text{ such that }supp(f)\subset(0,1/2).
\]
Moreover, since $L^{2}\subset L(2,\infty),$ we have%
\[
\left\|  Q_{I}f\right\|  _{L(2,\infty)}\leq C\left\|  f\right\|  _{L^{2}%
}\text{ for all }f\geq0\text{ such that }supp(f)\subset(0,1/2).
\]
Let $f=\chi_{(0,r)},$ with $r\leq1/2.$ Then, the previous inequality readily
gives%
\[
\sup_{t}t^{1/2}\int_{t}^{r}\frac{ds}{I(s)}\leq Cr^{1/2},
\]
and, since $I(t)$ increases on $(0,1/2),$ we get%
\[
\frac{1}{I(r)}\sup_{t}t^{1/2}(r-t)\preceq Cr^{1/2}.
\]
Moreover, since on the other hand%
\[
\sup_{t<r}t^{1/2}(r-t)\geq\left(  \frac{r}{2}\right)  ^{1/2}\frac{r}{2}%
\]
we see that%
\[
I(t)\succeq t,\text{ \ }t\in(0,1/2].
\]

It is also elementary to see that the operator $Q$ defined above is a bounded
operator $Q:L^{\infty}\mapsto\exp L.$ Indeed, using an equivalent norm for
$\exp L$ (cf. \cite{jm}) we compute%
\[
\left\|  \int_{t}^{1}f(s)\frac{ds}{s}\right\|  _{\exp(L)}=\sup_{0<t<1}%
\frac{\int_{t}^{1}f(s)\frac{ds}{s}}{1+\log\frac{1}{t}}\leq\left\|  f\right\|
_{L^{\infty}}.
\]
Therefore, if (E1) holds then by (\ref{valida}),%
\[
Q_{I}:L^{\infty}\rightarrow\exp(L),
\]
and therefore, by the isoperimetric Hardy property, we see that for all $f\in
Lip(\Omega)$ we have%
\begin{equation}
\left\|  f-m_{e}\right\|  _{\exp(L)}\preceq\left\|  \left|  \nabla f\right|
\right\|  _{L^{\infty}}. \label{valida1}%
\end{equation}
In other words, the exponential concentration inequality (E3) holds.
Conversely, suppose that (\ref{valida1}) holds. Then, by the isoperimetric
Hardy property, we have,%
\begin{equation}
\sup_{t}\frac{\int_{t}^{1/2}f(s)\frac{ds}{I(s)}}{1+\log\frac{1}{t}}%
\preceq\left\|  f\right\|  _{L^{\infty}}. \label{valida3}%
\end{equation}
Insert the function $f(s)=\chi_{(0,1/2)}(s)\in L^{\infty}$ in (\ref{valida3});
then, using the fact that $s/I(s)$ increases, we see that for all
$t\in(0,1/2)$ we have
\begin{align*}
c  &  \succeq\sup_{t<1/2}\frac{\int_{t}^{1/2}\frac{s}{s}\frac{ds}{I(s)}%
}{1+\log\frac{1}{t}}\\
&  \succeq\frac{t}{I(t)}\frac{\int_{t}^{1/2}\frac{ds}{s}}{1+\log\frac{1}{t}}\\
&  \succeq\frac{t}{I(t)}\frac{\log\frac{1}{t}+\log\frac{1}{2}}{1+\log\frac
{1}{t}}\\
&  \succeq\frac{t}{I(t)}.
\end{align*}
Therefore Cheeger's inequality (E1) holds. Finally, (E3) combined with the
trivial embedding%
\[
\left\|  f-m_{e}\right\|  _{L^{1}}\leq c\left\|  f-m_{e}\right\|  _{\exp(L)}%
\]
implies
\[
\left\|  f-m_{e}\right\|  _{L^{1}}\preceq\left\|  \nabla f\right\|
_{L^{\infty}}.
\]
Therefore (E4) holds. Conversely, if (E4) holds then%
\[
\left\|  Q_{I}f\right\|  _{L^{1}}\leq C\left\|  f\right\|  _{L^{\infty}}\text{
for all }f\geq0\text{ such that }supp(f)\subset(0,1/2).
\]
A familiar calculation using $f=\chi_{(0,r)},$ with $r\leq1/2,$ gives%
\[
I(t)\succeq t^{2},\text{ }t\in(0,1/2].
\]
However (here we use an argument by E. Milman \cite{MiE}), we know that
$I(t)/t$ is decreasing and $I(t)$ is symmetric about $1/2$ so by a convexity
argument we can deduce that
\[
I(t)\succeq t,\text{ }t\in(0,1/2]
\]
concluding the proof.
\end{proof}

We shall now consider the equivalence between $(E5)$ and $(E6)$ in the setting
of metric probability spaces. We start the discussion observing that given a
r.i. space it is, in general, not possible to improve on (\ref{tango}) unless
we have more information about $X.$ On the other hand, when dealing with
Orlicz spaces, and we assume, moreover, some extra growth properties on the
Young's functions we can improve upon (\ref{tango}). More specifically,
suppose that $N$ is a Young's function such that$\frac{N(t)}{t^{q}}$ is
increasing, then%
\begin{equation}
\left\|  f\right\|  _{L_{N}}\preceq\left\|  f\right\|  _{\Lambda(\phi_{L_{N}%
},q)}=\left\{  \int_{0}^{1}\left[  f^{\ast}(s)\phi_{L_{N}}(s)\right]
^{q}\frac{ds}{s}\right\}  ^{1/q}, \label{merida}%
\end{equation}
while the opposite inequality holds if $\frac{N(t)}{t^{q}}$ decreases (cf.
\cite[pag 43]{merida}).

\begin{theorem}
Suppose that $(\Omega,d,\mu)$ is a metric probability space of isoperimetric
Hardy type. Let $1\leq q<\infty,$ and let $N$ be a Young's function such that
$\frac{N(t)^{1/q}}{t}$ is non-decreasing, and there exists $\alpha>\max
\{\frac{1}{q}-\frac{1}{2},0\}$ such that $\frac{N(t^{\alpha})}{t}$
non-increasing. Then $(E5)\Leftrightarrow(E6).$ In fact, $(E6)\Rightarrow(E5)$
is true without the assumption that $(\Omega,d,\mu)$ is of isoperimetric Hardy type.
\end{theorem}

\begin{proof}
If (E5) holds then, in view of (\ref{tango}), and the fact that $\Lambda
(L^{q})=L(q,1)$, we have
\[
\left\|  Q_{I}f\right\|  _{M(L_{N}(\Omega))}\preceq\left\|  f\right\|
_{L(q,1)}.
\]
Therefore, there exists a constant $C>0$ such that for $f=\chi_{(0,r)},$
$0<r<1/2,$ we have%
\[
\sup_{t<r}\left\{  \phi_{L_{N}}(t)\int_{t}^{r}\frac{ds}{I(s)}\right\}  \leq
Cr^{1/q}.
\]
Thus,%
\begin{align*}
\sup_{t<r}\phi_{L_{N}}(t)\frac{1}{I(r)}(r-t)  &  \geq\frac{1}{2}\phi_{L_{N}%
}(r/2)\frac{r}{I(r)}\\
&  \geq\frac{1}{4}\phi_{L_{N}}(r)\frac{r}{I(r)}\text{ (since }\phi_{L_{N}%
}(t)/t\text{ decreases).}%
\end{align*}
Summarizing, we have%
\[
I(r)\succeq r^{1-1/q}\phi_{L_{N}}(r),\;0<r<1/2.
\]
Consequently, recalling (\ref{quetiempos aquellos}) we obtain (E6).

Suppose now that (E6) holds. We will show below that%
\begin{equation}
\left\|  Q_{I}f\right\|  _{\Lambda(\phi_{L_{N}},q)}\preceq\left\|  f\right\|
_{L^{q}}. \label{ahora}%
\end{equation}
This given, and in view of (\ref{merida}), we see that
\[
\left\|  Q_{I}f\right\|  _{L_{N}}\preceq\left\|  f\right\|  _{L^{q}}.
\]
Therefore (E5) follows by the isoperimetric Hardy property. To prove
(\ref{ahora}) we use (E6) in order to estimate $Q_{I}$ by%
\[
Q_{I}f(t)\preceq\int_{t}^{1/2}\frac{f(s)s^{1/q-1}}{\phi_{L_{N}}(s)}s\frac
{ds}{s}\leq Q\left(  \frac{f(s)s^{1/q-1}}{\phi_{L_{N}}(s)}s\right)  (t).
\]
Thus, since $Q\left(  \frac{f(s)s^{1/q-1}}{\phi_{L_{N}}(s)}s\right)  (t)$ is
decreasing, using a suitable version of Hardy's inequality (cf. (\ref{cota})
below) we get%
\begin{align*}
\left\|  Q_{I}f\right\|  _{\Lambda(\phi_{L_{N}},q)}  &  \preceq\left\{
\int_{0}^{1}\left(  \int_{t}^{1}\frac{f(s)s^{1/q-1}}{\phi_{L_{N}}(s)}%
s\frac{ds}{s}\right)  ^{q}\left(  \phi_{L_{N}}(t)\right)  ^{q}\frac{dt}%
{t}\right\}  ^{1/q}\\
&  \preceq\left\{  \int_{0}^{1}\left(  \frac{f(t)t^{1/q}}{\phi_{L_{N}}%
(t)}t\frac{1}{t}\right)  ^{q}\left(  \phi_{L_{N}}(t)\right)  ^{q}\frac{dt}%
{t}\right\}  ^{1/q}\\
&  =\left\|  f\right\|  _{L^{q}},
\end{align*}
as we wished to show. To justify the application of Hardy's inequality we need
to verify (see \cite[Page 45]{maz'yabook}) that%
\begin{equation}
\sup_{0<r<1}\left(  \int_{0}^{r}\left(  \phi_{L_{N}}(t)\right)  ^{q}\frac
{dt}{t}\right)  ^{1/q}\left(  \int_{r}^{1}\left(  \frac{\left(  \phi_{L_{N}%
}(t)\right)  ^{q}}{t}\right)  ^{\frac{-1}{q-1}}\frac{dt}{t^{\frac{q}{q-1}}%
}\right)  ^{\frac{q-1}{q}}\leq c. \label{cota}%
\end{equation}
To this end observe that, under our current assumptions on the growth of $N,$
we have%
\[
\frac{N(t)^{1/q}}{t}\text{ increasing}\Rightarrow\frac{\lbrack\phi_{L_{N}%
}(t)]^{q}}{t}\text{decreasing,}%
\]%
\[
\frac{N(t^{\alpha})}{t}\text{ decreasing}\Rightarrow\frac{\left(  \phi_{L_{N}%
}(t)\right)  ^{^{1/\alpha}}}{t}\text{ increasing}\Rightarrow\frac{\phi_{L_{N}%
}(t)}{t^{\alpha}}\text{ increasing.}%
\]
Therefore,
\begin{align}
\frac{1}{r}\int_{0}^{r}\left(  \phi_{L_{N}}(t)\right)  ^{q}\frac{dt}{t}  &
=\frac{1}{r}\int_{0}^{r}\left(  \phi_{L_{N}}(t)\right)  ^{q-1}\frac
{\phi_{L_{N}}(t)}{t^{\alpha}}\frac{t^{\alpha}dt}{t}\nonumber\\
&  \leq\frac{\phi_{L_{N}}(r)}{r^{\alpha}}\left(  \phi_{L_{N}}(t)\right)
^{q-1}\frac{1}{r}\int_{0}^{r}\frac{t^{\alpha}dt}{t}\nonumber\\
&  =\frac{\phi_{L_{N}}(r)}{r^{\alpha}}\left(  \phi_{L_{N}}(t)\right)
^{q-1}\frac{1}{r}\frac{r^{\alpha}}{\alpha}\nonumber\\
&  =\frac{1}{\alpha}\frac{\left(  \phi_{L_{N}}(r)\right)  ^{q}}{r}. \label{ok}%
\end{align}
To estimate the second integral in (\ref{cota}) let $w(s)=\frac{\left(
\phi_{L_{N}}(t)\right)  ^{q}}{t},$ then
\begin{align*}
\int_{r}^{1}\left(  w(t)\right)  ^{\frac{-1}{q-1}}\frac{dt}{t^{\frac{q}{q-1}%
}}  &  =\int_{r}^{1}\frac{w(t)}{\left(  tw(t)\right)  ^{\frac{q}{q-1}}}dt\\
&  \leq\frac{1}{\alpha}\int_{r}^{1}\frac{w(t)}{\left(  \int_{0}^{t}%
w(s)ds\right)  ^{\frac{q}{q-1}}}dt\text{ \ \ (by (\ref{ok}))}\\
&  \leq\frac{1}{\alpha}\frac{1}{\left(  \int_{0}^{r}w(s)ds\right)  ^{\frac
{q}{q-1}}}\int_{r}^{1}w(t)dt\\
&  =\frac{1}{\alpha}\left(  \int_{0}^{r}w(s)ds\right)  ^{\frac{-1}{q-1}}.
\end{align*}
Thus,%
\[
\left(  \int_{0}^{r}\left(  \phi_{L_{N}}(t)\right)  ^{q}\frac{dt}{t}\right)
^{1/q}\left(  \int_{r}^{1}\left(  \frac{\left(  \phi_{L_{N}}(t)\right)  ^{q}%
}{t}\right)  ^{\frac{-1}{q-1}}\frac{dt}{t^{\frac{q}{q-1}}}\right)
^{\frac{q-1}{q}}\leq\frac{1}{\alpha},
\]
and (\ref{cota}) holds.
\end{proof}

\begin{remark}
In the particular case when $L_{N}(\Omega)=L^{p}$ ($p\geq q),$ then we have
$\Lambda(\phi_{L_{N}},q)=L(p,q),$ and therefore we obtain%
\[
\left\|  f-\int_{\Omega}fd\mu\right\|  _{L(p,\infty)}\preceq\left\|  \left|
\nabla f\right|  \right\|  _{L^{q}}\Rightarrow\left\|  f-\int_{\Omega}%
fd\mu\right\|  _{L(p,q)}\preceq\left\|  \left|  \nabla f\right|  \right\|
_{L^{q}}.
\]
For more on this type of self improvement for Poincar\'{e} inequalities see
\cite{mmpote}.
\end{remark}

\begin{remark}
\label{remarkao}The fact that Cheeger's inequality implies concentration also
follows readily from (\ref{rea}). To see this observe that if $I(t)\succeq t,$
and $f$ is $1-Lip(\Omega)$ then from (\ref{rea}) we get%
\[
f^{\ast\ast}(t)-f^{\ast}(t)\preceq c,
\]
in other words $f\in L(\infty,\infty),$ the weak class of Bennett, De Vore and
Sharpley \cite{bds}. Since it is known (cf. \cite{BS}) that $L(\infty
,\infty)\subset e^{L}$ (cf. also \cite{mmjfa} for more general results) we see
that Cheeger's inequality indeed implies
\[
f\in Lip(\Omega)\Rightarrow f\in e^{L},
\]
i.e. Cheeger's inequality $\Rightarrow$concentration.
\end{remark}

\section{Transference Principle\label{secc::transference}}

A very useful property of symmetrization methods is to reduce complicated
problems to simpler model problems where symmetry can be used to find a
solution. In this section we show how to use symmetrization to transfer
inequalities\footnote{This circle of ideas of course is well known in the
theory of semigroups, and one can use the symmetrization inequalities in this
context as well (cf \cite{Bor2}, \cite{ledouxbk}). We hope to return to this
point elsewhere.} from one metric space to another. As we shall see the
isoperimetric Hardy property plays an important role in this process.

\begin{theorem}
\label{teosubordinado}Let $(\Omega,d,\mu)$ be a metric probability space of
isoperimetric Hardy type. Suppose that $(\Omega_{1},d_{1},\mu_{1})$ is a
probability metric space such that there exists $c>0$ such that%
\begin{equation}
I_{(\Omega_{1},d_{1},\mu_{1})}(t)\geq cI_{(\Omega,d,\mu)}(t),\text{ \ }%
t\in(0,1/2]. \label{envista}%
\end{equation}
Let $X(\Omega),Y(\Omega)$ be r.i. spaces\ for which there exists a constant
$c>0$ such that the following Poincar\'{e} inequality holds
\begin{equation}
\left\|  g-\int_{{\Omega}}gd\mu\right\|  _{Y(\Omega)}\leq c\left\|  \left|
\nabla g\right|  \right\|  _{X(\Omega)},\text{ \ for all }g\in Lip(\Omega).
\label{envista1}%
\end{equation}
Then, there exists a constant $c_{1}>0$ such that%
\[
\left\|  g-\int_{{\Omega}_{1}}gd\mu_{1}\right\|  _{Y(\Omega_{1})}\leq
c\left\|  \left|  \nabla g\right|  \right\|  _{X(\Omega_{1})},\text{ \ for all
}g\in Lip(\Omega_{1}).
\]
\end{theorem}

\begin{proof}
Since $(\Omega,d,\mu)$ is of isoperimetric Hardy type the Poincar\'{e}
inequality (\ref{envista1}) implies the existence of a constant $\tilde{c}>0$
such that%
\begin{equation}
\left\|  Q_{I_{(\Omega,d,\mu)}}f\right\|  _{\bar{Y}(0,1)}\leq\tilde{c}\left\|
f\right\|  _{\bar{X}(0,1)},\text{ for all }f\geq0,\text{ with supp}%
f\subset(0,1/2). \label{envista2}%
\end{equation}
In view of (\ref{envista}) we have
\[
\int_{t}^{1}f(s)\frac{ds}{I_{(\Omega_{1},d_{1},\mu_{1})}(s)}\preceq\int
_{t}^{1}f(s)\frac{ds}{I_{(\Omega,d,\mu)}(s)},\text{ for all }f\geq0,\text{
with supp}f\subset(0,1/2).
\]
Therefore, (\ref{envista2}) can be lifted to%
\[
\left\|  Q_{I_{(\Omega_{1},d_{1},\mu_{1})}}f\right\|  _{\bar{Y}(0,1)}%
\preceq\left\|  f\right\|  _{\bar{X}(0,1)},\text{ for all }f\geq0,\text{ with
supp}f\subset(0,1/2).
\]
Therefore we conclude by Theorem \ref{opti00}.
\end{proof}

\begin{corollary}
\label{corosuboesferico}Let $M$ be a (compact) connected Riemannian manifold
of dimension $n\geq2$, with Ricci curvature bounded from below by $\rho>0.$
Let $\sigma$ be the normalized volume on $M$. Let $\bar{X}(0,1),$ $\bar
{Y}(0,1)$ be two r.i. spaces$\ $\ for which the following Poincar\'{e}
inequality holds in the probability space $(\mathbb{S}^{n},d,\sigma_{n})$
\[
\left\|  g-\int_{\mathbb{S}^{n}}gd\sigma_{n}\right\|  _{Y(\mathbb{S}^{n}%
)}\preceq\left\|  \left|  \nabla g\right|  \right\|  _{X(\mathbb{S}^{n}%
)},\text{ \ \ }g\in Lip(\mathbb{S}^{n}).
\]
Then,
\[
\left\|  g-\int_{M}gd\sigma\right\|  _{Y(M)}\preceq\left\|  \left|  \nabla
g\right|  \right\|  _{X(M)},\text{ \ \ }g\in Lip(M).
\]
\end{corollary}

\begin{proof}
The L\'{e}vy-Gromov isoperimetric inequality (see \cite{Levy}, \cite{Grom},
\cite{GHL}) yields (recall $I_{n}=I_{\mathbb{S}^{n}},$ see Section
\ref{secc:sphere} above)%
\[
I_{M}\geq\sqrt{\frac{\rho}{n-1}}I_{n}.
\]

Therefore,
\[
\left\|  \int_{t}^{1}f(s)\frac{ds}{I_{M}(s)}\right\|  _{\bar{Y}}%
\preceq\left\|  f\right\|  _{\bar{X}},\text{ \ \ }\forall0\leq f\in\bar
{X},\text{ with }supp(f)\subset(0,1/2),
\]
and the result follows from Theorem \ref{teosubordinado} since $(\mathbb{S}%
^{n},d,\sigma_{n})$ is of isoperimetric Hardy type (cf. Example
\ref{examarkao}).
\end{proof}

\begin{remark}
A version of Corollary \ref{corosuboesferico} in the context of $L^{p}$ spaces
was given in \cite{islas}.
\end{remark}

Finally, let us now present our last example.

Let $1<p\leq2$, $\mu_{p}(x)=Z_{p}^{-1}\exp\left(  -\left|  x\right|
^{p}\right)  dx$, $x\in\mathbb{R}$, and let $\mu=\mu_{p}^{\otimes n}.$ Every
log-concave probability measure $\nu$ on $\mathbb{R}^{d}$ such that
$\exp(\varepsilon\left|  x\right|  ^{p})\in L^{1}(\nu)$ for some
$\varepsilon>0$ and $p\in\lbrack1,2]$ satisfies up to a constant the same
isoperimetric inequality as $\mu_{p}$ (see \cite{Bob1}, and \cite{bart1}).
This result was extended in \cite{BaKo} to the setting of Riemannian manifolds
under appropriate curvature conditions. Using these results we get

\begin{corollary}
\label{coco}Let $M$ be a smooth, complete, connected Riemannian manifold
without boundary. Let $d\nu(x)=e^{-V(x)}d\sigma(x)$ be a probability measure
on $M$, ($\sigma\ $ normalized volume on $M)$ with a twice continuously
differentiable potential $V$. Let $1<p\leq2,$ and suppose that there exists
$x_{0}\in M$ and $\varepsilon>0$ such that%
\[
\exp(\varepsilon d(x_{0},x)^{p})\in L^{1}(\mu),
\]
and, moreover, suppose that
\[
HessV+Ric\geq0.
\]
Let $\bar{X},$ $\bar{Y}$ be two r.i. spaces on $\mathbb{(}0,1)\ $\ for which
the following Poincar\'{e} inequality holds
\[
\left\|  \left(  g-\int_{\mathbb{R}^{n}}gd\mu\right)  _{\mu}^{\ast
}(t)\right\|  _{\bar{Y}}\preceq\left\|  \left|  \nabla g\right|  _{\mu}^{\ast
}\right\|  _{\bar{X}},\text{ \ \ }g\in Lip(\mathbb{R}^{n}).
\]
Then,
\[
\left\|  \left(  f-\int_{M}fd\nu\right)  _{\nu}^{\ast}\right\|  _{\bar{Y}%
}\preceq\left\|  \left|  \nabla g\right|  _{\nu}^{\ast}\right\|  _{\bar{X}%
},\text{ \ \ }g\in Lip(M).
\]
\end{corollary}

\begin{proof}
By the conditions imposed on the manifold (see \cite[Theorem 7.2]{BaKo}) there
exists $\kappa>0$ such that
\[
I_{M}(t)\geq\kappa s\left(  \log\frac{1}{s}\right)  ^{1-\frac{1}{p}}\simeq
I_{\mu_{p}}(s),\text{ \ \ \ \ \ }0<s\leq1/2,
\]
and we conclude using Theorem \ref{teosubordinado}.
\end{proof}

\begin{remark}
A transference principle of Sobolev inequalities for absolutely continuous
probabilities on $\mathbb{R}^{n}$ whose isoperimetric function can be
estimated from below by the isoperimetric function of an even log-concave
probability measure on $\mathbb{R}$ was obtained in \cite[Lemma 2]{bart1}.
\end{remark}

\begin{remark}
Let $M=M_{1}\times M_{2}$ be the product of Riemannian manifolds with volume
$1.$ Then, the isoperimetric profile of $I_{M}$, can be estimated in terms of
the isoperimetric profiles of $I_{M_{i}}$ as follows (see\footnote{For more
information about a comparison theorem for products see \cite[Section 3]{bart}
and \cite[Section 3.3]{Ros}.} \cite{Mor})
\[
I_{M}(s)\geq\frac{1}{\sqrt{2}}\inf\left\{  s_{1}I_{M_{1}}(s_{2})+s_{2}%
I_{M_{2}}(s_{1}):s_{1}s_{2}=s\text{ \ or \ }1-s\right\}  .
\]
For example, if $I_{M_{i}}(s)\geq c_{i}s^{1-1/p_{i}}$, $(p_{i}>1)$, then
\[
I_{M}(s)\geq cs^{1-1/(p_{1}+p_{2})}.
\]
Using this estimate, Theorems \ref{teosubordinado} and \ref{opti00}, we can
easily derive Poincar\'{e} inequalities on $M$.
\end{remark}

\subsection{Gaussian Isoperimetric type and a question of
Triebel\label{secc:trie}}

When we were revising an earlier version of our manuscript we received a query
from Professor Hans Triebel concerning certain Sobolev inequalities with
dimension free constants (cf. \cite{trie}). In this section we provide a
positive answer to Prof. Triebel's question using the transference principle.

We consider Triebel's notation. Let $Q^{n}=(0,1)^{n},$ the unit cube in
$\mathbb{R}^{n}.$ Triebel asks for a treatment of dimension free Sobolev
inequalities for the space $W_{0}^{1,1}(Q^{n})=\overline{C_{0}^{\infty}%
(Q^{n})}^{W^{1,1}(Q^{n})}$. More specifically, Triebel asks (in our notation)
if one can prove dimension free inequalities of the form%
\begin{equation}
\left(  \int_{0}^{1}[f^{\ast}(t)]^{q}(1+\log\frac{1}{t})^{\alpha}dt\right)
^{1/q}\preceq\left\|  \left|  \nabla f\right|  \right\|  _{L^{q}(Q^{n}%
)}+\left\|  f\right\|  _{L^{q}(Q^{n})}, \label{nuvellvague}%
\end{equation}
for a suitable power $\alpha=?$ of the logarithm. To resolve this question, we
first need to understand the ``correct'' power of the logarithm that is needed
here. For this we consider the isoperimetry of $Q^{n}.$ It is known that (cf.
\cite{Bamau}, \cite[Theorem 7]{Ros})
\[
I_{Q^{n}}\geq I_{\gamma}.
\]
Therefore, since $(\mathbb{R}^{n},\gamma_{n})$ is of Hardy isoperimetric type
(cf. \cite{mmjfa}), we can use Theorem \ref{teosubordinado} to transfer to
$Q^{n}$ the Gaussian Poincar\'{e} inequalities. By the asymptotic behavior of
$I_{\gamma_{n}}$ it follows that, for $1<q<\infty,$ we have%
\[
\left(  \int_{0}^{1}\left[  \left(  f-\int_{Q^{n}}f\right)  ^{\ast\ast
}(t)\right]  ^{q}\left(  1+\log\frac{1}{t}\right)  ^{q/2}dt\right)
^{1/q}\preceq\left\|  \left|  \nabla f\right|  \right\|  _{L^{q}(Q^{n})},
\]
with constants independent of the dimension. Finally, an application of the
triangle inequality yields%
\[
\left(  \int_{0}^{1}f^{\ast\ast}(t)^{q}\left(  1+\log\frac{1}{t}\right)
^{q/2}dt\right)  ^{1/q}\preceq\left\|  \left|  \nabla f\right|  \right\|
_{L^{q}(Q^{n})}+\left\|  f\right\|  _{L^{q}(Q^{n})},
\]
and the constants are independent of the dimension. This statement proves
(\ref{nuvellvague}) with $\alpha=q/2,$ thus providing a positive answer to
Professor Triebel's conjecture.

Let us consider a similar result for the $p-$unit ball, i.e. let%
\[
B_{p}^{n}=\left\{  x=(x_{1},\cdots,x_{n}):\left\|  x\right\|  _{p}^{p}=\left|
x_{1}\right|  ^{p}+\cdots+\left|  x_{n}\right|  ^{p}\leq1\right\}  ,\text{
\ \ \ }1\leq p\leq2,
\]
and consider on $B_{p}^{n}$ the normalized volume measure%

\[
V_{p}^{n}=\frac{vol\mid_{B_{p}^{n}}}{vol(B_{p}^{n})}.
\]
In the recent paper \cite{sod}, S. Sodin proves that,
\[
I_{V_{p}^{n}}(\tilde{a})\geq cn^{1/p}\tilde{a}\log^{1-1/p}\frac{1}{\tilde{a}%
};\text{ \ \ \ }\tilde{a}=\min(a,1-a);\text{ }0<a<1,
\]
where $c$ is an absolute constant; in particular, since $n\geq2,$ we get
\[
I_{V_{p}^{n}}(\tilde{a})\geq c2^{1/p}\tilde{a}\log^{1-1/p}\frac{1}{\tilde{a}%
}.
\]
At this point we can use again Theorem \ref{teosubordinado} to transfer to
$V_{p}^{n}$ the Poincar\'{e} inequalities. Indeed, let $1\leq p\leq2$ and
consider the measure
\[
\mu_{p}=Z_{p}^{-1}\exp\left(  -\left|  x\right|  ^{p}\right)  dx,\text{
\ \ \ }x\in\mathbb{R}.
\]
Since $(\mathbb{R}^{n},\mu_{p}^{\otimes n})$ is of Hardy isoperimetric type
(see Example \ref{examarkao} above) and by the asymptotic properties of
$I_{\mu_{p}^{\otimes n}}$ (see (\ref{asim}))$,$ there exist constants $c_{1}$
and $c_{2},$ that do not depend on $n,$ such that
\[
c_{1}\tilde{a}\log^{1-1/p}\frac{1}{\tilde{a}}\leq I_{\mu_{p}^{\otimes n}%
}(\tilde{a})\leq c_{2}\tilde{a}\log^{1-1/p}\frac{1}{\tilde{a}}.
\]
By Theorem \ref{teosubordinado} it follows that, for $1<q<\infty,$ we have%
\[
\left(  \int_{0}^{1}\left[  \left(  f-\int_{B_{p}^{n}}fdV_{p}^{n}\right)
^{\ast\ast}(t)\right]  ^{q}\left(  1+\log\frac{1}{t}\right)  ^{q\left(
1-1/p\right)  }dt\right)  ^{1/q}\preceq\left\|  \left|  \nabla f\right|
\right\|  _{L^{q}(B_{p}^{n},dV_{p}^{n})}.
\]
Consequently,%
\[
\left(  \int_{0}^{1}f^{\ast\ast}(t)^{q}\left(  1+\log\frac{1}{t}\right)
^{q\left(  1-1/p\right)  }dt\right)  ^{1/q}\preceq\left\|  \left|  \nabla
f\right|  \right\|  _{L^{q}(B_{p}^{n},dV_{p}^{n})}+\left\|  f\right\|
_{L^{q}(B_{p}^{n},dV_{p}^{n})},
\]
with constants that are independent of the dimension.

\begin{remark}
In the particular case $p=2$, $q=2$ and $f\in W_{0}^{1,2}(B_{2}^{n}%
)=\overline{C_{0}^{\infty}(Q^{n})}^{W_{0}^{1,2}(B_{2}^{n})}$ this result was
obtained in \cite{KS1}. For $p=2$ and $1<q<n/3$ and other related results see
\cite{KS2}.
\end{remark}

One could also approach other questions posed by Triebel using our techniques
but this would take us too far away from the main topics of this paper.

On the other hand the ideas discussed in this section can be pushed further.
Let $(M,d)$ be a Riemannian manifold endowed with a probability measure $\mu$
on $M$ which is absolutely continuous with respect the volume $dvol_{M}.$ We
say that $M$ admits a \textbf{Gaussian isoperimetric inequality}, if there is
a positive constant $c(\mu)$ such that
\[
I_{\mu}(t)\geq c(\mu)I_{\gamma}(t)
\]
(where $I_{\gamma}$ denotes the Gaussian isoperimetric profile). It is known
that this family includes any compact manifold (with or without boundary)
endowed with its Riemannian probability (see \cite{Ros} an the references
quoted therein).

\begin{corollary}
\label{corosubogauss}Let $\gamma_{n}$ be the Gaussian measure on
$\mathbb{R}^{n}.$ Let $(M,d)$ be a Riemannian manifold which admits a Gaussian
isoperimetric inequality. Suppose that $\bar{X},$ $\bar{Y}$ are r.i. spaces on
$(0,1),$ \ for which the Gaussian Poincar\'{e} inequality holds:
\[
\left\|  g-\int_{\mathbb{R}^{n}}gd\gamma_{n}\right\|  _{Y(\mathbb{R}%
^{n},\gamma_{n})}\preceq\left\|  \left|  \nabla g\right|  \right\|
_{X(\mathbb{R}^{n},\gamma_{n})},\text{ \ \ }g\in Lip(\mathbb{R}^{n}).
\]
Then,
\[
\left\|  g-\int_{{M}}gd\mu\right\|  _{Y(M,d)}\preceq\left\|  \left|  \nabla
g\right|  \right\|  _{X(M,d)},\text{ \ \ }g\in Lip(M).
\]
In particular, if $1<p<\infty,$ there exists a constant $c_{p}$ such that
\[
\int_{0}^{1}f^{\ast}(t)^{p}\left(  1+\log\frac{1}{t}\right)  ^{p/2}d\mu\leq
c_{p}\left(  \int_{M}\left|  \nabla f(x)\right|  ^{p}d\mu+\int_{M}\left|
f(x)\right|  ^{p}d\mu\right)  ,\text{ \ \ }f\in Lip(M).
\]
\end{corollary}

\section{Estimating isoperimetric profiles via semigroups\label{secc:semi}}

In this section we discuss an extension of the approach in \cite{MiE},
\cite{mie2} to the self improving results in Section \ref{secc:isohar}. In the
case of connected Riemannian manifolds, whose Ricci curvature is bounded from
below, E. Milman using methods of Ledoux (\cite{led1}, \cite{led2},
\cite{led3}) has developed a semigroup approach which produces isoperimetric
estimates starting from the Poincar\'{e} inequalities
\[
\left\|  g-\int_{\Omega}gd\mu\right\|  _{X}\preceq\left\|  \left|  \nabla
g\right|  \right\|  _{L^{q}},\ \ g\in Lip(\Omega),
\]
where $X$ is an $L^{p}$ space or an Orlicz space. In this section we show that
the analysis can be streamlined and extended to r.i. spaces.

Let $\Omega=(M,g)$ be a smooth complete connected Riemannian manifold equipped
with a probability measure $\mu,$ with density $d\mu=exp(-\psi)dVol_{M},$
$\psi\in C^{2}(M,\mathbb{R}).$ Let
\[
\Delta_{(\Omega,\mu)}=\Delta_{\Omega}-\nabla\psi\cdot\nabla,
\]
be the associated Laplacian ($\Delta_{\Omega}$ is the usual Laplace-Beltrami
operator on $\Omega$). Let $(P_{t})_{t\geq0}$ denote the semi-group associated
to the diffusion process with infinitesimal generator $\Delta_{(\Omega,\mu)}$
(see \cite{Dav}, \cite{led2}) characterized by the second order system
\[
\frac{\partial}{\partial t}P_{t}(f)=\Delta_{(\Omega,\mu)}(P_{t}(f)),\text{
\ \ }P_{0}(f)=f,
\]
where $f\in\mathcal{B}(\Omega)$ (the space of bounded smooth\footnote{we could
use $C^{\infty}$ functions here.} real functions on $\Omega).$

For each $t\geq0$, $p\geq1,$ $P_{t}:L^{p}(\Omega)\rightarrow L^{p}(\Omega)$ is
a bounded linear operator$.$ We list a few elementary properties of these operators

\begin{itemize}
\item $P_{t}1=1.$

\item $f\geq0\Rightarrow P_{t}f\geq0.$

\item $\int\left(  P_{t}f\right)  gd\mu=\int f\left(  P_{t}g\right)  d\mu.$

\item $\left(  P_{t}f\right)  ^{\alpha}\leq P_{t}f^{\alpha},$ $\forall
\alpha\geq1.$

\item $P_{t}\circ P_{s}=P_{s+t}.$

\item $P_{t}:X(\Omega)\rightarrow X(\Omega)$ is bounded on any r.i. space
$X(\Omega).$
\end{itemize}

Moreover, if the Bakry-\'{E}mery curvature-dimension condition holds (cf.
\cite{bakrled}):
\begin{equation}
Ric_{g}+Hess_{g}\psi\geq0, \label{be}%
\end{equation}
then, for all $t\geq0$ and $f\in\mathcal{B}(\Omega),$ we have the pointwise
inequality
\begin{equation}
2t\left|  \nabla P_{t}f\right|  ^{2}\leq P_{t}f^{2}-\left(  P_{t}f\right)
^{2}. \label{be1}%
\end{equation}

\begin{theorem}
Let$\ \Omega=(M,g)$ be a smooth complete connected Riemannian manifold which
satisfies the convexity assumption (\ref{be}). Let $X,Y$ be two r.i. spaces on
$\Omega$ such that conditions (a) and (b) hold:

Condition (a): One of the following conditions holds. Either (i) $X$ is $q$
concave for some $q\geq2$;

or

(ii) $\bar{\alpha}_{X}<1/2.$

Condition (b): There exists $c=c(X,Y)$ such that the $(Y,X)$ Poincar\'{e}
inequality holds for all $g\in Lip(\Omega)$%
\begin{equation}
\left\|  g-\int_{\Omega}gd\mu\right\|  _{Y}\leq c\left\|  \left|  \nabla
g\right|  \right\|  _{X}. \label{poipoi}%
\end{equation}
Then, there exists a constant $c_{1}>0$ such that%
\[
I_{(M,g,\mu)}(t)\geq c_{1}t(1-t)\frac{\phi_{Y}(t(1-t))}{\phi_{X}(t(1-t))},
\]
where $\phi_{X}$ and $\phi_{Y}$ are the fundamental functions of the r.i.
spaces $X$ and $Y.$

\begin{proof}
We shall follow closely Milman's proof of Theorem 2.9 in \cite{MiE}. Let $A$
denote an arbitrary Borel set in $\Omega$ with $\mu^{+}(A)<\infty.$ We need to
show
\begin{equation}
\mu^{+}(A)\geq c_{1}\mu(A)(1-\mu(A))\frac{\phi_{X}((1-\mu(A))\mu(A))}{\phi
_{Y}((1-\mu(A))\mu(A))}. \label{obtenida}%
\end{equation}
Using a standard approximation argument (cf. \cite{MiE}) we get%
\[
\sqrt{2t}\mu^{+}(A)\geq\int\left|  \chi_{A}-P_{t}\chi_{A}\right|  d\mu.
\]
Rewrite the right hand side as follows
\begin{align*}
\int\left|  \chi_{A}-P_{t}\chi_{A}\right|  d\mu &  =\int_{A}\left(
1-P_{t}\chi_{A}\right)  d\mu+\int_{\Omega\diagdown A}P_{t}\chi_{A}%
d\mu=2\left(  \mu(A)-\int_{A}P_{t}\chi_{A}d\mu\right) \\
&  =2\left(  \mu(A)\left(  1-\mu(A)\right)  -\int_{\Omega}\left(  P_{t}%
\chi_{A}-\mu(A)\right)  \left(  \chi_{A}-\mu(A)\right)  d\mu\right)  .
\end{align*}
Using the fact that $X$ satisfies condition (a) we will show that there exists
a constant $c>0$ such that
\begin{align}
J(t)  &  =\int_{\Omega}\left(  P_{t}\left(  \chi_{A}-\mu(A)\right)  \right)
\left(  \chi_{A}-\mu(A)\right)  d\mu\nonumber\\
&  \leq\frac{4c}{\sqrt{2t}}\phi_{X}((1-\mu(A))\mu(A))\frac{(1-\mu(A))\mu
(A)}{\phi_{Y}((1-\mu(A))\mu(A))}. \label{virtual}%
\end{align}
This given, we deduce that
\begin{align*}
\mu^{+}(A)  &  \geq\frac{\mu(A)\left(  1-\mu(A)\right)  -J(t)}{\sqrt{2t}}\\
&  \geq(1-\mu(A))\mu(A)\left(  \frac{1}{\sqrt{2t}}-\frac{2c}{t}\frac{\phi
_{X}((1-\mu(A))\mu(A))}{\phi_{Y}((1-\mu(A))\mu(A))}\right)  .
\end{align*}
Choosing%
\[
t_{0}=16\left(  c\frac{\phi_{X}((1-\mu(A))\mu(A))}{\phi_{Y}((1-\mu(A))\mu
(A))}\right)  ^{2},
\]
we obtain (\ref{obtenida}). It remains to prove (\ref{virtual}). By
H\"{o}lder's inequality, (\ref{poipoi}) and (\ref{be1}), we find
\begin{align}
J(t)  &  =\int_{\Omega}\left(  P_{t}\left(  \chi_{A}-\mu(A)\right)  \right)
\left(  \chi_{A}-\mu(A)\right)  d\mu\nonumber\\
&  \leq\left\|  P_{t}\left(  \chi_{A}-\mu(A)\right)  \right\|  _{Y}\left\|
\chi_{A}-\mu(A)\right\|  _{Y^{\prime}}\nonumber\\
&  \leq\frac{c}{\sqrt{2t}}\left\|  \nabla P_{t}\left(  \chi_{A}-\mu(A)\right)
\right\|  _{X}\left\|  \chi_{A}-\mu(A)\right\|  _{Y^{\prime}}\nonumber\\
&  \leq\frac{c}{\sqrt{2t}}\left\|  \sqrt{P_{t}\left(  \chi_{A}-\mu(A)\right)
^{2}}\right\|  _{X}\left\|  \chi_{A}-\mu(A)\right\|  _{Y^{\prime}}.
\label{nova}%
\end{align}

If $X$ is $q$ concave, then $X^{(\frac{1}{q})}$ is an r.i. space and,
therefore, $P_{t}$ is bounded on $X^{(\frac{1}{q})}.$ Consequently,
\begin{align}
\left\|  \sqrt{P_{t}\left(  \chi_{A}-\mu(A)\right)  ^{2}}\right\|  _{X} &
=\left\|  \left(  P_{t}\left(  \chi_{A}-\mu(A)\right)  ^{2}\right)  ^{\frac
{q}{2}}\right\|  _{X^{(\frac{1}{q})}}^{q}\nonumber\\
&  =\left\|  P_{t}\left(  \chi_{A}-\mu(A)\right)  ^{q}\right\|  _{X^{(\frac
{1}{q})}}^{q}\text{ \ \ \ \ \ \ (since }q/2\geq1)\nonumber\\
&  \leq\left\|  \left(  \chi_{A}-\mu(A)\right)  ^{q}\right\|  _{X^{(\frac
{1}{q})}}^{q}\nonumber\\
&  =\left\|  \chi_{A}-\mu(A)\right\|  _{X}.\label{nova1}%
\end{align}

On the other hand, suppose now that $\bar{\alpha}_{X}<1/2$ holds$.$ Then,%
\begin{align}
\left\|  \sqrt{P_{t}\left(  \chi_{A}-\mu(A)\right)  ^{2}}\right\|  _{X}  &
\leq\left\|  \left(  \frac{1}{r}\int_{0}^{r}\left[  P_{t}\left(  \chi_{A}%
-\mu(A)\right)  \right]  ^{\ast}(s)^{2}ds\right)  ^{1/2}\right\|  _{\bar{X}%
}\nonumber\\
&  \leq c\left\|  P_{t}\left(  \chi_{A}-\mu(A)\right)  \right\|
\text{\ \ \ \ \ \ (since }\bar{\alpha}_{X}<1/2)\nonumber\\
&  \leq c\left\|  \chi_{A}-\mu(A)\right\|  _{X}. \label{nova2}%
\end{align}

To estimate the right hand side of (\ref{nova1}) and (\ref{nova2}) we note
that for any r.i. space $Z=Z(\Omega)$ we have,
\begin{align}
\left\|  \chi_{A}-\mu(A)\right\|  _{Z}  &  \leq(1-\mu(A))\left\|  \chi
_{A}\right\|  _{Z}+\mu(A)\left\|  \chi_{\Omega\diagdown A}\right\|
_{Z}\nonumber\\
&  =(1-\mu(A))\phi_{Z}(\mu(A))+\mu(A)\phi_{Z}(1-\mu(A))\nonumber\\
&  \leq2\phi_{Z}((1-\mu(A))\mu(A)), \label{nova3}%
\end{align}
where in the last inequality we have used the concavity of $\phi_{Z}$.

Combining (\ref{nova3}), (\ref{nova2}), (\ref{nova1}) and (\ref{nova}) yields
\begin{align*}
J(t)  &  \leq\frac{c}{\sqrt{2t}}\left\|  \chi_{A}-\mu(A)\right\|  _{X}\left\|
\chi_{A}-\mu(A)\right\|  _{Y^{\prime}}\\
&  \leq\frac{4c}{\sqrt{2t}}\phi_{X}((1-\mu(A))\mu(A))\phi_{Y^{\prime}}%
((1-\mu(A))\mu(A))\\
&  =\frac{4c}{\sqrt{2t}}\phi_{X}((1-\mu(A))\mu(A))\frac{(1-\mu(A))\mu(A)}%
{\phi_{Y}((1-\mu(A))\mu(A))}\text{ \ (by (\ref{dual})}).
\end{align*}

Therefore, (\ref{virtual}) holds and the desired result follows.
\end{proof}
\end{theorem}

\begin{remark}
Note that for any r.i. space $Z=Z(\Omega),$ we have $Z^{(2)}\subset Z$, and
$Z^{(2)}$ is $2-$concave. It follows from the previous result that for any
smooth complete connected Riemannian manifold that satisfies the convexity
assumption (\ref{be}) the isoperimetric estimate
\[
I_{(M,g,\mu)}(t)\geq c_{1}t\frac{\phi_{Y}(t)}{\sqrt{\phi_{X}(t)}%
},\;\;\;0<t\leq1/2
\]
follows from
\[
\left\|  g-\int_{\Omega}gd\mu\right\|  _{Y}\leq c\left\|  \left|  \nabla
g\right|  \right\|  _{X},\ \ \ \ \forall g\in Lip(\Omega).
\]
\end{remark}

\section{Higher order Sobolev inequalities\label{secc:ga}}

In this section we consider the higher order versions of Theorem
\ref{teomain}. Since the setting of metric spaces is not adequate to deal with
higher order derivatives in this section we work on Riemannian manifolds.

Let $\Omega=(M,g)$ be a smooth complete connected Riemannian manifold equipped
with a probability measure $\mu.$ Under the presence of smoothness we can give
more precise formulae. The next result is essentially given in \cite{ga}, we
provide a detailed proof for the sake of completeness.

\begin{proposition}
\label{castigada}Let $I$ be an isoperimetric estimator. Suppose that $f\in
C^{\infty}\left(  \Omega\right)  $ is a positive function, and denote by
$d\mathcal{H}_{n-1}$ the corresponding $(n-1)$ dimensional measure on
$\{f=t\}$ associated with $d\mu.$ Moreover, suppose that $f$ has no degenerate
critical points. Then,

(i) For all regular values of $f$ (therefore a.e. $t>0$)%
\begin{equation}
\frac{d}{dt}(\mu_{f}(t))=\frac{1}{\left(  f_{\mu}^{\ast}\right)  ^{\prime}%
(\mu_{f}(t))}=-\int_{\{f=t\}}\frac{1}{\left|  \nabla f(x)\right|
}d\mathcal{H}_{n-1}(x). \label{uno}%
\end{equation}

(ii) For almost all $t$%
\begin{equation}
\int_{\{f=t\}}\left|  \nabla f(x)\right|  ^{q-1}d\mathcal{H}_{n-1}%
(x)\geq\left(  I(\mu_{f}(t))\right)  ^{q}\left(  \left(  -f_{\mu}^{\ast
}\right)  ^{\prime}(\mu_{f}(t))\right)  ^{q-1}. \label{dosii}%
\end{equation}
In particular, for all almost all $t\in\lbrack0,ess\sup f),$%
\[
\int_{\{f=f_{\mu}^{\ast}(t)\}}\left|  \nabla f(x)\right|  ^{q-1}%
d\mathcal{H}_{n-1}(x)\geq\left(  I(t))\right)  ^{q}\left(  \left(  -f_{\mu
}^{\ast}\right)  ^{\prime}(t))\right)  ^{q-1}.
\]

(iii) ($q-$Ledoux inequality)%
\begin{equation}
\int\left|  \nabla f(x)\right|  ^{q}d\mu\geq\int_{0}^{\infty}I(\mu_{f}%
(t))^{q}\left(  \left(  -f_{\mu}^{\ast}\right)  ^{\prime}(\lambda
_{f}(t))\right)  ^{q-1}dt. \label{dosi}%
\end{equation}
\end{proposition}

\begin{proof}
$(i)$ The co-area formula implies (cf. \cite[pag 157]{ciaetal})%
\[
\mu_{f}(t)=\mu\left(  \{f>t\}\cap\{\left|  \nabla f\right|  =0\}\right)
+\int_{t}^{\infty}\int_{\{f=s\}}\frac{1}{\left|  \nabla f(x)\right|
}d\mathcal{H}_{n-1}(x)ds.
\]
Our assumptions on $f$ imply that%
\[
\mu\left(  \{f>t\}\cap\{\left|  \nabla f\right|  =0\}\right)  =0,\text{
}a.e\text{.}%
\]
Consequently,%
\[
\frac{d}{dt}(\mu_{f}(t))=-\int_{\{f=t\}}\frac{1}{\left|  \nabla f(x)\right|
}d\mathcal{H}_{n-1}(x),\text{ }a.e.
\]
Since $f_{\mu}^{\ast}$ and $\mu_{f}$ restricted to $[0,ess\sup\left|
f\right|  ]$ are inverses (cf. \cite[pag 935]{Ta1}), we get%
\[
f_{\mu}^{\ast}(\mu_{f}(t))=t,
\]
and therefore the remaining formula in (\ref{uno}) follows.

$(ii)$ By the definition of isoperimetric profile%
\[
I(\mu_{f}(t))\leq\int_{\{f=t\}}d\mathcal{H}_{n-1}(x).
\]
We estimate the right hand side using H\"{o}lder's inequality,%
\begin{align*}
\int_{\{f=t\}}d\mathcal{H}_{n-1}(x)  &  =\int_{\{f=t\}}\left|  \nabla
f(x)\right|  ^{1/q^{\prime}}\frac{1}{\left|  \nabla f(x)\right|
^{1/q^{\prime}}}d\mathcal{H}_{n-1}(x)\\
&  \leq\left(  \int_{\{f=t\}}\left|  \nabla f(x)\right|  ^{q-1}d\mathcal{H}%
_{n-1}(x)\right)  ^{1/q}\left(  \int_{\{f=t\}}\frac{1}{\left|  \nabla
f(x)\right|  }d\mathcal{H}_{n-1}(x)\right)  ^{1/q^{\prime}}.
\end{align*}
Combining these inequalities we obtain%
\[
I(\mu_{f}(t))^{q}\leq\left(  \int_{\{f=t\}}\left|  \nabla f(x)\right|
^{q-1}d\mathcal{H}_{n-1}(x)\right)  \left(  \int_{\{f=t\}}\frac{1}{\left|
\nabla f(x)\right|  }d\mathcal{H}_{n-1}(x)\right)  ^{q-1}.
\]
Therefore, by (\ref{uno})%
\[
I(\mu_{f}(t))^{q}\left(  \left(  -f_{\mu}^{\ast}\right)  ^{\prime}(\mu
_{f}(t))\right)  ^{q-1}\leq\int_{\{f=t\}}\left|  \nabla f(x)\right|
^{q-1}d\mathcal{H}_{n-1}(x).
\]
$(iii)$ The co-area formula implies%
\[
\int_{0}^{\infty}\left(  \int_{\{f=t\}}\left|  \nabla f(x)\right|
^{q-1}d\mathcal{H}_{n-1}(x)\right)  dt=\int_{\Omega}\left|  \nabla
f(x)\right|  ^{q}d\mu,
\]
consequently (\ref{dosi}) follows by integrating (\ref{dosii}).
\end{proof}

\begin{remark}
\label{remarkao1}In particular if $q=1$ then (\ref{dosi}) becomes Ledoux's
inequality (cf. (\ref{ledo}) above)%
\[
\int_{0}^{\infty}I(\mu_{f}(t))dt\leq\int_{\Omega}\left|  \nabla f(x)\right|
d\mu.
\]
\end{remark}

\begin{remark}
Formulae (\ref{uno}) appears in several places in the literature (cf.
\cite[(1), pag 709]{Ta}, \cite[pag 81]{ber}, \cite[pag 52]{Ban}) with
different degrees of generality. In concrete applications when the ``correct''
symmetrization $f^{\circ}$ is available (e.g. $\mathbb{R}^{n},$ with Lebesgue
or Gaussian measure), then for smooth enough $f,$ we have for $a.e.$ $t,$%
\[
\mu\left(  \{f^{\circ}>t\}\cap\{\left|  \nabla f^{\circ}\right|  =0\}\right)
=0
\]
and
\[
\frac{d}{dt}(\mu_{f}(t))=-\int_{\{f^{\circ}=t\}}\frac{1}{\left|  \nabla
f^{\circ}(x)\right|  }d\mathcal{H}_{n-1}(x),\text{ }a.e.
\]
follows.
\end{remark}

\begin{remark}
To extend these inequalities we can use Morse theory. Indeed, it is well known
(cf. \cite[pag 37]{milnor}) that bounded smooth functions can be uniformly
approximated (together with their derivatives) by smooth functions with non
degenerate critical points.
\end{remark}

Our objective is to extend the first order estimates (\ref{dosa}) and
(\ref{rea}) of Theorem \ref{teomain}. The corresponding results are given by
our next theorem

\begin{theorem}
\label{teomarkao}Suppose that the assumptions of Proposition \ref{castigada}
hold. Then,

\begin{enumerate}
\item [(i)]Maz'ya-Talenti second order inequality%
\begin{equation}
-I(t)^{2}\left(  -f_{\mu}^{\ast}\right)  ^{\prime}(t)\leq\int_{0}^{t}\left|
\Delta f\right|  _{\mu}^{\ast}(s)ds,\text{ a.e.} \label{tresa}%
\end{equation}

\item[(ii)] Oscillation inequality%
\begin{equation}
f_{\mu}^{\ast\ast}(t)-f_{\mu}^{\ast}(t)\leq\frac{1}{t}\int_{0}^{t}\left(
\frac{s}{I(s)}\right)  ^{2}\left|  \Delta f\right|  _{\mu}^{\ast\ast}(s)ds
\label{tres}%
\end{equation}
\end{enumerate}
\end{theorem}

\begin{proof}
(i) In preparation to use Green's formula we write%
\[
\Delta f=-div(\nabla f).
\]
Note that the level surface $\{f=t\}=\partial\{f>t\}$ and moreover that the
formula for the inner unit normal to $\{f=t\}$ at a point $x$ is given by
\[
\nu(x)=\frac{\nabla f(x)}{\left|  \nabla f(x)\right|  }.
\]
Therefore, by Green's theorem,%
\begin{align*}
-\int_{\{f>t\}}\Delta f(x)d\mu &  =\int_{\{f>t\}}div(\nabla f)\\
&  =\int_{\{f=t\}}\frac{\left|  \nabla f(x)\right|  ^{2}}{\left|  \nabla
f(x)\right|  }d\mathcal{H}_{n-1}(x)\\
&  \geq I(\mu_{f}(t))^{2}\left(  -f_{\mu}^{\ast}\right)  ^{\prime}(\mu
_{f}(t))\text{ (by (\ref{dosii})).}%
\end{align*}
Consequently for a.e. $t,$%
\begin{align*}
I(t)^{2}\left(  -f_{\mu}^{\ast}\right)  ^{\prime}(t)  &  \leq\int_{\{f>f_{\mu
}^{\ast}(t)\}}\left|  \Delta f(x)\right|  d\mu\\
&  \leq\int_{0}^{t}\left|  \Delta f(x)\right|  _{\mu}^{\ast}(s)ds,
\end{align*}
as we wished to show.

(ii) We start with the familiar (cf. Theorem \ref{teomain} above, specially
the proof of $(3)\Rightarrow(5)$)$,$
\[
f_{\mu}^{\ast\ast}(t)-f_{\mu}^{\ast}(t)\leq\frac{1}{t}\int_{0}^{t}s\left(
-f_{\mu}^{\ast}\right)  ^{\prime}(s)ds.
\]
We work with the right hand side as follows,%
\begin{align*}
\frac{1}{t}\int_{0}^{t}s\left(  -f_{\mu}^{\ast}\right)  ^{\prime}(s)ds  &
=\frac{1}{t}\int_{0}^{t}\frac{s}{I(s)^{2}}I(s)^{2}\left(  -f_{\mu}^{\ast
}\right)  ^{\prime}(s)ds\\
&  \leq\frac{1}{t}\int_{0}^{t}\frac{s}{I(s)^{2}}\left(  \frac{s}{s}\int
_{0}^{s}\left|  \Delta f\right|  _{\mu}^{\ast}(u)du\right)  ds\text{ (by
(\ref{tresa}))}\\
&  =\frac{1}{t}\int_{0}^{t}\left(  \frac{s}{I(s)}\right)  ^{2}\left|  \Delta
f\right|  _{\mu}^{\ast\ast}(s)ds.
\end{align*}
\end{proof}

\begin{remark}
Since in this paper we assume that I(s) is concave, therefore we see that
(\ref{tres}) implies the more suggestive inequality%
\begin{equation}
f_{\mu}^{\ast\ast}(t)-f_{\mu}^{\ast}(t)\leq\left(  \frac{t}{I(t)}\right)
^{2}\frac{1}{t}\int_{0}^{t}\left|  \Delta f\right|  _{\mu}^{\ast\ast}(s)ds.
\label{cuatro}%
\end{equation}
\end{remark}

We discuss briefly some examples. It follows from (\ref{cuatro}) and a routine
approximation that for r.i. spaces away from $L^{1}$ (i.e. $\bar{\alpha}%
_{X}<1$) we have%
\begin{equation}
\left\|  \left(  f_{\mu}^{\ast\ast}(t)-f_{\mu}^{\ast}(t)\right)  \left(
\frac{I(t)}{t}\right)  ^{2}\right\|  _{\bar{X}}\preceq\left\|  \left|  \Delta
f\right|  \right\|  _{X},\text{ }f\in C^{\infty}(\Omega). \label{jornada}%
\end{equation}

In the Euclidean case (\ref{jornada}) can be used to extend the results in
\cite{mp}, while in the Gaussian case they provide an extension of the results
in \cite{fei}, \cite{bakme}, \cite{bakme2}, \cite{shi} to the context of r.i.
spaces. For comparison we note that the method of proof used in these
references is completely different.

For example, to recover the higher order Gaussian $L^{p}$ Sobolev results in
these references, we just need to observe that in this case%
\[
\left\|  \left(  f_{\mu}^{\ast\ast}(t)-f_{\mu}^{\ast}(t)\right)  \left(
\frac{I(t)}{t}\right)  ^{2}\right\|  _{L^{p}}\simeq\left\|  f\right\|
_{L^{p}(LogL)^{p}}.
\]

Our inequalities also apply to the measures
\[
\mu_{p,\alpha}=Z_{p,\alpha}^{-1}\exp\left(  -\left|  x\right|  ^{p}%
(\log(\gamma+\left|  x\right|  )^{\alpha}\right)  dx,
\]
discussed in Example \ref{exex} above. The corresponding inequalities can be
readily obtained since we have precise estimates of the isoperimetric profiles
$I_{\mu_{p,\alpha}^{\otimes n}}(s).$

In the next section we shall see a considerable extension of these results, as
well as applications to the study of non-linear elliptic equations.

\section{Integrability of solutions of elliptic equations\label{secc:el}}

The techniques discussed in this paper also have applications to the study of
the integrability and regularity of the solutions of non-linear elliptic
equations of the form%

\begin{equation}
\left\{
\begin{array}
[c]{ll}%
-div(a(x,u,\nabla u))=fw & \text{ in }G,\\
u=0 & \text{ on }\partial G,
\end{array}
\right.  \label{equaxx}%
\end{equation}
where $G$ is domain of $\mathbb{R}^{n}$ ($n\geq2),$ such that $\mu=w(x)dx$ is
a probability measure on $\mathbb{R}^{n},$ or $G$ has Lebesgue measure $1$ if
$w=1,$ and $a(x,\eta,\xi):G\times\mathbb{R}\times\mathbb{R}^{n}\rightarrow
\mathbb{R}^{n}$ is a Carath\'{e}odory function such that for some fixed
$p>1$,
\begin{equation}
a(x,t,\xi).\xi\geq w(x)\left|  \xi\right|  ^{p},\text{ \ \ for a.e. }x\in
G\subset\mathbb{R}^{n},\text{ \ }\forall\eta\in\mathbb{R},\text{ \ }\forall
\xi\in\mathbb{R}^{n}. \label{elipxx}%
\end{equation}

In what follows to fix ideas and simplify the presentation we take
\[
p=2,
\]
but were appropriate we shall indicate the necessary changes to deal with the
general case (cf. Remark \ref{remarkp} below).

To see what results are possible consider the special case$,w=1,a(x,t,\xi
)=\xi.$ Then (\ref{equaxx}) becomes
\[
\left\{
\begin{array}
[c]{ll}%
\tilde{\Delta}u=f & \text{ in }G,\\
u=0 & \text{ on }\partial G.
\end{array}
\right.
\]
In this case we can derive \textit{a priori} sharp integrability of the
solutions directly from the results in Section \ref{secc:ga} to find that
\[
\left(  -u_{\mu}^{\ast}\right)  ^{\prime}(t)\left(  \frac{I(t)}{t}\right)
^{2}\leq\frac{1}{t}\int_{0}^{t}f_{\mu}^{\ast\ast}(s)ds,
\]
where $I=I_{(\mathbb{R}^{n};\mu)}$ is the isoperimetric profile of
$(\mathbb{R}^{n};\mu).$ These estimates lead to the following \textit{a
priori} sharp integrability result
\[
\left\|  \left(  u_{\mu}^{\ast\ast}(t)-u_{\mu}^{\ast}(t)\right)  \left(
\frac{I(t)}{t}\right)  ^{2}\right\|  _{\bar{X}}\preceq\left\|  f_{\mu}%
^{\ast\ast}\right\|  _{X}.
\]

In this section we shall extend these estimates to solutions of (\ref{equaxx})
(cf. Theorem \ref{base}). Moreover, we also obtain results on the regularity
of $\left|  \nabla f\right|  .$ For example, we will show that
\[
\left|  \nabla u\right|  _{\mu}^{\ast}(t)\leq\left(  \frac{2}{t}\int
_{t/2}^{\mu(G)}\left(  \frac{I(s)}{s}f_{\mu}^{\ast\ast}(s)\right)
^{2}ds\right)  ^{1/2}.
\]
These estimates can be used to obtain, under suitable assumptions on $\bar{X}
$ (cf. Theorem \ref{remderv} below),
\[
\left\|  \frac{I(t)}{t}\left|  \nabla u\right|  _{\mu}^{\ast}(t)\right\|
_{\bar{X}}\preceq\left\|  f_{\mu}^{\ast\ast}\right\|  _{\bar{X}}.
\]

As with most other results in this paper, our estimates incorporate the
isoperimetric profile and thus are valid for different geometries. In
particular, our results are valid for domains on $\mathbb{R}^{n}$ provided
with Lebesgue or Gaussian measure, and in both instances our \textit{a priori}
integrability results are sharp. In fact, the integrability results that we
obtain contain all the known results (previously known for specific r.i.
spaces like Orlicz or Lorentz spaces), and, furthermore, are new and sharper
on the borderline cases. The integrability of the gradient is a more difficult
problem for these methods, and here our results are not definitive even
though, for a certain range of values of the parameters, we extend and improve
on the classical results (cf. \cite{AFT}, \cite{BBGBPV}, \cite{muller}, for
more on this point as well as an extensive list of references).

To proceed we needed an adequate notion of solution. Indeed, in the literature
one can find a number of different definitions of what is ``a'' solution for
problem (\ref{equaxx}). However, under fairly general conditions it is well
known that many of these definitions coincide (cf. \cite{AFT}). We adopt the
definition of entropy (or entropic) solution{\footnote{For example, in the
classical case (i.e. $w(x)=1$ and $G$ bounded), under further assumptions on
$a(x,t,\xi),$ it has been proved that an entropy solution of (\ref{equaxx})
exists (see, for example, \cite{BBGBPV} and the references therein).} since it
is better adapted for our techniques}. We recall that a measurable function
$u$ is an entropy solution of (\ref{equaxx}) if, for all $t>0$, $\max
\{|u|,t\}\mbox
{sign}\{u\}$ belongs to $W_{0}^{1,2}(w,G)$\footnote{One could start with more
general $u^{\prime}s$ but it can be showed that if $f\in L^{1}(w,G),$ then an
entropy solution will automatically belong to $W_{0}^{1,2}(w,G).$ If $p>1,$
then one requires $p>2-1/n,$ in order to gurantee that entropy solutions
belong to $W_{0}^{1,p}(w,G).$}, and
\[
\int_{|u-\psi|<t}a(x,u,\nabla u)(\nabla u-\nabla\psi)dx\leq\int_{|u-\psi
|<t}fwdx,
\]
for every $\psi\in W_{0}^{1,2}(w,G)\cap L^{\infty}(G)$, where the weighted
Sobolev space $W_{0}^{1,2}(w,G)$ is the closure of $C_{0}^{\infty}(G)$ under
the norm
\[
\left\|  u\right\|  _{W_{0}^{1,2}(w,G)}^{2}=\int_{G}\left|  u(x)\right|
^{2}w(x)dx+\int_{G}\left|  \nabla u(x)\right|  ^{2}w(x)dx.
\]
It is known, for example, that if $f\in W^{-1,2}(w,G),$ the notion of entropy
solution coincides with the usual definition of weak solution (cf. \cite{AFT}).

The relation between, isoperimetry and the rearrangements of entropic
solutions is given by the following:

\begin{theorem}
\label{base}Let $u\in W_{0}^{1,1}(w,G)$ be a solution of (\ref{equaxx}). Let
$\mu=w(x)dx,$ and let $I=I_{(\mathbb{R}^{n};\mu)}$ be the isoperimetric
profile of $(\mathbb{R}^{n};\mu).$ Then, the following inequalities hold

\begin{enumerate}
\item
\begin{equation}
\left(  -u_{\mu}^{\ast}\right)  ^{\prime}(t)I(t)^{2}\leq\int_{0}^{t}f_{\mu
}^{\ast}(s)ds,\text{ }a.e. \label{func}%
\end{equation}

\item
\begin{equation}
\int_{t}^{\mu\left(  G\right)  }\left(  \left|  \nabla u\right|  ^{2}\right)
_{\mu}^{\ast}(s)ds\leq\int_{t}^{\mu\left(  G\right)  }\left(  \left(  -u_{\mu
}^{\ast}\right)  ^{\prime}(s)\int_{0}^{s}f_{\mu}^{\ast}(z)dz\right)  ds.
\label{derv}%
\end{equation}
\end{enumerate}
\end{theorem}

\begin{proof}
As in \cite[pag 712]{Ta} (or \cite{BBMP} when $w$ is the Gaussian density
function) we can suppose without loss of generality that $G=\mathbb{R}^{n},$
since any function from $W_{0}^{1,1}(w,G)$ is a function belonging to
$W_{0}^{1,1}(w,\mathbb{R}^{n})$ vanishing outside $G.$ Let $u$ be an (entropy)
solution of (\ref{equaxx}). Let $0<t<t+h<\infty.$ Consider the test function
given by \footnote{This is a standard procedure which has been used by many
authors see for example, \cite{Ta}, \cite{Ta1}, \cite{BBGBPV}, \cite{AFT} and
the references therein.}
\[
u_{t}^{t+h}(x)=\left\{
\begin{array}
[c]{ll}%
hsign(u)\text{ } & \text{if }\left|  u(x)\right|  >t+h,\\
\left(  \left|  u(x)\right|  -t\right)  sign(u) & \text{if }t<\left|
u(x)\right|  \leq t+h,\\
0 & \text{if }\left|  u(x)\right|  \leq t.
\end{array}
\right.
\]
Then, by the definition of entropic solution, we get
\begin{align}
J(t,h)  &  =\frac{1}{h}\int_{\left\{  t<\left|  u(x)\right|  \leq t+h\right\}
}\left|  \nabla u(x)\right|  ^{2}d\mu\nonumber\\
&  \leq\int_{\left\{  t<\left|  u(x)\right|  \leq t+h\right\}  }\left|
f(x)\right|  d\mu+\int_{\left\{  \left|  u(x)\right|  >t+h\right\}  }\left|
f(x)\right|  d\mu. \label{base1'}%
\end{align}
By H\"{o}lder's inequality,
\[
\left(  \frac{1}{h}\int_{\left\{  t<\left|  u(x)\right|  \leq t+h\right\}
}\left|  \nabla u(x)\right|  d\mu\right)  ^{2}\leq J(t,h)\left(  \frac{\mu
_{u}(t)-\mu_{u}(t+h)}{h}\right)  .
\]
Combining the last inequality (\ref{base1'}), and then letting $h\rightarrow
0,$ we find that
\[
\left(  -\frac{d}{dt}\int_{\left\{  \left|  u(x)\right|  >t\right\}  }\left|
\nabla u(x)\right|  d\mu\right)  ^{2}\leq-\frac{d\mu_{u}}{dt}(t)\int_{\left\{
\left|  u(x)\right|  >t\right\}  }\left|  f(x)\right|  d\mu.
\]
Replacing $t$ by $u_{\mu}^{\ast}(t)$ and using the chain rule and (\ref{dosa})
of Theorem \ref{teomain}, we obtain%
\[
\left(  \left.  \frac{d}{dt}\int_{\{\left|  u(x)\right|  >\cdot\}}\left|
\nabla u(x)\right|  d\mu\right|  _{u_{\mu}^{\ast}(t)}\right)  ^{2}\geq
(-u_{\mu}^{\ast})^{\prime}(t)\left[  I(t)\right]  ^{2}.
\]
On the other hand, as shown in \cite[pag 936, discussion in (iii)]{tal1},
\[
-\frac{d\mu_{u}}{dt}(u_{\mu}^{\ast}(t))\leq1,\text{ a.e.}%
\]
Therefore we arrive at
\[
(-u_{\mu}^{\ast})^{\prime}(t)\left[  I(t)\right]  ^{2}\leq\int_{0}^{t}f_{\mu
}^{\ast}(s)ds,
\]
as we wished to show.

Following \cite{AFT} we consider the function
\[
\Phi(t)=\int_{\left\{  \left|  u(x)\right|  \leq t\right\}  }\left|  \nabla
u(x)\right|  ^{2}d\mu,\text{ \ \ }t\in(0,\infty).
\]
It is plain that $\Phi$ is increasing, moreover, by a suitable change of
notation, (\ref{base1'}) yields that, for $0<t_{1}<t_{2},$%
\begin{align*}
\Phi(t_{1})-\Phi(t_{2})  &  =\int_{\left\{  t_{1}<\left|  u(x)\right|  \leq
t_{2}\right\}  }\left|  \nabla u(x)\right|  ^{2}d\mu\\
&  \leq\left(  t_{2}-t_{1}\right)  \left(  \int_{\left\{  t_{1}<\left|
u(x)\right|  \leq t_{2}\right\}  }\left|  f(x)\right|  d\mu+\int_{\left\{
\left|  u(x)\right|  >t_{2}\right\}  }\left|  f(x)\right|  d\mu\right) \\
&  \leq2\left(  t_{2}-t_{1}\right)  \left\|  f\right\|  _{L^{1}}.
\end{align*}
Consequently, $\Phi$ is a Lipschitz continuous function. Pick $t_{1}=u_{\mu
}^{\ast}(s+h)$ and $t_{2}=u_{\mu}^{\ast}(s)$, then, upon dividing both sides
of the previous inequality by $h,$ we find that
\begin{align*}
&  \frac{\Phi(u_{\mu}^{\ast}(s+h))-\Phi(u_{\mu}^{\ast}(s))}{h}\\
&  \leq\left(  \frac{u_{\mu}^{\ast}(s)-u_{\mu}^{\ast}(s+h)}{h}\right)  \left(
\int_{\left\{  u_{\mu}^{\ast}(s+h)<\left|  u(x)\right|  \leq u_{\mu}^{\ast
}(s)\right\}  }\left|  f(x)\right|  d\mu+\int_{\left\{  \left|  u(x)\right|
>u_{\mu}^{\ast}(s)\right\}  }\left|  f(x)\right|  d\mu\right)  .
\end{align*}
Letting $h\rightarrow0$ we obtain
\begin{equation}
-\frac{\partial}{\partial s}\left(  \Phi(u_{\mu}^{\ast}(s))\right)
\leq\left(  -u_{\mu}^{\ast}\right)  ^{\prime}(s)\int_{0}^{s}f_{\mu}^{\ast
}(r)dr. \label{basicder}%
\end{equation}
Integrating (\ref{basicder}) from $t$ to $\mu\left(  G\right)  $ we get
\[
\Phi(u_{\mu}^{\ast}(t))-\Phi(u_{\mu}^{\ast}(\mu\left(  G\right)  )\leq\int
_{t}^{\mu\left(  G\right)  }\left(  \left(  -u_{\mu}^{\ast}\right)
^{^{\prime}}(s)\int_{0}^{s}f_{\mu}^{\ast}(r)dr\right)  ds.
\]
Now, since $u=0$ on $\partial G,$ it follows that $u_{\mu}^{\ast}(\mu\left(
G\right)  )=0$ (cf. also \cite[(317)]{tal1})$.$ Thus $\Phi(u_{\mu}^{\ast}%
(\mu\left(  G\right)  )=0,$ and consequently the previous inequality becomes
\begin{equation}
\int_{\left\{  \left|  u\right|  \leq u_{\mu}^{\ast}(t)\right\}  }\left|
\nabla u(x)\right|  ^{2}d\mu\leq\int_{t}^{\mu\left(  G\right)  }\left(
\left(  -u_{\mu}^{\ast}\right)  ^{^{\prime}}(s)\int_{0}^{s}f_{\mu}^{\ast
}(r)dr\right)  ds. \label{vers1}%
\end{equation}
On the other hand, by the definition of decreasing rearrangement (see
\cite[Page 70]{KPS}), we have
\begin{align}
\int_{\left\{  \left|  u\right|  \leq u_{\mu}^{\ast}(t)\right\}  }\left|
\nabla u(x)\right|  ^{2}d\mu &  \geq\inf_{\mu(E)=\mu\left\{  \left|  u\right|
\leq u_{\mu}^{\ast}(t)\right\}  }\int_{E}\left|  \nabla u(x)\right|  ^{2}%
d\mu=\int_{\mu\left\{  \left|  u\right|  >u_{\mu}^{\ast}(t)\right\}  }%
^{\mu\left(  G\right)  }\left(  \left|  \nabla u\right|  ^{2}\right)  _{\mu
}^{\ast}(s)ds\nonumber\\
&  \geq\int_{t}^{\mu\left(  G\right)  }\left(  \left|  \nabla u\right|
^{2}\right)  _{\mu}^{\ast}(s)ds. \label{ver2}%
\end{align}
Combining (\ref{vers1}) and (\ref{ver2}) we obtain (\ref{derv}).
\end{proof}

We now make explicit the sharp \textit{a priori} integrability conditions for
solutions of (\ref{equaxx}) that are implied by our analysis. It is here that
the isoperimetric profile pays a crucial role in determining the correct
nature of the estimates: e.g. in the Gaussian case it automatically leads to
$L^{p}(LogL)^{q}$ integrability conditions, etc.

The analysis that follows is natural extension of the one given in Section
\ref{secc::po}. Consequently, there is a natural Hardy type operator
associated with the isoperimetric profile that we shall use to study the
integrability of solutions of (\ref{equaxx}), namely the operator $R_{I}$
(compare with the operator $Q_{I}$ defined by (\ref{olvidada}) above),
\[
R_{I}(h)(t)=\int_{t}^{\mu(G)}\left(  \frac{s}{I(s)}\right)  ^{2}h(s)\frac
{ds}{s}.
\]

\begin{theorem}
\label{opti00Elip}Let $X,Y$ be two r.i. spaces on $G$ such that $\overline
{\alpha}_{X}<1$ (cf. Remark \ref{alcance}), and
\begin{equation}
\left\|  R_{I}(h)\right\|  _{\bar{Y}}\preceq\left\|  h\right\|  _{\bar{X}}.
\label{integral}%
\end{equation}
Then, if $u$ is a solution of (\ref{equaxx}) with datum $f\in X(G),$ we have
\begin{equation}
\left\|  u_{\mu}^{\ast}\right\|  _{\bar{Y}}\preceq\left\|  f_{\mu}^{\ast
}\right\|  _{\bar{X}}. \label{one}%
\end{equation}
and
\begin{equation}
\left\|  u_{\mu}^{\ast}\right\|  _{\bar{Y}}\preceq\left\|  \left(  \frac
{I(t)}{t}\right)  ^{2}\left(  u_{\mu}^{\ast\ast}(t)-u_{\mu}^{\ast}(t)\right)
\right\|  _{\bar{X}}+\left\|  u_{\mu}^{\ast}\right\|  _{L^{1}}\preceq\left\|
f_{\mu}^{\ast}\right\|  _{\bar{X}}. \label{three}%
\end{equation}
Moreover, in the case that the operator $\tilde{R}_{I}(h)(t)=\left(
\frac{I(s)}{s}\right)  ^{2}\int_{t}^{\mu(G)}\left(  \frac{s}{I(s)}\right)
^{2}h(s)\frac{ds}{s}$ is bounded on $\bar{X},$ then if $u$ is the solution of
(\ref{equaxx}) with datum $f\in X(G)$, we have
\begin{equation}
\left\|  u_{\mu}^{\ast}\right\|  _{\bar{Y}}\preceq\left\|  \left(  \frac
{I(t)}{t}\right)  ^{2}u_{\mu}^{\ast}(t)\right\|  _{\bar{X}}\preceq\left\|
f_{\mu}^{\ast}\right\|  _{\bar{X}}. \label{two}%
\end{equation}
\end{theorem}

\begin{proof}
Using the fundamental theorem of calculus, the fact that $u_{\mu}^{\ast}%
(\mu(G))=0,$ and (\ref{func}), we get
\[
u_{\mu}^{\ast}(t)=\int_{t}^{\mu(G)}\left(  -u_{\mu}^{\ast}\right)  ^{\prime
}(s)ds\leq\int_{t}^{\mu(G)}\left(  \frac{s}{I(s)}\right)  ^{2}f_{\mu}%
^{\ast\ast}(s)\frac{ds}{s}=R_{I}(f_{\mu}^{\ast\ast})(t).
\]
Therefore (\ref{one}) follows from (\ref{integral}).

We shall now prove (\ref{three}). First we shall prove
\[
\left\|  u_{\mu}^{\ast}\right\|  _{\bar{Y}}\preceq\left\|  \left(  \frac
{I(t)}{t}\right)  ^{2}\left(  u_{\mu}^{\ast\ast}(t)-u_{\mu}^{\ast}(t)\right)
\right\|  _{\bar{X}}+\left\|  u_{\mu}^{\ast}\right\|  _{L^{1}}.
\]
By the fundamental theorem of calculus we have
\begin{align*}
u_{\mu}^{\ast\ast}(t)  &  \leq\int_{t}^{\mu(G)}\left(  \frac{s}{I(s)}\right)
^{2}\left\{  \left(  \frac{I(s)}{s}\right)  ^{2}\left(  u_{\mu}^{\ast\ast
}(s)-u_{\mu}^{\ast}(s)\right)  \right\}  \frac{ds}{s}+\left\|  u_{\mu}^{\ast
}\right\|  _{L^{1}}\\
&  =R_{I}(\{\cdot\cdot\})(t)+\left\|  u_{\mu}^{\ast}\right\|  _{L^{1}}.
\end{align*}
Therefore,
\begin{align*}
\left\|  u_{\mu}^{\ast}\right\|  _{\bar{Y}}  &  \leq\left\|  u_{\mu}^{\ast
\ast}\right\|  _{\bar{Y}}\\
&  \leq\left\|  R_{I}(\{\cdot\cdot\})\right\|  _{\bar{Y}}+\left\|  u_{\mu
}^{\ast}\right\|  _{L^{1}}\\
&  \preceq\left\|  \left(  \frac{I(s)}{s}\right)  ^{2}\left(  u_{\mu}%
^{\ast\ast}(s)-u_{\mu}^{\ast}(s)\right)  \right\|  _{\bar{X}}+\left\|  u_{\mu
}^{\ast}\right\|  _{L^{1}}.
\end{align*}

Now, we prove the remaining inequality of (\ref{three}). Suppose that $u$ is a
solution of (\ref{equaxx}). Then, since $u\in W_{0}^{1,1}(w;G),$ we get that
\begin{align*}
\left(  \frac{I(t)}{t}\right)  ^{2}\left(  u_{\mu}^{\ast\ast}(t)-u_{\mu}%
^{\ast}(t)\right)   &  =\left(  \frac{I(t)}{t}\right)  ^{2}\frac{1}{t}\int
_{0}^{t}s\left(  -u_{\mu}^{\ast}\right)  ^{\prime}(s)ds\\
&  \leq\frac{1}{t}\int_{0}^{t}I(s)^{2}\frac{1}{s}\left(  -u_{\mu}^{\ast
}\right)  ^{\prime}(s)ds\text{ (since }I(t)/t\text{ decreases)}\\
&  \leq\frac{1}{t}\int_{0}^{t}f_{\mu}^{\ast\ast}(s)ds\text{ \ (by
(\ref{func}))}.
\end{align*}
Therefore,
\[
\left\|  \left(  \frac{I(t)}{t}\right)  ^{2}\left(  u_{\mu}^{\ast\ast
}(t)-u_{\mu}^{\ast}(t)\right)  ^{2}\right\|  _{\bar{X}}\preceq\left\|  f_{\mu
}^{\ast}\right\|  _{\bar{X}}\text{ \ (since }\overline{\alpha}_{X}<1).
\]

Finally, to prove (\ref{two}) it will be convenient to define the r.i. space
on $(0,1)$,
\[
\bar{X}_{I^{2}}=\left\{  h:\left\|  h\right\|  _{\bar{X}_{I^{2}}}=\left\|
h(t)\left(  \frac{I(t)}{t}\right)  ^{2}\right\|  _{\bar{X}}<\infty\right\}  .
\]
Using the same argument given in the proof of Theorem \ref{opti00} part (a),
we can prove that
\[
\left\|  f\right\|  _{\bar{Y}}\preceq\left\|  f_{\mu}^{\ast}(t)\right\|
_{\bar{X}_{I^{2}}}.
\]
Now, we show that for all $f\in\bar{X},$
\[
\left\|  R_{I}(f)\right\|  _{\bar{X}_{I^{2}}}\preceq\left\|  f\right\|
_{\bar{X}}.
\]
Indeed, this is equivalent to the boundedness of the operator $\tilde{R}_{I}%
$:
\begin{align*}
\left\|  R_{I}(f)\right\|  _{\bar{X}_{I^{2}}}  &  =\left\|  \int_{t}^{\mu
(G)}\left(  \frac{s}{I(s)}\right)  ^{2}f(s)\frac{ds}{s}\right\|  _{\bar
{X}_{I^{2}}}\\
&  =\left\|  \left(  \frac{I(s)}{s}\right)  ^{2}\int_{t}^{\mu(G)}\left(
\frac{s}{I(s)}\right)  ^{2}f(s)\frac{ds}{s}\right\|  _{\bar{X}}\\
&  =\left\|  \tilde{R}_{I}f\right\|  _{\bar{X}}\\
&  \preceq\left\|  f\right\|  _{\bar{X}}.
\end{align*}
Consequently, by the first part of the theorem, we have that
\[
\left\|  \left(  \frac{I(t)}{t}\right)  ^{2}u_{\mu}^{\ast}(t)\right\|
_{\bar{X}}=\left\|  u_{\mu}^{\ast}\right\|  _{\bar{X}_{I^{2}}}\preceq\left\|
f_{\mu}^{\ast}\right\|  _{\bar{X}}.
\]
\end{proof}

In view of (\ref{two}), for a given datum $f\in X(G),$ $\bar{X}_{I}$ is the
``natural space'' to measure the regularity of the gradient, in fact we have

\begin{theorem}
\label{remderv} Let $u$ be any entropic solution of (\ref{equaxx}). Then,
\begin{equation}
\left|  \nabla u\right|  _{\mu}^{\ast}(t)\leq\left(  \frac{2}{t}\int
_{t/2}^{\mu(G)}\left(  \frac{I(s)}{s}f_{\mu}^{\ast\ast}(s)\right)
^{2}ds\right)  ^{1/2}. \label{revisada}%
\end{equation}
Furthermore, suppose that $f,$ the right hand side of (\ref{equaxx}), belongs
to a r.i. space $X(G),$ such that $1/2<\underline{\alpha}_{\bar{X}_{I}}$.
Then,
\begin{equation}
\left\|  \frac{I(t)}{t}\left|  \nabla u\right|  _{\mu}^{\ast}(t)\right\|
_{\bar{X}}\preceq\left\|  f_{\mu}^{\ast\ast}\right\|  _{\bar{X}}. \label{dd2}%
\end{equation}
\end{theorem}

\begin{proof}
Indeed, by (\ref{derv}), we know that
\begin{align*}
\int_{t/2}^{\mu(G)}\left(  \left|  \nabla u\right|  ^{2}\right)  _{\mu}^{\ast
}(s)ds  &  \leq\int_{t/2}^{\mu(G)}\left(  \left(  -u_{\mu}^{\ast}\right)
^{^{\prime}}(s)\int_{0}^{s}f_{\mu}^{\ast}(z)dz\right)  ds\\
&  \leq\int_{t/2}^{\mu(G)}\left(  \frac{s}{I(s)}f_{\mu}^{\ast\ast}(s)\right)
^{2}ds.
\end{align*}
Moreover,
\[
\int_{t/2}^{\mu(G)}\left(  \left|  \nabla u\right|  ^{2}\right)  _{\mu}^{\ast
}(s)ds\geq\int_{t/2}^{t}\left(  \left|  \nabla u\right|  ^{2}\right)  _{\mu
}^{\ast}(s)ds\geq\left(  \left|  \nabla u\right|  ^{2}\right)  _{\mu}^{\ast
}(t)\frac{t}{2}.
\]
Thus
\[
\left|  \nabla u\right|  _{\mu}^{\ast}(t)\leq\left(  \frac{2}{t}\int
_{t/2}^{\mu(G)}\left(  \frac{s}{I(s)}f_{\mu}^{\ast\ast}(s)\right)
^{2}ds\right)  ^{1/2}.
\]
Finally we prove (\ref{dd2}):
\begin{align*}
\left\|  \left|  \nabla u\right|  _{\mu}^{\ast}\right\|  _{\bar{X}_{I}}  &
=\left\|  \frac{I(t)}{t}\left|  \nabla u\right|  _{\mu}^{\ast}(t)\right\|
_{\bar{X}}\\
&  \leq\left\|  \frac{I(t)}{t}\left(  \frac{2}{t}\int_{t/2}^{\mu(G)}\left(
\frac{s}{I(s)}f_{\mu}^{\ast\ast}(s)\right)  ^{2}ds\right)  ^{1/2}\right\|
_{\bar{X}}\\
&  \leq\left\|  \frac{I(t/2)}{t/2}\left(  \frac{2}{t}\int_{t/2}^{\mu
(G)}\left(  \frac{s}{I(s)}f_{\mu}^{\ast\ast}(s)\right)  ^{2}ds\right)
^{1/2}\right\|  _{\bar{X}}\\
&  \leq2\left\|  \frac{I(t)}{t}\left(  \frac{1}{t}\int_{t}^{\mu(G)}\left(
\frac{s}{I(s)}f_{\mu}^{\ast\ast}(s)\right)  ^{2}ds\right)  ^{1/2}\right\|
_{\bar{X}}\text{ \ (by (\ref{ccdd}))}\\
&  =2\left\|  \left(  \frac{1}{t}\int_{t}^{\mu(G)}\left(  \frac{s}{I(s)}%
f_{\mu}^{\ast\ast}(s)\right)  ^{2}ds\right)  ^{1/2}\right\|  _{\bar{X}_{I}}\\
&  \preceq\left\|  \frac{s}{I(s)}f_{\mu}^{\ast\ast}(s)\right\|  _{\bar{X}_{I}%
}\text{ \ (by Lemma \ref{cota01}, since }1/2<\underline{\alpha}_{\bar{X}_{I}%
}\text{)}\\
&  =\left\|  f_{\mu}^{\ast\ast}\right\|  _{\bar{X}}\text{.}%
\end{align*}
\end{proof}

\begin{remark}
\label{remarkp}The results in this section can be easily adapted to the study
of ellipticity conditions of the type
\[
a(x,t,\xi).\xi\geq w(x)\left|  \xi\right|  ^{p},\text{ \ \ for a.e. }x\in
G\subset\mathbb{R}^{n},\text{ \ }\forall\eta\in\mathbb{R},\text{ \ }\forall
\xi\in\mathbb{R}^{n},
\]
where $1<p<\infty.$ In this case inequalities (\ref{func}) and (\ref{derv})
became respectively
\[
\left(  -u_{\mu}^{\ast}\right)  ^{\prime}(t)I(t)^{\frac{p}{p-1}}\leq\left(
\int_{0}^{t}f_{\mu}^{\ast}(s)ds\right)  ^{\frac{1}{p-1}},
\]%
\[
\int_{t}^{\mu\left(  G\right)  }\left(  \left|  \nabla u\right|  ^{p}\right)
_{\mu}^{\ast}(s)ds\leq\int_{t}^{1}\left(  \left(  -u_{\mu}^{\ast}\right)
^{\prime}(s)\int_{0}^{s}f_{\mu}^{\ast}(z)dz\right)  ds,
\]
and condition (\ref{integral}) needs to be replaced by
\[
\left\|  \int_{t}^{\mu(G)}\left(  \left(  \frac{s}{I(s)}\right)  ^{p}f_{\mu
}^{\ast\ast}(s)\right)  ^{\frac{1}{p-1}}\frac{ds}{s}\right\|  _{\bar{Y}%
}\preceq\left\|  f^{\ast}\right\|  _{\bar{X}}^{\frac{1}{p-1}}.
\]
\end{remark}

We omit the details and refer to \cite{mmelliptic} for more details.

\begin{remark}
To fix ideas in this paper we have only considered elliptic equations in
divergence form on domains of $\mathbb{R}^{n}.$ However, the proof of Theorem
\ref{base} can be easily adapted to the setting of $n-$dimensional Riemannian
manifolds$\;M$ with finite volume (say $vol(M)=1$) as considered by Cianchi in
\cite{Cia90}. Indeed, mutatis mutandi Theorem \ref{opti00Elip} can be easily
reformulated and is valid in this more general setting (cf. \cite{mmelliptic}).
\end{remark}

\subsection{Sharpness of the results\label{secc:sharp}}

We comment briefly on the sharpness of the results obtained in this section
and refer to \cite{mmelliptic} for a more detailed analysis. In the classical
papers of Talenti and his school (cf. \cite{Ta}, \cite{tal}, \cite{tal1},
\cite{Ta1} and the many references therein) the sharpness of the estimates is
obtained, roughly speaking, by comparing solutions of the Dirichlet problems
for suitable classes of elliptic equations in divergence form, with radial
solutions of the Laplace equation on a ball, whose measure is equal to the
measure of the original domain.

Under sufficient symmetry (for example in the case model cases discussed in
Section \ref{secc::p-s}, and in particular the abstract model of Section
\ref{secc:ros}), one can construct comparison equations and show the sharpness
of the results. We do not pursue this matter further in this long paper but it
is appropriate to mention that the natural extremal functions for comparison
in the model cases have rearrangements given by an explicit formula, namely
functions $v$ such that
\[
v_{\mu}^{\ast}(t)=\int_{t}^{\mu\left(  G\right)  }\left(  \frac{s}%
{I(s)}\right)  ^{2}f_{\mu}^{\ast\ast}(s)\frac{ds}{s}.
\]
In fact note that, by Theorem \ref{opti00Elip}, any entropic solution $u$ of
(\ref{equaxx}) must satisfy
\[
u_{\mu}^{\ast}(t)\preceq v_{\mu}^{\ast}(t).
\]
This is the pointwise domination is captured in the papers mentioned earlier.
Moreover, a suitable oscillation of $u$ is also controlled by the oscillation
of $v!.$ Indeed, the oscillation under control is none other than $u_{\mu
}^{\ast\ast}(t)-u_{\mu}^{\ast}(t):$%

\begin{align*}
u_{\mu}^{\ast\ast}(t)-u_{\mu}^{\ast}(t)  &  =\frac{1}{t}\int_{0}^{t}s\left(
-u_{\mu}^{\ast}\right)  ^{\prime}(s)ds\\
&  \leq\frac{1}{t}\int_{0}^{t}\left(  \frac{s}{I(s)}\right)  ^{2}f_{\mu}%
^{\ast\ast}(s)ds\text{ \ (by (\ref{func}))}\\
&  =v_{\mu}^{\ast\ast}(t)-v_{\mu}^{\ast}(t)\text{.}%
\end{align*}
Furthermore, the analysis of the proof of Theorem \ref{opti00Elip} shows that,
if $\tilde{R}_{I}$ is bounded on $\bar{X}$,
\[
\left\|  \left(  \frac{I(t)}{t}\right)  ^{2}v_{\mu}^{\ast}(t)\right\|
_{\bar{X}}\simeq\left\|  f_{\mu}^{\ast}\right\|  _{\bar{X}}.
\]
Therefore, if $\overline{\alpha}_{X}<1,$%
\[
\left\|  \left(  \frac{I(t)}{t}\right)  ^{2}\left(  v_{\mu}^{\ast\ast
}(t)-v_{\mu}^{\ast}(t)\right)  \right\|  _{\bar{X}}+\left\|  v_{\mu}^{\ast
}\right\|  _{L^{1}}\simeq\left\|  f_{\mu}^{\ast}\right\|  _{\bar{X}}.
\]

\subsubsection{Between exponential and Gaussian measure}

Let us consider the following set of elliptic problems associated with
Gaussian measures and explain how they fit our models. Let $\alpha\geq0,$
\ $p\in\lbrack1,2]$ and $\gamma=\exp(2\alpha/(2-p)),$ and let
\[
\mu_{p,\alpha}=Z_{p,\alpha}^{-1}\exp\left(  -\left|  x\right|  ^{p}%
(\log(\gamma+\left|  x\right|  )^{\alpha}\right)  dx=\varphi_{\alpha
,p}(x)dx,\text{ \ \ \ }x\in\mathbb{R},
\]
and
\[
\varphi_{\alpha,p}^{n}(x)=\varphi_{\alpha,p}(x_{1})\cdots\varphi_{\alpha
,p}(x_{n}),\text{ \ and }\mu=\mu_{p,\alpha}^{\otimes n}.
\]
Consider
\begin{equation}
\left\{
\begin{array}
[c]{ll}%
-div(a(x,u,\nabla u))=f\varphi_{\alpha,p}^{n} & \text{ in }G,\\
u=0 & \text{ on }\partial G,
\end{array}
\right.  \label{equabetween}%
\end{equation}
with the ellipticity condition,
\[
a(x,t,\xi).\xi\succeq\varphi_{\alpha,p}^{n}(x)\left|  \xi\right|  ^{2},\text{
\ \ for a.e. \ \ }x\in G,\text{ \ }\forall\eta\in\mathbb{R},\text{ \ }%
\forall\xi\in\mathbb{R}^{n},
\]
where $G\subset$ $\mathbb{R}^{n}$ is an open set such that $\mu(G)<1.$

Theorems \ref{opti00Elip} and \ref{remderv}, yield: Let $u$ be a solution of
(\ref{equabetween}) with datum $f\in X(G).$ Assume that $\overline{\alpha
}_{\bar{X}}<1.$ Then,

\begin{enumerate}
\item  If $0<\underline{\alpha}_{\bar{X}},$%
\begin{equation}
\left\|  \left(  \log\frac{1}{s}\right)  ^{2\left(  1-\frac{1}{p}\right)
}\left(  \log\log\left(  e+\frac{1}{s}\right)  \right)  ^{2\frac{\alpha}{p}%
}u_{\mu}^{\ast}(s)\right\|  _{\bar{X}}\preceq\left\|  f\right\|  _{X}.
\label{eli001}%
\end{equation}

\item  If $0=\underline{\alpha}_{\bar{X}},$%
\[
\left\|  \left(  \log\frac{1}{s}\right)  ^{2\left(  1-\frac{1}{p}\right)
}\left(  \log\log\left(  e+\frac{1}{s}\right)  \right)  ^{2\frac{\alpha}{p}%
}(u_{\mu}^{\ast\ast}(s)-u_{\mu}^{\ast}(s))\right\|  _{\bar{X}}+\left\|
u\right\|  _{L^{1}}\preceq\left\|  f\right\|  _{X}.
\]

\item  If $\underline{\alpha}_{\bar{X}}>1/2,$%
\begin{equation}
\left\|  \left(  \log\frac{1}{s}\right)  ^{\left(  1-\frac{1}{p}\right)
}\left(  \log\log\left(  e+\frac{1}{s}\right)  \right)  ^{\frac{\alpha}{p}%
}\left|  \nabla u\right|  _{\mu}^{\ast}(s)\right\|  _{\bar{X}}\preceq\left\|
f\right\|  _{X}. \label{eli002}%
\end{equation}
\end{enumerate}

Indeed, \ since $\mu(G)<1,$ it follows from (\ref{asim}) that
\[
I_{\mu_{p,\alpha}^{\otimes n}}(s)\simeq s\left(  \log\frac{1}{s}\right)
^{1-\frac{1}{p}}\left(  \log\log\left(  e+\frac{1}{s}\right)  \right)
^{\frac{\alpha}{p}},\text{ \ \ }0<s<\mu(G).
\]
Therefore,
\[
R_{I}h(s)\simeq\int_{t}^{\mu\left(  G\right)  }\left(  \frac{1}{\left(
\log\frac{1}{s}\right)  ^{1-\frac{1}{p}}\left(  \log\log\left(  e+\frac{1}%
{s}\right)  \right)  ^{\frac{\alpha}{p}}}\right)  ^{2}f_{\mu}^{\ast\ast
}(s)\frac{ds}{s}.
\]
The method given in example \ref{exex} can be easily adapted to see that
$\tilde{R}_{I}\ $is bounded on $\bar{X}$, if $0<\underline{\alpha}_{\bar{X}}$
and $\overline{\alpha}_{\bar{X}}<1.$ Statement $(2)$ follows similarly.
Finally to see $(3)$, notice that
\[
\left\|  f\right\|  _{\bar{X}_{I_{\mu}}}\simeq\left\|  \left(  \log\frac{1}%
{s}\right)  ^{\left(  1-\frac{1}{p}\right)  }\left(  \log\log\left(
e+\frac{1}{s}\right)  \right)  ^{\frac{\alpha}{p}}f(s)\right\|  _{\bar{X}}%
\]
and an easy computation shows that $\underline{\alpha}_{\bar{X}}%
=\underline{\alpha}_{\bar{X}_{I_{\mu}}}$, hence, Theorem \ref{remderv} applies.

In this context (see section \ref{secc:logconcave}) there is a suitable
rearrangement $f^{\circ}:\mathbb{R}^{n}\rightarrow\mathbb{R}$ defined by
\[
f^{\circ}(x)=f^{\ast}(H(x_{1})).
\]
where $H:\mathbb{R}\rightarrow(0,1)$ is given by
\[
H(r)=\int_{-\infty}^{r}\varphi_{\alpha,p}(x)dx.
\]
Therefore one is led to compare (\ref{equabetween}) with
\begin{equation}
\left\{
\begin{array}
[c]{ll}%
-\left(  \varphi_{\alpha,p}^{n}v_{x_{1}}\right)  _{x_{1}}=f^{\circ}%
\varphi_{\alpha,p}^{n} & \text{ in }G^{\bigstar},\\
v=0 & \text{ on }\partial G^{\bigstar},
\end{array}
\right.  \label{equaequa}%
\end{equation}
where $G^{\bigstar}$ is the half space defined by
\[
G^{\bigstar}=\{x=(x_{1},.....x_{n}):x_{1}<r\},
\]
and $r\in\mathbb{R}$ is selected so that $H(r)=\mu(G).$ The solution of
(\ref{equaequa}) is given by inspection:
\[
v(x_{1})=\int_{x_{1}}^{r}\left(  Z_{p,\alpha}^{-1}\exp\left(  \left|
t\right|  ^{p}(\log(\gamma+\left|  t\right|  )^{\alpha}\right)  \int_{-\infty
}^{t}f^{\circ}(s)\varphi_{\alpha,p}(s)ds\right)  dt,\text{ \ \ \ \ }x_{1}\in
G^{\bigstar}.
\]
Note that since
\begin{align*}
v^{\circ}(x)  &  =\int_{H(x_{1})}^{r}Z_{p,\alpha}^{-1}\exp\left(  \left|
t\right|  ^{p}(\log(\gamma+\left|  t\right|  )^{\alpha}\right)  \int_{-\infty
}^{t}f^{\circ}(s)\varphi_{\alpha,p}(s)dsdt\\
&  =\int_{x_{1}}^{\mu(G)}Z_{p,\alpha}^{-1}\exp\left(  \left|  H^{-1}%
(t)\right|  ^{p}(\log(\gamma+\left|  H^{-1}(t)\right|  )^{\alpha}\right)
\int_{-\infty}^{H^{-1}(s)}f^{\circ}(s)\varphi_{\alpha,p}(s)ds\frac{\partial
H^{-1}}{\partial t}(t)dt\\
&  =\int_{x_{1}}^{\mu(G)}\left(  \frac{s}{I_{\mu_{p,\alpha}}(s)}\right)
^{2}\frac{1}{s}\int_{0}^{s}f_{\mu}^{\ast}(z)dzds,
\end{align*}
and
\[
v_{\mu}^{\ast}=(v^{\circ})_{\mu}^{\ast},
\]
we have
\[
v_{\mu}^{\ast}(t)\simeq\int_{t}^{\mu(G)}\left(  \frac{s}{I_{\mu_{p,\alpha}%
}(s)}\right)  ^{2}f_{\mu}^{\ast\ast}(s)ds.
\]

\begin{remark}
\label{exrn} Suppose that the datum $f$ belongs to the Lorentz-Zygmund space
$L^{q,m}(\log L)^{\lambda}L^{q,m}(\log L)^{\lambda},(1<q<\infty,m$ $\geq1$
$\lambda\in\mathbb{R)}$ and let $u$ be a solution of (\ref{equabetween}).
Then, from (\ref{eli001}) and the fact that (see \cite{BS})
\[
\underline{\alpha}_{L^{q,m}(\log L)^{\lambda}}=\bar{\alpha}_{L^{q,m}(\log
L)^{\lambda}}=\frac{1}{q},
\]
we get
\begin{align*}
&  \left(  \int_{0}^{\mu(G)}\left(  s^{\frac{1}{q}}\left(  1+\log\frac{1}%
{s}\right)  ^{2\left(  1-\frac{1}{p}\right)  +\lambda}\left(  \log\log\left(
e+\frac{1}{s}\right)  \right)  ^{2\frac{\alpha}{p}}u_{\mu}^{\ast}(s)\right)
^{m}\frac{ds}{s}\right)  ^{\frac{1}{m}}\\
&  \preceq\left\|  f\right\|  _{L^{q,m}(\log L)^{\lambda}}.
\end{align*}
Moreover, if $2<q,$ then by (\ref{eli002}),
\begin{align*}
&  \left(  \int_{0}^{\mu(G)}\left(  s^{\frac{1}{q}}\left(  1+\log\frac{1}%
{s}\right)  ^{\left(  1-\frac{1}{p}\right)  +\lambda}\left(  \log\log\left(
e+\frac{1}{s}\right)  \right)  ^{\frac{\alpha}{p}}\left|  \nabla u\right|
_{\mu}^{\ast}(s)\right)  ^{m}\frac{ds}{s}\right)  ^{\frac{1}{m}}\\
&  \preceq\left\|  f\right\|  _{L^{q,m}(\log L)^{\lambda}}.
\end{align*}
In the particular case $p=2$ and $\alpha=0$ (i.e the Gaussian case) a priori
estimates for elliptic equations (\ref{equabetween}) with datum in
Lorentz-Zygmund spaces $L^{q,m}(\log L)^{\lambda}$ have been considered by
several authors, see for example \cite{BBMP}, \cite{Bl}, \cite{Feo},
\cite{Feo1}. Our results are sharp (cf. \cite[Theorem 5.1]{Feo}).
\end{remark}

\section{Connection with some capacitary inequalities due to
Maz'ya\label{secc:ma}}

We comment briefly, and somewhat informally, on a connection between what we
have termed the Maz'ya-Talenti inequality (\ref{dosa}) and some of Maz'ya's
capacitary inequalities (cf. \cite{maz'yalectures}, \cite{maz'yacap}). Indeed,
we show explicitly how to derive symmetrization inequalities of the type
discussed in this paper, from Maz'ya's capacitary inequalities.

Recall that (\ref{dosa}) was originally formulated on $\mathbb{R}^{n}$ (cf.
\cite{tal} and the references therein) with Lebesgue measure, where of course
$I(t)=c_{n}t^{1-1/n}$, and we shall restrict ourselves to this
setting\footnote{We note that one interesting aspect of the method of
capacitary inequalities is that it can be implemented in very general
settings. On the other hand we have to postpone a general discussion for
another occasion.}. Moreover, although this is an important point, and the
constants can be made quite explicit, we shall not keep track of the absolute
constants in this discussion. We must also refer to \cite{maz'yalectures},
\cite{maz'yacap} for background and notation. In what follows we let $G$ be an
open set in $\mathbb{R}^{n},$ $\left|  \cdot\right|  =$Lebesgue measure. Then,
for a compact set $F\subset G$, Maz'ya \cite[cf. (8.7)]{maz'yalectures} shows
that, for $1\leq p<n,$
\begin{equation}
cap_{p}(F,G)\succeq\left|  \left|  G\right|  ^{\frac{p-n}{n(p-1)}}-\left|
F\right|  ^{\frac{p-n}{n(p-1)}}\right|  ^{1-p},\text{ }p<n,\label{Mazya1}%
\end{equation}
while for $p=n$ we have%
\begin{equation}
cap_{n}(F,G)\succeq\left(  \log\left|  G\right|  -\log\left|  F\right|
\right)  ^{1-n}.\label{Mazya2}%
\end{equation}

To develop the connection we shall compute capacities normalizing the smooth
truncations as follows. Let $0<t_{1}<t_{2}<\infty,$ $f\in C_{0}^{\infty}(G),$
then we define
\[
N[f_{t_{1}}^{t_{2}}(x)]=\frac{f_{t_{1}}^{t_{2}}(x)}{t_{2}-t_{1}}=\left\{
\begin{array}
[c]{ll}%
1 & \text{if }\left|  f(x)\right|  >t_{2},\\
\leq1 & \text{if }t_{1}<\left|  f(x)\right|  \leq t_{2},\\
0 & \text{if }\left|  f(x)\right|  \leq t_{1}.
\end{array}
\right.  .
\]
Therefore, by definition we can estimate%
\[
cap_{p}\left(  \overline{\{\left|  f(x)\right|  >t_{2}\}},\text{ }\{\left|
f(x)\right|  >t_{1}\}\right)  \leq\frac{1}{(t_{2}-t_{1})^{p}}\int
_{\{t_{1}<\left|  f\right|  <t_{2}\}}\left|  \nabla f(x)\right|  ^{p}dx.
\]
Let $t_{1}=f^{\ast}(t),$ $t_{2}=f^{\ast}(t+h),$ $h>0.$ Then, we have%
\begin{align*}
&  cap_{p}\left(  \left\{  \left|  f(x)\right|  \geq f^{\ast}(t)\right\}
,\text{ }\left\{  \left|  f(x)\right|  \geq f^{\ast}(t+h)\right\}  \right)
[f^{\ast}(t+h)-f^{\ast}(h)]^{p}\\
&  \leq\int_{\{f^{\ast}(t+h)<\left|  f\right|  <f^{\ast}(t)\}}\left|  \nabla
f(x)\right|  ^{p}dx.
\end{align*}
Combining with (\ref{Mazya1}) we obtain,%
\[
cap_{p}\left(  \{\left|  f(x)\right|  \geq f^{\ast}(t)\},\text{ }\left\{
\left|  f(x)\right|  \geq f^{\ast}(t+h)\right\}  \right)  \succeq\left|
\left|  t+h\right|  ^{\frac{p-n}{n(p-1)}}-\left|  t\right|  ^{\frac
{p-n}{n(p-1)}}\right|  ^{1-p},
\]
and therefore.%
\[
\lbrack f^{\ast}(t+h)-f^{\ast}(h)]^{p}\left|  \left|  t+h\right|  ^{\frac
{p-n}{n(p-1)}}-\left|  t\right|  ^{\frac{p-n}{n(p-1)}}\right|  ^{1-p}%
\preceq\int_{\{f^{\ast}(t+h)<\left|  f\right|  <f^{\ast}(t)\}}\left|  \nabla
f(x)\right|  ^{p}dx,
\]
and%
\[
\left(  \frac{f^{\ast}(t+h)-f^{\ast}(h)}{h}\right)  ^{p}\left|  \frac{\left|
t+h\right|  ^{\frac{p-n}{n(p-1)}}-\left|  t\right|  ^{\frac{p-n}{n(p-1)}}}%
{h}\right|  ^{1-p}\preceq\frac{1}{h}\int_{\{f^{\ast}(t+h)<\left|  f\right|
<f^{\ast}(t)\}}\left|  \nabla f(x)\right|  ^{p}dx.
\]
Now we let $h\rightarrow0,$ to find%
\[
\left(  \frac{(p-n)}{n(p-1)}\right)  ^{1-p}[\left(  -f^{\ast}\right)
^{\prime}(t)]^{p}\left(  t^{\frac{p-n}{n(p-1)}-1}\right)  ^{1-p}\preceq
\frac{d}{dt}\int_{\{\left|  f\right|  >f^{\ast}(t)\}}\left|  \nabla
f(x)\right|  ^{p}dx.
\]
In particular, for $p=1$ we actually get%
\[
s^{1-1/n}\left(  -f^{\ast}\right)  ^{\prime}(s)\preceq\frac{\partial}{\partial
s}\int_{\left\{  \left|  f\right|  >f^{\ast}(s)\right\}  }\left|  \nabla
f(x)\right|  dx.
\]
Moreover, for $p=n$ the same argument, but using (\ref{Mazya2}) instead,
yields%
\[
\left(  \frac{f^{\ast}(t+h)-f^{\ast}(h)}{h}\right)  ^{n}\left|  \frac
{\log\left|  t+h\right|  -\log\left|  t\right|  }{h}\right|  ^{1-n}%
\preceq\frac{1}{h}\int_{\{f^{\ast}(t+h)<\left|  f\right|  <f^{\ast}%
(t)\}}\left|  \nabla f(x)\right|  ^{n}dx,
\]
so that%
\[
s^{n-1}\left(  \left(  -f^{\ast}\right)  ^{\prime}(s)\right)  ^{n}\preceq
\frac{\partial}{\partial s}\int_{\left\{  \left|  f\right|  >f^{\ast
}(s)\right\}  }\left|  \nabla f(x)\right|  ^{n}dx.
\]

The previous argument can easily be made rigorous and extended to the more
general setting of Section \ref{secc:trunc}.

\section{Appendix: A few (and only a few) bibliographical
notes\label{secc:appendix}}

It has not been out intention to provide a comprehensive bibliography. Indeed,
the topics discussed in this paper have been intensively studied for a long
time, with a variety of different approaches, and even though the bibliography
we have collected is rather large it is by definition very incomplete and many
times during the text we had to refer the reader to papers quoted within the
quoted papers and books... Therefore, we must apologize in advance for
oversights. With this important proviso we make a few (and only a few)
bibliographical notes and add a few more references that were not mentioned in
the main text. Moreover, we take the opportunity to very briefly comment on
some results and correct some of our previous bibliographical oversights in
earlier publications for which we must apologize yet again.

As was pointed in out in \cite{bmr}, the inequality (\ref{int4}), which in the
Euclidean case takes the form%
\begin{equation}
f^{\ast\ast}(t)-f^{\ast}(t)\leq c_{n}t^{1/n}\left|  \nabla f\right|
^{\ast\ast}(t), \label{nula}%
\end{equation}
is implicit in \cite[Appendix]{alvino}. However, it was not used in this form
in \cite{alvino}, but rather as%
\[
f^{\ast\ast}(t)\leq c_{n}t^{1/n}\left|  \nabla f\right|  ^{\ast\ast
}(t)+f^{\ast}(t),
\]
followed by the triangle inequality. This step however destroys the effect of
the cancellation afforded by (\ref{nula}). In \cite{kol} one can find a
similar inequality but with the left hand side $f^{\ast\ast}(t)-f^{\ast}(t)$
replaced by $f^{\ast}(t)-f^{\ast}(2t).$ This leads to equivalent type of
inequalities as it was shown, much later, in \cite{bmr} and \cite{perez}.
Neither of these papers uses isoperimetry explicitly and the proofs are
ad-hoc. For yet another approach using maximal operators see \cite{kami} (and
the references therein!).

Oscillation inequalities have a long history, for example they appear very
prominently in the work of Garsia-Rodemich \cite{gar}. A discrete version of
Talenti's inequality was also recorded in \cite[Proposition 4]{tar}.

The role of the oscillation spaces as limiting spaces seems to have originated
with the work of Bennett, De Vore and Sharpley \cite{bds}. At any rate
$f^{\ast\ast}(t)-f^{\ast}(t)$ has interesting interpretations in interpolation
theory (cf. \cite{bds}, \cite{ss} and for still a different interpretation see
\cite{jm2} and \cite{mmcw}). The role of oscillation spaces in the limiting
cases of the Sobolev embedding theorem seems to have been noticed first by
Tartar \cite{tar}. Using the notation of \cite{mapi} it follows from
\cite[Proposition 4]{tar} that $W^{1,n}(\Omega)\subset H_{n}(\Omega).$ This
result was also pointed out later in \cite{mapi}. At the time we wrote
\cite{mmjfa} we were also unaware of the results in \cite{ga}, we hope to have
rectified this oversight with the discussion presented in Section
\ref{secc:ga}.

Sobolev embeddings have a long history (for different perspectives cf.
\cite{maz'yabook}, \cite{adams}, \cite{edmund}, just to name a few). The first
complete treatment of embeddings of Sobolev spaces in the setting of
rearrangement invariant spaces with necessary and sufficient conditions that
we know is \cite{cwp}, and later extended in \cite[in particular see the
comments at the bottom of page 310]{edmkerpic}. A good deal of this work on
r.i. spaces been inspired by the classical work of Moser-Trudinger and O'Neil
(cf. \cite{on}, \cite{brezis}, \cite{han} and the references therein).

We conclude mentioning that in this paper we have not considered compactness
of embeddings. However, we believe that the methods of \cite{pu1} and
\cite{mmpjfa} can be generalized to the setting of this paper, and we hope to
return to the matter elsewhere.

\end{document}